\documentclass[a4paper,12pt]{article} 
\title{Pixton's formula and Abel-Jacobi theory \\on the Picard stack}
\usepackage{amsmath, amssymb, mathrsfs, amsthm, shorttoc, geometry, stmaryrd}
\usepackage{mathtools} 
\usepackage{hyperref}
\usepackage{xcolor}
\hypersetup{
colorlinks,
linkcolor={black},
citecolor={blue!50!black},
urlcolor={blue!80!black}
}

\newcommand{\MWvone}[1]{}
\newcommand{\MWref}[1]{#1}

\usepackage[all]{xy}

\usepackage{tabularx}
\usepackage{longtable}

\usepackage{enumitem}

\usepackage[capitalise]{cleveref}
\crefformat{equation}{(#2#1#3)}

\usepackage{pgf,tikz}
\usetikzlibrary{matrix, calc, arrows}

\usepackage{tikz-cd} 

\makeatletter
\newcommand*{\doublerightarrow}[2]{\mathrel{
  \settowidth{\@tempdima}{$\scriptstyle#1$}
  \settowidth{\@tempdimb}{$\scriptstyle#2$}
  \ifdim\@tempdimb>\@tempdima \@tempdima=\@tempdimb\fi
  \mathop{\vcenter{
    \offinterlineskip\ialign{\hbox to\dimexpr\@tempdima+1em{##}\cr
    \rightarrowfill\cr\noalign{\kern.5ex}
    \rightarrowfill\cr}}}\limits^{\!#1}_{\!#2}}}
\newcommand*{\triplerightarrow}[1]{\mathrel{
  \settowidth{\@tempdima}{$\scriptstyle#1$}
  \mathop{\vcenter{
    \offinterlineskip\ialign{\hbox to\dimexpr\@tempdima+1em{##}\cr
    \rightarrowfill\cr\noalign{\kern.5ex}
    \rightarrowfill\cr\noalign{\kern.5ex}
    \rightarrowfill\cr}}}\limits^{\!#1}}}
\makeatother

\newcommand{\on}[1]{\operatorname{#1}}
\newcommand{\bb}[1]{{\mathbb{#1}}}
\newcommand{\cl}[1]{{\mathscr{#1}}}
\newcommand{\ca}[1]{{\mathcal{#1}}}
\newcommand{\bd}[1]{{\mathbf{#1}}}

\newcommand{\ul}[1]{{\underline{#1}}}

\def\D{\mathrm{D}}

\def\log{\mathrm{log}}

\def\virt{\mathrm{vir}}
\def\P{\mathsf{P}}
\def\PP{\mathbb{P}}
\def\DR{\mathsf{DR}}

\def\CP{{{\mathbb {CP}}}}

\def\cO{\mathcal{O}}
\def\oM{\overline{\mathcal{M}}}
\def\cM{{\mathcal{M}}}
\def\C{{\mathcal{C}}}

\def\Z{\mathbb{Z}}

\def\C{\mathbb{C}}

\def\qed{{\hfill $\Diamond$}}

\def\Aut{{\rm Aut}}

\def\E{\mathrm{E}}
\def\n{\mathrm{n}}
\def\L{\mathrm{L}}
\def\V{\mathrm{V}}
\def\H{\mathrm{H}}
\def\HH{\mathcal{H}}
\def\g{\mathrm{g}}
\def\G{\mathsf{G}}
\def\rarr{\rightarrow}
\def\D{\mathsf{D}}

\def\nice{\displaystyle}

\DeclareMathAlphabet{\mymathbb}{U}{BOONDOX-ds}{m}{n}

\def\sheafhom{\mathcal{H}om}

\newcommand{\hra}{\hookrightarrow}
\newcommand{\sub}{\subseteq}

\theoremstyle{definition}
\newtheorem{definition}{Definition} 

\newtheorem{remark}[definition]{Remark}

\theoremstyle{plain}

\newtheorem*{conjectureA}{Conjecture A}
\newtheorem{proposition}[definition]{Proposition}
\newtheorem{lemma}[definition]{Lemma}
\newtheorem{theorem}[definition]{Theorem}
\newtheorem*{theorem*}{Theorem}
\newtheorem{corollary}[definition]{Corollary}

\newcommand{\MbarX}{\Mbar_{g,n}(X, \beta)}

\usepackage{letltxmacro}
\LetLtxMacro{\phiorig}{\phi}
\renewcommand{\phi}{\varphi}
\newcommand{\aj}{{\scriptstyle \mathsf{AJ}}}
\newcommand{\vdim}{\mathrm{vdim}}

\newcommand{\Picabs}{\mathfrak{Pic}}
\newcommand{\Picrel}{\mathfrak{Pic}^{\mathrm{rel}}}
\newcommand{\Chow}{\mathsf{CH}}
\newcommand{\CHop}{\Chow_{\mathsf{op}}}

\newcommand{\invnrI}{I}
\newcommand{\invnrn}{II}
\newcommand{\invnrA}{III}
\newcommand{\invnrII}{IV}
\newcommand{\invnrIII}{V}
\newcommand{\invnrIV}{VI}

 \newcommand{\jocomment}[1]{}
 \newcommand{\Dcomment}[1]{}
\author{Y. Bae, D. Holmes, R. Pandharipande, J. Schmitt, R. Schwarz}
\date{May 2021}

\newcounter{nootje}
\newcommand{\beq}{\begin{equation}}
\newcommand{\eeq}{\end{equation}}
\newcommand{\beqs}{\begin{equation*}}
\newcommand{\eeqs}{\end{equation*}}

\renewcommand{\k}{k}

\newcommand{\DRop}{\mathsf{DR}^{\mathsf{op}}}

\tikzset{
  symbol/.style={
    draw=none,
    every to/.append style={
      edge node={node [sloped, allow upside down, auto=false]{$#1$}}}
  }
}
\tikzset{
    labl/.style={anchor=south, rotate=-90, inner sep=.5mm}
}

\begin{document}
\maketitle
\begin{abstract} 
Let $A=(a_1,\ldots,a_n)$ be a vector of integers with $d=\sum_{i=1}^n a_i$.
By partial resolution of the classical Abel-Jacobi map, we construct a universal
twisted double ramification cycle $\mathsf{DR}^{\mathsf{op}}_{g,A}$ as an operational Chow class
on the Picard stack $\Picabs_{g,n,d}$
of $n$-pointed genus $g$ curves carrying a degree $d$ line bundle. The method of construction
follows the log (and b-Chow) approach to the standard
double ramification cycle with canonical twists on the moduli space of curves \cite{Holmes2017Extending-the-d, Holmes2017Multiplicativit, Marcus2017Logarithmic-com}.

Our main result is a calculation of $\mathsf{DR}^{\mathsf{op}}_{g,A}$ on 
the Picard stack $\Picabs_{g,n,d}$
via an appropriate interpretation of Pixton's formula in the tautological ring. The basic
new tool used in the
proof is the theory of double ramification cycles for target varieties \cite{Janda2018Double-ramifica}.
The formula on the Picard stack is obtained from \cite{Janda2018Double-ramifica} for
target varieties $\mathbb{CP}^n$ in
the limit $n\rightarrow \infty$. The result may be viewed as a
universal calculation in Abel-Jacobi theory.

As a consequence of the calculation of  $\mathsf{DR}^{\mathsf{op}}_{g,A}$ on the Picard stack
$\Picabs_{g,n,d}$, we prove that
the fundamental classes of the moduli spaces of twisted meromorphic differentials
in $\overline{\mathcal{M}}_{g,n}$
are exactly given by Pixton's formula (as conjectured in \cite[Appendix]{Farkas2016The-moduli-spac} and \cite{Schmitt2016Dimension-theor}).
The comparison result  of fundamental classes proven in \cite{Holmes2019Infinitesimal-s} plays a crucial
role in our argument. 
We also prove the set of relations in the
tautological ring of the Picard stack $\Picabs_{g,n,d}$  associated
to Pixton's formula.

\end{abstract}


\tableofcontents

\newcommand{\Mtildes}{ \widetilde{\ca M}^\Sigma}
\newcommand{\sch}[1]{\textcolor{blue}{#1}}

\newcommand{\Mbar}{\overline{\ca M}}
\newcommand{\MD}{\ca M^\blacklozenge}
\newcommand{\Md}{\ca M^\lozenge}
\newcommand{\DRL}{\operatorname{DRL}}
\newcommand{\DRC}{\operatorname{DRC}}
\newcommand{\isom}{\stackrel{\sim}{\longrightarrow}}
\newcommand{\Ann}[1]{\on{Ann}(#1)}
\newcommand{\fm}{\mathfrak m}
\newcommand{\Mdk}{\Mbar^{\m, 1/\k}}
\newcommand{\field}{K}
\newcommand{\Mdm}{\Mbar^\m}
\newcommand{\m}{{\bd m}}
\newcommand{\cat}[1]{\mathbf{#1}}
\newcommand{\targetmap}{\ell}
\newcommand{\sourcemap}{\ell'}
\newcommand{\rel}{\mathsf{rel}}
\newcommand{\pre}{\mathsf{pre}}
\setcounter{section}{-1}
\section{Introduction}

\parskip=5pt

\subsection{Double ramification cycles}
\label{Ssec:DRclassic}

Let $A = (a_1, \dots, a_n)$ be a vector of $n$ integers satisfying
$$\sum_{i=1}^n a_i = 0\,. $$ In the moduli space $\cM_{g,n}$ of 
nonsingular curves of genus $g$ with $n$ marked points,
 consider the substack defined by the following classical condition: 
\begin{equation}\label{g445}
\left\{ (C, p_1, \dots, p_n) \in \cM_{g,n} \; \rule[-0.8em]{0.05em}{2em} \; 
\cO_C\Big(\sum_{i=1}^n a_i p_i\Big) \simeq \cO_C \right\}\, . 
\end{equation}
From the point of view of relative Gromov-Witten theory, 
the most natural compactification of the substack \eqref{g445} 
is the space $\oM^{\sim}_{g,A}$ of stable maps to
{\em rubber}\,: stable maps to
 $\CP^1$ relative to $0$ and $\infty$ modulo the $\C^*$-action on $\CP^1$.
The rubber moduli space carries a natural virtual fundamental class $\left[\oM^{\sim}_{g,A}\right]^\virt$ of (complex)
dimension $2g-3+n$. The pushforward 
via the canonical morphism
$$ \epsilon:\oM^{\sim}_{g,A} \rightarrow  \oM_{g,n}$$
is the {\em double ramification cycle} 
$$\epsilon_*\left[\oM^{\sim}_{g,A}\right]^\virt\, =\, \mathsf{DR}_{g,A}\ \in 
{\mathsf{CH}}_{2g-3+n}(\overline{\cM}_{g,n})\,. $$

Double ramification cycles have been studied intensively for the 
past two decades.
Examples of early results can be found in \cite{BSSZ,Cavalieri2011Polynomial-fami,cj,Faber2005Relative-maps-a,Grushevsky2012The-double-rami,Hain2013Normal-function, Marcus2013Stable-maps-to-}.
A complete formula was conjectured by
Pixton in 2014 and proven in \cite{Janda2016Double-ramifica}. 
For subsequent 
study and applications, see \cite{Bae,BGR,CGJZ,FanWu,Farkas2016The-moduli-spac,Holmes2017Extending-the-d,
Holmes2017Jacobian-extens,
Holmes2017Multiplicativit,MPS20,ObPix,Pix18,Schmitt2016Dimension-theor,
Tseng2,Tseng1}. Essential for our work is the formula for double ramification
cycles for target varieties in \cite{Janda2018Double-ramifica}.

We refer the
reader to \cite[Section 0]{Janda2016Double-ramifica} and \cite[Section 5]{RPSLC} for introductions to the subject. For a classical perspective from the point of view
of Abel-Jacobi theory, see 
\cite{Holmes2017Extending-the-d}.

\subsection{Twisted double ramification cycles}
We develop here a theory which extends the study of double ramification
cycles from the moduli space of stable  curves $\overline{\cM}_{g,n}$ to the Picard stack of
curves with line bundles $\Picabs_{g,n}$.
An object of $\Picabs_{g,n}$ over $\mathcal{S}$ is
a flat family $$\pi:\mathcal{C} \rightarrow \mathcal{S}$$ of prestable{\footnote{A prestable $n$-pointed
curve is a connected nodal curve with markings at distinct nonsingular points.
For the entire paper, we avoid the $(g,n)=(1,0)$ case because of 
non-affine stabilizers.}}
    $n$-pointed genus $g$ curves together with a line bundle
    $$\mathcal{L} \rightarrow \mathcal{C}\,. $$
The Picard stack $\Picabs_{g,n}$ is an algebraic (Artin) stack
which is locally of finite type, see \cref{sec:pic_stacks} for a treatment
of foundational issues.

Since the degree of a line bundle is constant in flat families, there is a
disjoint union
$$\Picabs_{g,n} =\bigcup_{d\in \mathbb{Z}}  \Picabs_{g,n,d}\, ,$$
where $\Picabs_{g,n,d}$ is the Picard stack of curves with degree $d$
line bundles.
Let $$A = (a_1, \dots, a_n)\,,\ \ \ \ \  \sum_{i=1}^n a_i=d \,$$
be a vector of integers.
The first result of the paper is the construction of a 
{\em universal twisted double ramification
cycle} in the operational Chow theory{\footnote{All
Chow theories in the paper will be taken with $\mathbb{Q}$-coefficients.
}} of $\Picabs_{g,n,d}$,
$$\mathsf{DR}^{\mathsf{op}}_{g,A} \in {\mathsf{CH}}^g_{\mathsf{op}}(\Picabs_{g,n,d})\, .$$

The geometric intuition behind the construction is simple.
Let 
$$\pi:\mathcal{C}\rightarrow \mathcal{S}\, ,
\ \ \ \ p_1,\ldots,p_n: \mathcal{S} \rightarrow \mathcal{C}\, ,
\ \ \ \ \mathcal{L}\rightarrow \mathcal{C}$$
be an object of $\Picabs_{g,n,d}$.
The class
$\mathsf{DR}^{\mathsf{op}}_{g,A}$ should operate as the 
locus in the base $\mathcal{S}$ heuristically determined by the condition
$$
\cO_C\Big(\sum_{i=1}^n a_i p_i\Big) \simeq \mathcal{L}|_{C} \,.$$


To make the above idea precise, we do {\em not} use the virtual class of the moduli space of stable
maps in Gromov-Witten theory, but rather an alternative approach by
partially resolving the classical Abel-Jacobi map.
The method follows the path 
of \cite{Holmes2017Extending-the-d, Holmes2017Multiplicativit}
and
may be viewed  as a universal Abel-Jacobi construction
over the Picard stack. 
Log geometry based on the stack of tropical divisors constructed in \cite{Marcus2017Logarithmic-com}
plays a crucial role.
Our construction is presented in Section \ref{Sect:AJextension}.

The basic compatibility of our new operational class $$\mathsf{DR}^{\mathsf{op}}_{g,A} \in {\mathsf{CH}}^g_{\mathsf{op}}(\Picabs_{g,n,d})$$ with the
standard double ramification cycle is as follows.
Let 
$$A = (a_1, \dots, a_n)\,,\ \ \ \ \  \sum_{i=1}^n a_i=0 \,,$$
be given. 
The universal data 
\begin{equation}\label{fftt66677}
\pi: \mathcal{C}_{g,n} \rightarrow \overline{\cM}_{g,n}\, , \ \ \ \ \cO \rightarrow \mathcal{C}_{g,n}
\end{equation}
determine a map $\phi_{\ca O} : \overline{\cM}_{g,n} \to \Picabs_{g,n,0}$.
The action of $\mathsf{DR}^{\mathsf{op}}_{g,A}$ on the fundamental class
of $\overline{\mathcal{M}}_{g,n}$ corresponding to the family \eqref{fftt66677} then equals the previously defined
double ramification cycle
$$\mathsf{DR}^{\mathsf{op}}_{g,A}(\phi_{\ca O})\Big( [\overline{\mathcal{M}}_{g,n}]\Big)= \mathsf{DR}_{g,A}
\in {\mathsf{CH}}_{2g-3+n}(\overline{\mathcal{M}}_{g,n})
\, .$$

More generally, 
for a vector $A = (a_1, \dots, a_n)$ of integers satisfying
$$\sum_{i=1}^n a_i = k(2g-2)\,, $$ 
canonically twisted double ramification cycles,
$$\mathsf{DR}_{g,A, \omega^k}
\in {\mathsf{CH}}_{2g-3+n}(\overline{\mathcal{M}}_{g,n})\, ,$$
related to the classical loci
\begin{equation*}
\left\{ (C, p_1, \dots, p_n) \in \cM_{g,n} \; \rule[-0.8em]{0.05em}{2em} \; 
\cO_C\Big(\sum_{i=1}^n a_i p_i\Big) \simeq  \omega^k_C \right\},\,  
\end{equation*}
have been constructed in
\cite{Guere2016A-generalizatio} for $k=1$ and in \cite{Holmes2017Extending-the-d,Holmes2017Jacobian-extens,Marcus2017Logarithmic-com} for all $k\geq 1$.
The universal data 
\begin{equation}\label{fftt66672}
\pi: \mathcal{C}_{g,n} \rightarrow \overline{\cM}_{g,n}\, , \ \ \ \ \omega_\pi^k \rightarrow \mathcal{C}_{g,n}
\end{equation}
determine a map 
$\phi_{\omega_\pi^k} : \overline{\cM}_{g,n} \to \Picabs_{g,n,k(2g-2)}$. 
Here,
$\omega_\pi$ is the relative dualizing sheaf
of the morphism $\pi$.

The action of $\mathsf{DR}^{\mathsf{op}}_{g,A}$ on the fundamental class
of $\overline{\mathcal{M}}_{g,n}$ corresponding to the family \eqref{fftt66672} is compatible with the constructions of \cite{Guere2016A-generalizatio,Holmes2017Extending-the-d,Holmes2017Jacobian-extens,Marcus2017Logarithmic-com},
$$\mathsf{DR}^{\mathsf{op}}_{g,A} (\phi_{\omega_\pi^k})\Big( [\overline{\mathcal{M}}_{g,n}]\Big)= \mathsf{DR}_{g,A, \omega^k}
\in {\mathsf{CH}}_{2g-3+n}(\overline{\mathcal{M}}_{g,n})
\, $$
for all $k\geq 1$.

The above compatibilities of $\mathsf{DR}^{\mathsf{op}}_{g,A}$ with
the standard and canonically twisted double ramification cycles are proven in Section \ref{proof1111}.

\begin{theorem} \label{cccc}
Let $g\geq 0$ and $d\in \mathbb{Z}$.
Let $A=(a_1,\ldots,a_n)$ be a vector of integers satisfying
$$\sum_{i=1}^n a_i=d\, .$$
Logarithmic compactification of the Abel-Jacobi map yields a universal twisted double
ramification 
cycle
$$\mathsf{DR}^{\mathsf{op}}_{g,A} \in {\mathsf{CH}}^g_{\mathsf{op}}(\Picabs_{g,n,d})$$
which is compatible with the standard double ramification
cycle 
$$\mathsf{DR}_{g,A, \omega^k}\in \mathsf{CH}_{2g-3+n}(\overline{\cM}_{g,n})\, $$
in case $d=k(2g-2)$ for $k\geq 0$.
\end{theorem}



\subsection{Pixton's formula}

\subsubsection{Prestable graphs}\label{xvalsg}
We define the set $\G_{g,n}$ of {\em prestable graphs} as
follows. A prestable graph $\Gamma \in \G_{g,n}$ consists of the data
$$\Gamma=(\V\, ,\ \H\, ,\ \L\, , \ \mathrm{g}:\V \rarr \Z_{\geq 0}\, ,
\ v:\H\rarr \V\, , 
\ \iota : \H\rarr \H)$$
satisfying the properties:
\begin{enumerate}
\item[(i)] $\V$ is a vertex set with a genus function $\g:\V\to \Z_{\geq 0}$,
\item[(ii)] $\H$ is a half-edge set equipped with a 
vertex assignment $v:\H \to \V$ and an involution $\iota$,
\item[(iii)] $\E$, the edge set, is defined by the
2-cycles of $\iota$ in $\H$ (self-edges at vertices
are permitted),
\item[(iv)] $\L$, the set of legs, is defined by the fixed points of $\iota$ and is
placed in bijective correspondence with a set of $n$ markings,
\item[(v)] the pair $(\V,\E)$ defines a {\em connected} graph
satisfying the genus condition 
$$\nice \sum_{v \in \V} \g(v) + h^1(\Gamma) = g\, .$$
\end{enumerate}
To emphasize $\Gamma$, the notation $\V(\Gamma)$, $\H(\Gamma)$, $\L(\Gamma)$, and $\E(\Gamma)$ will also be used for the vertex, half-edges, legs, and
edges of $\Gamma$. 

An isomorphism between $\Gamma$, $\Gamma' \in \G_{g,n}$ 
consists of bijections $\V \to \V'$ and $\H \to \H'$ respecting the
structures $\L$, $\mathrm{g}$, $v$, and $\iota$.
Let $\Aut(\Gamma)$ denote the automorphism group of $\Gamma$.


While the set of isomorphism classes of prestable graphs is infinite, the set of isomorphism classes of prestable graphs with prescribed bounds on the number of edges is finite.

Let $\mathfrak{M}_{g,n}$ be the algebraic (Artin) stack of prestable curves of genus $g$ with $n$ marked points.
A prestable graph $\Gamma$ determines an algebraic stack $\mathfrak{M}_\Gamma$ of 
curves
with degenerations forced by the graph,
$$
\mathfrak{M}_\Gamma = \prod_{v \in \V} \mathfrak{M}_{\g(v), \n(v)}\, $$
together with a 
canonical\footnote{To define the map, we choose an ordering on the half-edges at each vertex. } map
$$j_\Gamma : \mathfrak{M}_\Gamma \to \mathfrak{M}_{g,n}\, .$$
Since $\mathfrak{M}_{g,n}$ is smooth and
the morphism $j_\Gamma$ is proper, representable, and lci, we obtain an operational Chow class on the algebraic
stack of curves,
$$j_{\Gamma*}[\mathfrak{M}_\Gamma] \in \mathsf{CH}_{\mathsf{op}}^{|\E(\Gamma)|}(\mathfrak{M}_{g,n})\, .$$
Via the morphism of algebraic stacks,
$$\epsilon: \Picabs_{g,n,d} \rightarrow \mathfrak{M}_{g,n}\, ,$$
$j_{\Gamma*}[\mathfrak{M}_\Gamma]$ also defines an operational Chow class
on the Picard stack,
$$\epsilon^* j_{\Gamma*}[\mathfrak{M}_\Gamma] \in \mathsf{CH}_{\mathsf{op}}^{|\E(\Gamma)|}(\Picabs_{g,n,d})\, .$$

\subsubsection{Prestable graphs with degrees} \label{Sect:Prestgrwdeg}
We will require a refinement of the prestable graphs of Section \ref{xvalsg} which includes degrees
of line bundles. 

We define the set $\G_{g,n,d}$ of {\em prestable graphs of degree $d$} as
follows: $$\Gamma_\delta=(\Gamma,\delta) \in \G_{g,n,d}$$ consists of the data
\begin{enumerate}
    \item[$\bullet$] a prestable graph $\Gamma\in \mathsf{G}_{g,n}$,
    \item[$\bullet$] a function $\delta: \V \rightarrow \mathbb{Z}$
    satisfying the degree condition
$$\sum_{v\in V} \delta(v) = d\, .$$
\end{enumerate}
The function $\delta$ is  often called the {\em multidegree}.

An automorphism of $\Gamma_\delta \in \G_{g,n,d}$ 
consists of an automorphism of $\Gamma$
leaving $\delta$ invariant.
Let $\Aut(\Gamma_\delta)$ denote the automorphism group of $\Gamma_\delta$.


For $\Gamma_\delta\in \mathsf{G}_{g,n,k}$, let $\mathfrak{M}_\Gamma$ be the algebraic stack of
curves defined in Section \ref{xvalsg} with respect to the underlying prestable graph $\Gamma$. 
Let
$\Picabs_{\Gamma_\delta}$ be the Picard stack,
$$\epsilon:\Picabs_{\Gamma_\delta} \rightarrow \mathfrak{M}_{\Gamma}\, ,$$
parameterizing curves with degenerations forced by $\Gamma$ and with  line bundles which
have 
degree $\delta(v)$ restriction to the components corresponding to the vertex $v\in \V$.
We have a
canonical map
$$j_{\Gamma_\delta} : \Picabs_{\Gamma_\delta} \to \Picabs_{g,n,d}\, .$$
Since $\Picabs_{g,n,d}$ is smooth and
the morphism $j_{\Gamma_\delta}$ is proper, representable, and lci,
we obtain an operational Chow class,
$$j_{\Gamma_\delta *}[\Picabs_{\Gamma_\delta}] \in \mathsf{CH}_{\mathsf{op}}^{|\E(\Gamma)|}(\Picabs_{g,n,d})\, .$$
As operational Chow classes, the following formula holds:
\begin{equation}\label{5599w}
\epsilon^*j_{\Gamma*}[\mathfrak{M}_\Gamma] = \sum_{\delta} j_{\Gamma_\delta *}[\Picabs_{\Gamma_\delta}]
\ \in \mathsf{CH}_{\mathsf{op}}^{|\E(\Gamma)|}(\Picabs_{g,n,d})
\, ,
\end{equation}
where the sum{\footnote{The sum is infinite, but only finitely many terms
are nonzero in any operation.}} is over all functions $\delta:\V \rightarrow \mathbb{Z}$ satisfying the degree condition.
Equivalently, we may write \eqref{5599w} as
\begin{equation*}
\frac{1}{|{\text{Aut}}(\Gamma)|}\, \epsilon^*j_{\Gamma*}[\mathfrak{M}_\Gamma] = 
\sum_{\Gamma_\delta\in \mathsf{G}_{g,n,d}} 
\frac{1}{|{\text{Aut}}(\Gamma_\delta)|}\, 
j_{\Gamma_\delta *}[\Picabs_{\Gamma_\delta}]\ 
\in \mathsf{CH}_{\mathsf{op}}^{|\E(\Gamma)|}(\Picabs_{g,n,d})
\, ,
\end{equation*}
where the sum on the right side is now over all isomorphism classes of
prestable graphs of degree $d$ with underlying prestable graph $\Gamma$.

\subsubsection{Tautological \texorpdfstring{$\psi$}{psi}, \texorpdfstring{$\xi$}{xi}, and \texorpdfstring{$\eta$}{eta} classes} \label{Ssec:PsiXiEta}

The universal curve $$\pi:\mathfrak{C}_{g,n} \to \Picabs_{g,n}$$ 
carries two natural line bundles: the relative dualizing sheaf $\omega_{\pi}$ 
and the universal line bundle
$$\mathfrak{L} \rightarrow \mathfrak{C}_{g,n}\, .$$
Let $p_i$ be the $i$th section of the universal curve, let 
$$\mathfrak{S}_i\subset \mathfrak{C}_{g,n}$$
be  the corresponding divisor, and let
$$\omega_\log = \omega_\pi\Big(\sum_{i=1}^n \mathfrak{S}_i\Big)$$ be the relative log-canonical line bundle 
with first Chern class $c_1 (\omega_\log)$. 
Let $$\xi = c_1(\mathfrak{L})$$ be the first Chern class of $\mathfrak{L}$.

\begin{definition} \label{Not:PsiXiEta}
The following operational classes on $\Picabs_{g,n}$ are obtained from the
universal structures:
\begin{itemize}
    \item $\psi_i = c_1(p_i^* \omega_\pi)\in \mathsf{CH}^1_{\mathsf{op}}(\Picabs_{g,n})\, $,
    \item $\xi_i = c_1(p_i^* \mathfrak{L})\in \mathsf{CH}^1_{\mathsf{op}}(\Picabs_{g,n})
    \, $,
    \item $\eta_{a,b} = \pi_*\left(c_1(\omega_\log)^a \xi^b\right)
    \in \mathsf{CH}^{a+b-1}_{\mathsf{op}}(\Picabs_{g,n})
    \, $.
\end{itemize}
\end{definition}
For simplicity in the formulas, 
we will use the
notation
$$\eta= \eta_{0,2} =\pi_*(\xi^2)\, .$$
The standard $\kappa$ classes are defined by the $\pi$ pushforwards
of powers of $c_1(\omega_\log)$, 
$$\eta_{a,0} = \kappa_{a-1}\, .$$

\begin{definition}
A {\em decorated prestable graph} $[\Gamma_\delta,\gamma]$ {\em of degree $d$} 
is a prestable graph $\Gamma_\delta \in \G_{g,n,d}$ of degree $d$ together with the following
decoration data $\gamma$:
\begin{itemize}
\item each leg $i\in \L$ is decorated with a monomial $\psi_i^a \xi_i^b$,
\item each half-edge $h\in \H\setminus \L$ is 
decorated with a monomial $\psi^a_h$,
\item each edge $e\in \E$ is decorated with a monomial $\xi^a_e$,
\item each vertex in $\V$ is decorated with a monomial in the variables 
$\{\eta_{a,b}\}_{a+b \geq 2}$.
\end{itemize}
In all four cases, the monomial may be be trivial.
\end{definition}

Let $\D\G_{g,n,d}$ be the set of decorated prestable graphs of degree $d$.
To each decorated graph of degree $d$, 
$$[\Gamma_\delta,\gamma]\in \D\G_{g,n,d}\,,$$ we assign the 
operational class 
$$j_{\Gamma_\delta*}[\gamma] \in \mathsf{CH}^*_{\mathsf{op}}(\Picabs_{g,n,d})$$
obtained
via the proper representable morphism
$$j_{\Gamma_\delta} : \Picabs_{\Gamma_\delta} \rightarrow \Picabs_{g,n,d}$$
and the action of the decorations.

The action of decorations is described
as follows.
Given $\Gamma_\delta\in \mathsf{G}_{g,n,k}$, the stack $\Picabs_{\Gamma_\delta}$ admits a morphism{\footnote{The fibers of the map are torsors under the group $\mathbb{G}_m^{h^1(\Gamma)}$.}}
\[\Picabs_{\Gamma_\delta} \to \prod_{v \in \V(\Gamma_\delta)} \Picabs_{\g(v),\n(v),\delta(v)}\]
which sends a line bundle $\mathcal{L}$ on a prestable curve $\mathcal{C}$ to its restrictions on the various components,
$$ \mathcal{L}|_{\mathcal{C}_v}\, , \  \ \  v\in \V(\Gamma_\delta)\, . $$
For $v \in \V(\Gamma_\delta)$, we define the operational
class $\eta(v)$ on $\Picabs_{\Gamma_\delta}$ as the pullback of the operational class $\eta$ on the factor $\Picabs_{\g(v),\n(v),\delta(v)}$ above. The operational
classes $\psi$ at the markings and $\xi$ at the half-edges are
defined similarly.

\begin{definition} The {\em tautological  classes} in
 $\mathsf{CH}^*_{\mathsf{op}}(\Picabs_{g,n,d})$ consist of the 
 $\mathbb{Q}$-linear span
 of the operational classes associated to 
 all $[\Gamma_\delta,\gamma]\in \D\G_{g,n,d}$.
\end{definition}

By standard analysis \cite{GraberPandharipande}, the tautological classes
have a natural ring structure.
Our formula for $\mathsf{DR}_{g,A}^{\mathsf{op}}$ 
will be a sum of operational classes
determined by decorated prestable graphs of degree $d=\sum_{i=1}^n a_i$
(and hence will be tautological).

\subsubsection{Weightings mod \texorpdfstring{$r$}{r}} \label{Ssec:weightings}
Let $\Gamma_\delta\in \G_{g,n,d}$ be a prestable graph of degree $d$, and
let $r$ be a positive integer.

\begin{definition} A {\em weighting mod~$r$} of $\Gamma_\delta$ is a function on the set of half-edges,
$$ w:\H(\Gamma_\delta) \rightarrow \{0,1, \ldots, r-1\}\, ,$$
which satisfies the following three properties:

\begin{enumerate}
\item[(i)] $\forall i\in \L(\Gamma_\delta)$, corresponding to
 the marking $i\in \{1,\ldots, n\}$,
$$w(i)=a_i  \mod r\, ,$$
\item[(ii)] $\forall e \in \E(\Gamma_\delta)$, corresponding to two half-edges
$h,h' \in \H(\Gamma_\delta)$,
$$w(h)+w(h')=0 \mod r\, ,$$
\item[(iii)] $\forall v\in \V(\Gamma_\delta)$,
$$\sum_{v(h)= v} w(h)= \delta(v) \mod r\, ,$$ 
where the sum is taken over {\em all} $\mathsf{n}(v)$ half-edges incident 
to $v$.\end{enumerate}
\end{definition}

We denote by $\mathsf{W}_{\Gamma_\delta,r}$ the finite set of all possible weightings
mod $r$
of
$\Gamma_\delta$. The set $\mathsf{W}_{\Gamma_\delta,r}$ has cardinality $r^{h^1(\Gamma_\delta)}$.
We view $r$ as a {\em regularization parameter}.

\subsubsection{Calculation of the twisted double ramification cycle} \label{Ssec:MainFormula}
We denote by
$\P_{g,A,d}^{c,r}\in \mathsf{CH}^c_{\mathsf{op}}(\Picabs_{g,n,d})$
the codimension{\footnote{Codimension here is usually
called degree. But since we already have line bundle degrees,
we use the term codimension for clarity.}}
 $c$ component of the tautological operational class 
\begin{multline*}
\hspace{-10pt}\sum_{
\substack{\Gamma_\delta\in \G_{g,n,d} \\
w\in \mathsf{W}_{\Gamma_\delta,r}}
}
\frac{r^{-h^1(\Gamma_\delta)}}{|\Aut(\Gamma_\delta)| }
\;
j_{\Gamma_\delta*}\Bigg[
\prod_{i=1}^n \exp\left(\frac12 a_i^2 \psi_i + a_i \xi_i \right)
\prod_{v \in \V(\Gamma_\delta)} \exp\left(-\frac12 \eta(v) \right)
\\ \hspace{+10pt}
\prod_{e=(h,h')\in \E(\Gamma_\delta)}
\frac{1-\exp\left(-\frac{w(h)w(h')}2(\psi_h+\psi_{h'})\right)}{\psi_h + \psi_{h'}} \Bigg]\, .
\end{multline*} 
Several remarks about the formula are required:
\begin{enumerate}
    \item[(i)]
    The sum is over all isomorphism classes of prestable graphs of degree $d$ in the
    set $\mathsf{G}_{g,n,d}$. 
    Only finitely many underlying prestable graphs $\Gamma\in \mathsf{G}_{g,n}$ can
    contribute in fixed codimension $c$. However, for each such prestable graph, the above formula
    has infinitely many summands corresponding to the infinitely many functions
    $$\delta: \V \rightarrow \mathbb{Z}$$
    which satisfy the degree condition.
    The operational Chow class $\P_{g,A,d}^{c,r}$ is nevertheless well-defined since
      only finitely many summands have nonvanishing operation on any given family of curves
       carrying a degree $d$ line bundle over a base $\mathcal{S}$ of finite type.
    \item[(ii)] 
    Once the prestable graph $\Gamma_\delta$ is chosen, we sum over all $r^{h^1(\Gamma_\delta)}$ different 
    weightings $w\in \mathsf{W}_{\Gamma_\delta,r}$.
    \item[(iii)]
Inside the pushforward in the above formula, the first product 
$$\prod_{i=1}^n \exp\left(\frac{1}{2}a_i^2 \psi_{h_i}+a_i \xi_{h_i}\right)\, $$
is over $h\in \L(\Gamma)$
via the correspondence of legs and markings.
\item[(iv)]The class $\eta(v)$ is the $\eta_{0,2}$ class 
of Definition \ref{Not:PsiXiEta} associated to the vertex.
\item[(v)]
The third product is over all $e\in \E(\Gamma_\delta)$.
The factor 
$$\frac{1-\exp\left(-\frac{w(h)w(h')}{2}(\psi_h+\psi_{h'})\right)}{\psi_h + \psi_{h'}}$$
is well-defined since 
\begin{enumerate}
\item[$\bullet$]
the denominator formally divides
the numerator,
\item[$\bullet$] the factor is symmetric in $h$ and $h'$.
\end{enumerate}
No edge orientation is necessary.
\end{enumerate}

The following fundamental polynomiality property of $\P_{g,A,d}^{c,r}$ 
 is
parallel to Pixton's polynomiality in \cite[Appendix]{Janda2016Double-ramifica} and
is a consequence of \cite[Proposition $3''$]{Janda2016Double-ramifica}.

\begin{proposition} \label{pply} For fixed $g$, $A$, $d$, and $c$ and a decorated graph $[\Gamma_\delta, \gamma]$ of
degree $d$, the coefficient of $j_{\Gamma_\delta*} [\gamma]$ in the
tautological class
$$\P_{g,A,d}^{c,r} \in \mathsf{CH}^c_{\mathsf{op}}(\Picabs_{g,n,d})$$
is polynomial in $r$ (for all sufficiently large $r$).
\end{proposition}

We denote by $\P_{g,A,d}^c$ the value at $r=0$ 
of the polynomial associated to $\P_{g,A,d}^{c,r}$ by Proposition~\ref{pply}. In other words, $\P_{g,A,d}^c$ is the {\em constant} term of the associated polynomial in $r$. 

The main result of the paper is a 
 formula for the universal twisted double ramification cycle in operational Chow.{\footnote{Our
handling of the prefactor $2^{-g}$ in \cite[Theorem 1]{Janda2016Double-ramifica} differs here.
The factors of $2$ are now placed in the definition of $\P_{g,A,d}^{c,r}$
as in \cite{Janda2018Double-ramifica}}

\begin{theorem} \label{Thm:main}
Let $g\geq 0$ and $d\in \mathbb{Z}$.
Let $A=(a_1,\ldots,a_n)$ be a vector of integers with
$\sum_{i=1}^n a_i=d\,$.
The universal twisted double ramification cycle is calculated by Pixton's formula:
$$\mathsf{DR}^{\mathsf{op}}_{g,A} =\P_{g,A,d}^g\
\in {\mathsf{CH}}^g_{\mathsf{op}}(\Picabs_{g,n,d})\, .$$
\end{theorem}

Theorem \ref{Thm:main} is the most fundamental formulation of the relationship
between Abel-Jacobi theory and Pixton's formula that we know. Certainly,  
Theorem \ref{Thm:main} 
implies the double ramification cycle and $X$-valued
double ramification cycle results of \cite{Janda2016Double-ramifica,Janda2018Double-ramifica} }. But since we will
use \cite{Janda2018Double-ramifica} in the proof of Theorem \ref{Thm:main}, we provide no new approach to
these older results. However, the additional depth of Theorem \ref{Thm:main} immediately
allows new applications.

\subsection{Vanishing}
From his original double
ramification cycle formula,
Pixton conjectured an associated
vanishing property in the
tautological ring of the moduli
space of curves which
was proven by Clader and Janda \cite{cj}.
The parallel vanishing statement
in the tautological ring of the
moduli space of stable maps to $X$
was proven in \cite{Bae}.
The most general vanishing
statement is the following
result.

\begin{theorem} \label{Thm:mainvan}
Let $g\geq 0$ and $d\in \mathbb{Z}$.
Let $A=(a_1,\ldots,a_n)$ be a vector of integers with
$\sum_{i=1}^n a_i=d$.
 Pixton's vanishing holds
 in operational Chow:
$$ \P_{g,A,d}^c\ = 0 \
\in {\mathsf{CH}}^c_{\mathsf{op}}(\Picabs_{g,n,d})\, \ \ \ 
{\text{for all}} \ \ c>g\, .
$$ 
\end{theorem}

Theorem \ref{Thm:mainvan} may be viewed
as providing relations among
tautological classes in the operational
Chow ring of the Picard stack -- a new
direction of study with many open questions.{\footnote{See \cite{RandalWilliams2012,Yin}
for tautological relations on the Picard variety over
the moduli space of smooth curves.}}
While Theorem \ref{Thm:mainvan} 
implies the
vanishings of
\cite{Bae, cj}, 
we will use these results in our
proof.

\subsection{Twisted holomorphic and meromorphic differentials} \label{Sect:Twistdiffintro}
\subsubsection{Fundamental classes}
Let $A=(a_1,\ldots,a_n)$ be a vector of zero and pole multiplicities
satisfying
$$\sum_{i=1}^n a_i= 2g-2\, .$$
Let $\HH_g(A) \subset \mathcal{M}_{g,n}$ be the quasi-projective locus 
of pointed curves $(C,p_1,\ldots,p_n)$ satisfying the condition
$$\mathcal{O}_C\Big(\sum_{i=1}^n a_i p_i\Big) \simeq \omega_C\, .$$
In other words, $\HH_g(A)$ is the locus of meromorphic differentials{\footnote{If all
the parts of $A$ are non-negative, then $\HH_g(A)$ is the locus of
holomorphic differentials.}}
with zero and pole multiplicities prescribed by $A$.
In \cite{Farkas2016The-moduli-spac}, a compact moduli space of twisted canonical divisors
$$\widetilde{\HH}_g(A) \subset \overline{\mathcal{M}}_{g,n}$$
is constructed which contains $\HH_g(A)$ as an open set.

In the strictly meromorphic
case, where $A$ contains at least one strictly negative part,
$\widetilde{\HH}_g(A)$ is of pure codimension $g$ in
$\overline{\mathcal{M}}_{g,n}$ by \cite[Theorem 3]{Farkas2016The-moduli-spac}. 
A weighted fundamental cycle of $\widetilde{\HH}_g(A)$, 
\begin{equation}\label{wwww4}
\mathsf{H}_{g,A}\in \mathsf{CH}_{2g-3+n}(\oM_{g,n})\ ,
\end{equation}
is constructed in
\cite[Appendix A]{Farkas2016The-moduli-spac} with explicit nontrivial 
weights on the irreducible components.
In the strictly meromorphic case,
$\HH_{g}(A)\subset \oM_{g,n}$ 
is also of pure codimension $g$.
The closure
$$\overline{\HH}_{g}(A)\subset \oM_{g,n}$$
contributes to the fundamental class $\mathsf{H}_{g,A}$ with
multiplicity 1, but there are additional boundary contributions, see
\cite[Appendix A]{Farkas2016The-moduli-spac}.

The  
universal family over the moduli space of
stable curves together with the 
relative dualizing sheaf,
\begin{equation}\label{rrrr}
\pi: \mathcal{C}_{g,n} \rightarrow \overline{\cM}_{g,n}\, , \ \ \ \ \omega_\pi \rightarrow \mathcal{C}_{g,n}\, ,
\end{equation}
determine an object of
$\Picabs_{g,n,2g-2}$.
By \cite{Holmes2019Infinitesimal-s} and
the compatibility of Theorem \ref{cccc},
the action of $\mathsf{DR}^{\mathsf{op}}_{g,A}$ on the fundamental class
of $\overline{\mathcal{M}}_{g,n}$ equals the  weighted
fundamental class of $\widetilde{\HH}_g(A)$,
$$\mathsf{DR}^{\mathsf{op}}_{g,A}(\phi_\omega)\Big( [\overline{\mathcal{M}}_{g,n}]\Big)= \mathsf{H}_{g,A}
\in {\mathsf{CH}}_{2g-3+n}(\overline{\mathcal{M}}_{g,n})
\, .$$
We can now apply Theorem \ref{Thm:main} to prove the following result.
 
\begin{theorem}\label{FPA1}
In the strictly meromorphic case,
$$\mathsf{H}_{g,A} = \mathsf{P}_{g,A,2g-2}^g[\oM_{g,n}]$$
for the universal family
$$\pi: \mathcal{C}_{g,n} \rightarrow \overline{\cM}_{g,n}\, ,\ \ \ \
\omega_\pi \rightarrow \mathcal{C}_{g,n}\,. $$
\end{theorem}

Theorem \ref{FPA1} is exactly equivalent to Conjecture A of 
\cite[Appendix A]{Farkas2016The-moduli-spac}.
Since both the moduli space  $\widetilde{\HH}_g(A)$ and the weighted
fundamental cycle 
$\mathsf{H}_{g,A}$
have explicit geometric definitions, 
the result provides a geometric representative
of Pixton's cycle class in terms of twisted differentials. 
Theorem \ref{FPA1} is proven in Section \ref{Sect:ProofConjA} where the parallel conjectures \cite{Schmitt2016Dimension-theor}
for higher differentials  are also proven (by parallel  arguments).

\subsubsection{Closures}
Let $A=(a_1,\ldots, a_n)$ be a vector of integers satisfying
$$\sum_{i=1}^n a_i =2g-2\, .$$
A careful investigation of the closure
$$\HH_g(A) \subset \overline{\HH}_g(A) \subset \oM_{g,n}$$
is carried out in \cite{Bainbridge2018Compactificatio} in
both the holomorphic and meromorphic cases.
By a simple procedure presented in \cite[Appendix]{Farkas2016The-moduli-spac},
Theorem \ref{FPA1} determines the cycle classes of the 
closures
$$ [\overline{\HH}_g(A)] \in \mathsf{CH}_{*} (\oM_{g,n})\, .$$
for $A$ in both the holomorphic and meromorphic cases. 

A similar procedure determines the corresponding classes for $k$-differentials, see \cite[Section 3.4]{Schmitt2016Dimension-theor} for an explanation. In particular, 
our results imply that the cycle classes of the closures 
are tautological{\footnote{The precise statement is given in
Corollary \ref{kk333} of Section \ref{clozzz}.}} 
for all $k$ (as was previously known only for $k=1$ due to \cite{Sauvaget2017Cohomology-clas}).

In the case of holomorphic differentials, another approach to the class of
the closure $\overline{\HH}_g(A) \subset \oM_{g,n}$ is provided
by Conjecture A.1 of the Appendix of \cite{Pandharipande2019Tautological-re} via a limit of
Witten's $r$-spin class. A significant first step in the proof of \cite[Conjecture A.1]{Pandharipande2019Tautological-re} by 
Chen, Janda, Ruan, and Sauvaget can be found in \cite{Chen2018Towards-a-Theor}. Further progress
requires 
a virtual localization analysis for moduli spaces of 
stable log maps. An approach to 
Theorem \ref{FPA1} using log stable maps, virtual localization in the log context, and
the strategy of \cite{Janda2016Double-ramifica} also appears possible (once the
required moduli spaces and localization formulas are established).


\subsection{Invariance properties} \label{Sect:introinvariance}


The 
universal twisted double ramification cycle has several
basic invariance properties 
which play 
an important role in our study. 

Recall that an object of $\Picabs_{g,n,d}$ over $\mathcal{S}$ is
a flat family  of prestable
    $n$-pointed genus $g$ curves together with a line bundle
    of relative degree $d$,
    \begin{equation}
    \label{ff449911}    
    \pi:\mathcal{C} \rightarrow \mathcal{S}\, , \ \ \ \ p_1,\ldots,p_n: \mathcal{S} \rightarrow \mathcal{C}\, , \ \ \ \
        \mathcal{L} \rightarrow \mathcal{C}\, .
        \end{equation}
    Let $\mathsf{DR}^{\mathsf{op}}_{g,A,\mathcal{L}}\in \mathsf{CH}^g_{\mathsf{op}}(\mathcal{S})$
    be the twisted double ramification cycle associated to the above
    family \eqref{ff449911} and
     the vector
    $$A=(a_1,\ldots, a_n)\, , \ \ \ \ d =\sum_{i=1}^n a_i\, .$$

        \noindent{\bf{Invariance \invnrI: {\rm Dualizing}.}}
\vspace{10pt}

A new object of $\Picabs_{g,n,-d}$  over $\mathcal{S}$ is
obtained from \eqref{ff449911} by dualizing $\mathcal{L}$:
    \begin{equation}
    \label{ff44992}    
    \pi:\mathcal{C} \rightarrow \mathcal{S}\, , \ \ \ \ p_1,\ldots,p_n: \mathcal{S} \rightarrow \mathcal{C}\, , \ \ \ \
        \mathcal{L}^* \rightarrow \mathcal{C}\, .
        \end{equation}
 Let $\mathsf{DR}^{\mathsf{op}}_{g,-A, \mathcal{L}^*}\in \mathsf{CH}^g_{\mathsf{op}}(\mathcal{S})$ be the twisted double ramification cycle associated to the new family \eqref{ff44992} and the vector $-A=(-a_1,\ldots,-a_n)$. We have the invariance
        $$\mathsf{DR}^{\mathsf{op}}_{g,-A, \mathcal{L}^*}
        =\epsilon^* \mathsf{DR}^{\mathsf{op}}_{g,A, \mathcal{L}}
        \, , $$
where $\epsilon: \Picabs_{g,n,-d} \rightarrow \Picabs_{g,n,d}$
is the natural map obtained via dualizing the line bundle.

 \vspace{10pt}
        \noindent{\bf{Invariance \invnrn: {\rm Unweighted markings}.}}
\vspace{10pt}

Assume we have an additional section $p_{n+1}: \mathcal{S} \rightarrow \mathcal{C}$ of $\pi$ which yields an object of $\Picabs_{g,n+1,d}$,
    \begin{equation}
    \label{ff44993}    
    \pi:\mathcal{C} \rightarrow \mathcal{S}\, , \ \ \ \ p_1,\ldots,p_n,p_{n+1}: \mathcal{S} \rightarrow \mathcal{C}\, , \ \ \ \
        \mathcal{L} \rightarrow \mathcal{C}\, .
        \end{equation}
 Let $A_0 \in \mathbb{Z}^{n+1}$ be the vector obtained by appending $0$ (as the last coefficient) to $A$.
 Let $\mathsf{DR}^{\mathsf{op}}_{g,A_0, \mathcal{L}}\in \mathsf{CH}^g_{\mathsf{op}}(\mathcal{S})$ be the twisted double ramification cycle associated to the new family \eqref{ff44993} and the vector $A_0$. We have the invariance
        $$\mathsf{DR}^{\mathsf{op}}_{g,A_0, \mathcal{L}}
        =F^*\mathsf{DR}^{\mathsf{op}}_{g,A, \mathcal{L}}
        \, , $$
    where $F:\Picabs_{g,n+1,d} \rightarrow \Picabs_{g,n,d}$
    is the map forgetting the last marking.
 
 \vspace{10pt}
        \noindent{\bf{Invariance \invnrA: {\rm Weight translation}.}}
\vspace{10pt}

Let $B=(b_1,\ldots,b_n) \in \bb Z^n$ satisfy $\sum_{i=1}^n b_i=e$, then the family 
    \begin{equation}
    \label{ff44994}    
    \pi:\mathcal{C} \rightarrow \mathcal{S}\, , \ \ \ \ p_1,\ldots,p_n: \mathcal{S} \rightarrow \mathcal{C}\, , \ \ \ \
        \mathcal{L}\big(\sum_{i=1}^n b_i p_i\big) \rightarrow \mathcal{C}\, .
        \end{equation}
 defines an object of $\Picabs_{g,n,d+e}$. 
 Let $\mathsf{DR}^{\mathsf{op}}_{g,A+B, \mathcal{L}(\sum_i b_i p_i)}\in \mathsf{CH}^g_{\mathsf{op}}(\mathcal{S})$ be the twisted double ramification cycle associated to the new family \eqref{ff44994} and the vector $A+B$. We have the invariance
        $$\mathsf{DR}^{\mathsf{op}}_{g,A+B, \mathcal{L}(\sum_i b_i p_i)}
        =\mathsf{DR}^{\mathsf{op}}_{g,A, \mathcal{L}}
        \, . $$
    \vspace{10pt}   

        \noindent{\bf{Invariance \invnrII: {\rm Twisting by pullback.}}}
\vspace{10pt}
    
    Let $\mathcal{B} \rightarrow \mathcal{S}$ be any line bundle on the
    base. By tensoring \eqref{ff449911} with $\pi^*\mathcal{B}$, we obtain a new object of $\Picabs_{g,n,d}$ over $\mathcal{S}$:
    \begin{equation}
    \label{ff4499}    
    \pi:\mathcal{C} \rightarrow \mathcal{S}\, , \ \ \ \ p_1,\ldots,p_n: \mathcal{S} \rightarrow \mathcal{C}\, , \ \ \ \
        \mathcal{L}\otimes \pi^*\mathcal{B} \rightarrow \mathcal{C}\, .
        \end{equation}
        Let $\mathsf{DR}^{\mathsf{op}}_{g,A, \mathcal{L}\otimes\pi^*\mathcal{B}}\in \mathsf{CH}^g_{\mathsf{op}}(\mathcal{S})$ be the twisted double ramification cycle associated to the new family \eqref{ff4499} and the vector $A$. We have the invariance
        $$\mathsf{DR}^{\mathsf{op}}_{g,A, \mathcal{L}\otimes\pi^*\mathcal{B}}
        =\mathsf{DR}^{\mathsf{op}}_{g,A, \mathcal{L}}
        \, . $$
    
 \vspace{10pt}
        \noindent{\bf{Invariance \invnrIII: {\rm Vertical twisting}.}}
\vspace{10pt}
    
    Consider a partition of the genus, marking, and degree data,
    \begin{equation}\label{ss445}
    g_1+g_2=g\, , \ \ \ N_1\sqcup N_2 = \{1,\ldots,n\}\,, \ \ \ d_1+d_2=d\, ,
    \end{equation}
    which is {\em not} symmetric.{\footnote{We require
    $(g_1,N_1,d_1)\neq (g_2,N_2,d_2)$ so that the two sides of a separating node
    with separating data \eqref{ss445} can be distinguished.}}
    Such a partition defines a divisor 
    $$ \Delta_1 \in \mathsf{CH}^1(\mathfrak{C}_{g,n,d})$$
    in the universal curve over $\Picabs_{g,n,d}$
    by twisting by the $(g_1,N_1,d_1)$-component of a 
    curve with a separating node with separating data \eqref{ss445}.

    By tensoring \eqref{ff449911} with $\mathcal{O}_{\mathcal{C}}(\Delta_1)$, we obtain a new object of $\Picabs_{g,n,d}$ over $\mathcal{S}$:
    \begin{equation}
    \label{ff44999}    
    \pi:\mathcal{C} \rightarrow \mathcal{S}\, , \ \ \ \ p_1,\ldots,p_n: \mathcal{S} \rightarrow \mathcal{C}\, , \ \ \ \
        \mathcal{L}(\Delta_1) \rightarrow \mathcal{C}\, .
        \end{equation}
        Let $\mathsf{DR}^{\mathsf{op}}_{g,A, \mathcal{L}(\Delta_1)}\in \mathsf{CH}^g_{\mathsf{op}}(\mathcal{S})$ be the twisted double ramification cycle associated to the new family \eqref{ff44999} and the vector $A$. We have the invariance
        \begin{equation}\label{k55k9}
        \mathsf{DR}^{\mathsf{op}}_{g,A, \mathcal{L}(\Delta_1)} =
        \mathsf{DR}^{\mathsf{op}}_{g,A, \mathcal{L}}
        \, . \end{equation}
        
    For \emph{symmetric} separating data \eqref{ss445}, equality \eqref{k55k9}
    holds with $\Delta_1\subset \mathfrak{C}_{g,n,d}$ defined as the full preimage of the locus $\Delta \subset \Picabs_{g,n,d}$ of curves with a separating node \eqref{ss445}. Then, equality \eqref{k55k9}
    follows from Invariance \invnrII{} with $\mathcal{B}=\mathcal{O}(\Delta)$.
    
\vspace{10pt}
        \noindent{\bf{Invariance \invnrIV: {\rm Partial stabilization}.}}
\vspace{10pt}
    
    Consider a second family 
    of prestable
    $n$-pointed genus $g$ curves over $\mathcal{S}$,
    \begin{equation*}
    \pi':\mathcal{C}' \rightarrow \mathcal{S}\, , \ \ \ \ p'_1,\ldots,p'_n: \mathcal{S}
    \rightarrow \mathcal{C}'\, ,\end{equation*}
    together with a birational $\mathcal{S}$-morphism 
    $$f: \mathcal{C}' \rightarrow \mathcal{C}\, , \ \ \ \ f\circ p'_i = p_i\, .$$
    A line bundle of relative degree $d$ is defined on $\mathcal{C}'$ by
    $$f^*\mathcal{L} \rightarrow \mathcal{C}'\, .$$
    We require the following property to hold: {\em if the section $p'_i$ meets the
    exceptional locus of $f$, then $a_i=0$}. 
        
        Let $\mathsf{DR}^{\mathsf{op}}_{g,A, f^*\mathcal{L}}\in \mathsf{CH}^g_{\mathsf{op}}(\mathcal{S})$ be the twisted double ramification cycle associated to the new family
    \begin{equation} \label{eqn:primefamily}    \pi':\mathcal{C}' \rightarrow \mathcal{S}\, , \ \ \ \ p'_1,\ldots,p'_n: \mathcal{S}
    \rightarrow \mathcal{C}'\, , \ \ \ \ f^*\mathcal{L} \rightarrow \mathcal{C}' \end{equation}
    and the vector $A$.
        We have the invariance
        $$\mathsf{DR}^{\mathsf{op}}_{g,A, f^*\mathcal{L}}=
        \mathsf{DR}^{\mathsf{op}}_{g,A, \mathcal{L}}\, . $$

\vspace{10pt}
Theorem \ref{Thm:main} provides two paths to viewing 
the  above  invariance properties. The invariances can be 
seen
either from
formal
properties of the
geometric construction of the universal twisted double ramification
cycle or from formal symmetries of
Pixton's formula. In fact, all invariances hold not only for the codimension $g$ part $\P^g_{g,A,d}$ which computes the double ramification cycle, but for the full mixed-degree class $\P_{g,A,d}^\bullet$.

For example, Invariance \invnrIV{} on the 
formula side says that for the maps 
\[\phi_{\ca L}, \phi_{f^* \ca L} : \ca S \to \Picabs_{g,n,d}\]
obtained from the families  \eqref{ff449911} and \eqref{eqn:primefamily}, we have  
\[\phi_{f^* \ca L}^* \P^g_{g,A,d}=
\phi_{\ca L}^* \P^g_{g,A,d}  \in \mathsf{CH}^g_{\mathsf{op}}(\mathcal{S})\,  \] 
 for
$\P^g_{g,A,d} \in \CHop^g(\Picabs_{g,n,d})$.

Proofs of all of the invariances
will be presented in Section \ref{Sect:proofInvariance}.
The above invariances (together with geometric definitions
when transversality to the Abel-Jacobi map holds)  do not
characterize{\footnote {Further geometric properties are
required, see Section 1.6 of \cite{HolSch} for a discussion.}} 
$\DRop$. 

\subsection{Universal formula in degree 0}\label{sec:intro_deg_0}
The most efficient statement of the double ramification cycle formula
on the Picard stack of curves occurs in the degree $d=0$ case with {\em no} markings. 
In order to avoid{\footnote{For $g=1$, a parallel formula holds for $n=1$ and $A=(0)$.}} the unpointed genus 1 case, let $g\neq 1$.

The specialization of Theorem \ref{Thm:main} to $d=0$ calculates $\DRop_{g,\emptyset}$ as
the value at $r=0$ of the degree $g$ part of
\begin{align*} 
\exp\left(-\frac12 \eta \right) \sum_{
\substack{\Gamma_\delta\in \G_{g,0,0} \\
w\in \mathsf{W}_{\Gamma_\delta,r}}
}
\frac{r^{-h^1(\Gamma_\delta)}}{|\Aut(\Gamma_\delta)| }
\;
j_{\Gamma_\delta*}\Bigg[&
\prod_{e=(h,h')\in \E(\Gamma_\delta)}
\frac{1-\exp\left(-\frac{w(h)w(h')}2(\psi_h+\psi_{h'})\right)}{\psi_h + \psi_{h'}} \Bigg]\,  
\end{align*}
as an operational Chow class on 
$$\Picabs_g=\Picabs_{g,0,0}\, .$$
The full
statement of Theorem \ref{Thm:main} can 
be recovered from the above $d=0$ specialization via pullback under the map
\[\Picabs_{g,n,d} \to \Picabs_{g}\, ,\ \ \  (C,p_1, \ldots, p_n, \ca L) \mapsto \left(C, \ca L\left(-\sum_{i=1}^n a_i p_i\right)\right).\]
Indeed, the above map is the composition of the morphism 
\[\tau_{-A} : \Picabs_{g,n,d} \to \Picabs_{g,n,0}, (C,p_1, \ldots, p_n, \ca L) \mapsto \left(C, p_1, \ldots, p_n, \ca L\left(-\sum_{i=1}^n a_i p_i\right)\right)\]
with the morphism $$F : \Picabs_{g,n,0} \to \Picabs_{g}$$ forgetting the markings $p_1, \ldots, p_n$. 

\vspace{5pt}
\noindent $\bullet$
For the $\DRop_{g,A}$ side of Theorem \ref{Thm:main}, 
Invariance \invnrn{} 
implies $$F^* \DRop_{g,\emptyset} = \DRop_{g,\bd 0}$$ for the zero vector $\bd 0 \in \mathbb{Z}^n$. Furthermore, Invariance \invnrA{} implies $$\tau_{-A}^*\DRop_{g,\bd 0} = \DRop_{g,A}\, .$$ 

\vspace{5pt}
\noindent $\bullet$ For the $\mathsf{P}^g_{g,A,d}$
side of Theorem \ref{Thm:main}, 
the corresponding invariance properties of Pixton's formula
(discussed in Section \ref{Sect:proofInvariance}) 
yield the parallel transformation 
$$\tau_{-A}^*F^*\mathsf{P}^g_{g,\emptyset,0} = \mathsf{P}^g_{g,A,d}\, .$$ 

\vspace{5pt}
\noindent Therefore, the equality in Theorem \ref{Thm:main} for general $A$ and $d$ follows from the specialization to $A=\emptyset$ and $d=0$.
For certain steps in our proof of Theorem \ref{Thm:main}, the 
$A=\emptyset$ and $d=0$
geometry is
advantageous and is used.


By restricting to suitable open subsets of $\Picabs_{g}$, we can simplify the $d=0$ formula even further. Let $$\Picabs_{g}^\textup{ct} \subset \Picabs_{g}$$
be the locus where the curve $C$ is of compact type. 
We obtain 
\begin{equation}\label{ccttt}
 \DRop_{g,\emptyset}|_{\Picabs_{g}^\textup{ct}} 
 = \frac{\theta^g}{g!} \,, \ \ \  
 \text{for}\ \ \theta = -\frac12 \left(\eta +  \sum_{\Delta} d_\Delta^2 [\Delta] \right),
\end{equation}
where the sum is over the boundary divisors $\Delta \subset \Picabs_g$ on which generically the curve splits into two components carrying line bundles of degrees $d_\Delta$ and $-d_\Delta$. 
The class $\theta$ here may be viewed as a universal theta divisor on $\Picabs_{g}^\textup{ct}$. 

Formula \eqref{ccttt} was first written on the moduli space of stable curves of compact type in 
\cite{Grushevsky2012The-double-rami,Hain2013Normal-function}. The operational Chow class  $\DRop_{g,\emptyset}$ on $\Picabs_g$,
however, is {\em not} the  power of a divisor. 


\subsection{Acknowledgements}
We thank D. Chen, A. Chiodo, F. Janda, G. Farkas, T. Graber, S. Grushevsky, A. Kresch,
S. Molcho, M. M\"oller, A. Pixton, D. Ranganathan, A. Sauvaget, H.-H. Tseng,
J. Wise, and D. Zvonkine
for  
many discussions about Abel-Jacobi theory, double ramification cycles,
and meromorphic differentials. The AIM workshop on {\em Double ramification cycles and integrable systems} 
played a key role at the start of the paper. We are grateful to the organisers A. Buryak, R. Cavalieri, E. Clader, and P. Rossi. We thank the referees for several 
improvements in the presentation.

Y.~B. was supported by ERC-2017-AdG-786580-MACI and the
Korea Foundation for Advanced Studies.
D.~H. was partially supported by NWO grant 613.009.103. 
R.~P. was partially supported by 
SNF-200020-182181, SwissMAP, and
the Einstein Stiftung. 
J.~S. was supported by the SNF Early Postdoc.Mobility grant 184245 and thanks
the Max Planck Institute for Mathematics in Bonn for its hospitality. 
R.~S. was supported by  NWO grant 613.009.113.

The project has received funding from the European Research
Council (ERC) under the European Union Horizon 2020 research and
innovation program (grant agreement No. 786580).

\section{Notation, conventions, and the plan}

\subsection{Ground field}

In the Introduction, the ground field
was the complex numbers $\bb C$. However, for the remainder of the paper,
we will work more generally over a field ${\field}$ of characteristic zero. 
We will make essential use of the results of \cite{Janda2018Double-ramifica} which are stated over $\bb C$, 
but also
hold over $\bb Q$ by the following
standard argument:
\begin{enumerate}
\item[(i)] Both the $\mathsf{DR}$ cycle  (via the b-Chow approach \cite{Holmes2017Extending-the-d}) and Pixton's class are defined over $\bb Q$. 
\item[(ii)] Rational equivalence of cycles uses finitely many subschemes and rational functions and  hence descends to a finitely generated 
$\bb Q$-subalgebra of $\bb C$. A non-empty scheme of finite type over $\bb Q$ has points over some finite extension, hence the rational equivalence descends to a finite extension of $\bb Q$. 
\item[(iii)]
The rational equivalence descends (via a Galois argument)
further to $\bb Q$ since we work in Chow with rational coefficients. 
\end{enumerate}
By similar arguments,
our results are in fact true over $\bb Z[1/N]$ for a positive integer $N$ depending on the ramification data. Understanding what happens at small primes (or integral Chow groups)
is an interesting question.

\subsection{Basics}\label{sec:prestable_vs_log_curves}

Let ${\field}$ be a ground field of characteristic zero. When we work in the logarithmic category, we assume $\on{Spec} {\field}$ to be equipped with the trivial log structure. 

We write $\overline{\mathcal M}$ for the stack of all stable (ordered) marked curves over ${\field}$ 
and $\mathfrak M$ for the stack of all prestable curves with ordered marked points. Both come with natural log structures, and the universal marked curves over these spaces are naturally log curves.
For $\Mbar$, the log structure is described in \cite{Kato2000Log-smooth-defo}. 
The same construction applies unchanged to $\frak M$, see
\cite[Appendix A]{Gross2013Logarithmic-Gro}.
The natural open immersion $$\overline{\mathcal M} \to \frak M$$ is strict (though, in contrast to  \cite{Gross2013Logarithmic-Gro,Kato2000Log-smooth-defo},
we order the markings of our log curves). We use subscripts $g$ and $n$ to fix the genus and number of markings when necessary. 

Let $\frak C$ be
the universal prestable curve over $\frak M$. 
For efficiency of notation, we will also denote by $\frak C$ the universal curve over the various other moduli stacks
of curves with additional structure which
will appear in the paper. 
These universal curves are always obtained by 
pulling-back $\frak C$ over $\frak M$.

For the convenience of the reader, we provide here a table of the key symbols.

{\renewcommand{\arraystretch}{1.5}
\begin{table}[htp]
\caption{Key notation}
\begin{center}
\begin{tabular}{|c|l|}
\hline
$\frak M_{g,n}$ & stack of prestable marked curves of genus $g$ with $n$ markings\\
$\Mbar_{g, n}$ & stack of stable marked curves of genus $g$ with $n$ markings\\
$\frak C$ & universal prestable curve over $\frak M$ \\
$A = (a_1, \dots, a_n)$ & $A\in \bb Z^n$ with $\sum_{i=1}^n a_i$  denoted by $d$\\
$\Picabs_{g,n}$ &  Picard stack,  Section \ref{sec:pic_stacks}  \\ 
$\Picrel_{g,n}$ & relative universal Picard stack,  Section \ref{sec:pic_stacks} \\
$\CHop$ & operational Chow group, Section \ref{ssec:Chopstacks}    \\
$\xi_i$ &  tautological class on $\Picabs_{g,n}$, Section \ref{Ssec:PsiXiEta}\\
$\eta_{a,b}$ &  tautological class on $\Picabs_{g,n}$, Section \ref{Ssec:PsiXiEta}\\
$\psi_i$ &  tautological class on $\Picabs_{g,n}$, Section \ref{Ssec:PsiXiEta}\\
$\mathsf{P}_{g,A,d}^{c}$ & Pixton's cycle, Section \ref{Ssec:MainFormula}\\
$c(\phi)$& homomorphisms $c(\phi): \Chow_*(B) \to \Chow_{*-p}(B)$ given by an operational \\
~& class $c \in \CHop^p(\frak X)$ and $\phi: B \to \frak X$ with $B$ finite type scheme, \cref{ssec:Chopstacks}\\
$\DRop_{g,A}$&  operational DR cycle, Section \ref{sec:uni_DR_defs}\\
\hline
\end{tabular}
\end{center}
\label{default}
\end{table}%
}


\subsection{Plan of the paper}

Section \ref{sec:pic_and_chow}  concerns
 the treatment
of several technical issues related to operational Chow groups
of 
the Picard stack $\Picabs$.
In fact,
we develop a general theory of operational Chow groups of algebraic stacks which are
locally of finite type over ${\field}$. 
The theory is certainly known to experts, but for our later results, we will require
the precise definitions. In particular, to a proper representable morphism of algebraic stacks, we associate an operational class, which will be the key to constructing the operational double ramification cycle. 

The core  of the
paper starts in Section \ref{sec:uni_DR_defs} where
we give three equivalent definitions of the universal double ramification cycle on $\Picabs$. Our first definition is by taking a closure in the spirit of \cite{Holmes2017Jacobian-extens}
which is simple, but rather difficult to work with. The second is via logarithmic geometry following \cite{Marcus2017Logarithmic-com}. The third is a  b-Chow definition along the path of
\cite{Holmes2017Multiplicativit}. 
In  Section \ref{sec:image_of_aj}, we give an explicit description of the 
set-theoretic image of the double ramification cycle in $\Picabs$.
We prove Theorem \ref{cccc} in Section \ref{proof1111}.

In Section \ref{sec:univeral_pixton}, we discuss properties 
of Pixton's  cycle $\P_{g,A,d}^{c}$ defined  Section \ref{Ssec:MainFormula}. In particular, 
formal properties of Pixton's cycle parallel to the invariances of the double ramification cycles
are proven. Compatibilities with definitions in previously studied
cases are also proven.

\Cref{sec:proof_of_main_theorem} contains the proof of Theorem \ref{Thm:main}, the main
result of the paper, by an eventual
reduction to the formula of \cite{Janda2018Double-ramifica} in the case of target $\bb P^n$ for large $n$.  
A crucial step in the proof is the matching of the double ramification cycle defined in \cite{Janda2018Double-ramifica} via rubber maps with
 our new universal definition on $\Picabs$ in a suitable sense when the target is $\bb P^n$. 
 The matching is verified in Section \ref{sec:comparing_stacks} where we follow the pattern of the proof given in \cite{Marcus2017Logarithmic-com} in case the target is a point. 

In Section \ref{Sect:proofInvariance}, we prove the  invariance properties of
Section \ref{Sect:introinvariance}. Theorems \ref{Thm:mainvan} and \ref{FPA1} are
proven as a consequence of Theorem \ref{Thm:main} in Section
 \ref{sec:applications}. The connections between the vanishing result of
 Theorem \ref{Thm:mainvan} and past (and future) work is discussed in Section \ref{pasfut}.
 

\section{Picard stacks and operational Chow}\label{sec:pic_and_chow}

\subsection{The Picard stack and relative Picard space} \label{sec:pic_stacks}
Our stacks will be with respect to the fppf topology (\cite[Definition 9.1.1]{OlssonAlgebraic-spaces}). 
We define the Picard stack $\Picabs_{g,n}$ as the fibred category over $\frak M_{g,n}$ whose fibre over a scheme $T \to \frak M_{g,n}$ is the groupoid of line bundles on $\frak C_{g,n} \times_{\frak M_{g,n}} T$ with morphisms given by isomorphisms of line bundles, see \cite[14.4.7]{Laumon2000Champs-algebriq}. We define the relative Picard space $\Picrel_{g,n}/\frak M_{g,n}$ to be the quotient of $\Picabs_{g,n}$ by its relative inertia over $\frak M_{g,n}$. Equivalently $\Picrel_{g,n}$ is the fppf-sheafification of the fibred category of \emph{isomorphism classes} of line bundles on $\frak C_{g,n} \times_{\frak M_{g,n}} T$,
see \cite[Chapter 8]{Bosch1990Neron-models} and \cite[Chapter 9]{Fantechi2005Fundamental-alg}. 

Relative representability of $\Picrel_{g,n}/\frak M_{g, n}$ by smooth algebraic spaces can be checked locally on $\frak M_{g,n}$. It then follows from \cite[Appendix]{Artin1974Versal-deform} as the curve $$\frak C_{g,n} \to \frak M_{g,n}$$
is flat, proper, relatively representable by algebraic spaces, and cohomologically flat in dimension 0 (reduced and connected geometric fibers). The Picard stack $\Picabs_{g,n}$ is a $\bb G_m$-gerbe over $\Picrel_{g,n}$, hence is a (smooth) algebraic stack. In particular, $\Picabs_{g,n}$ is smooth over $\field$ of pure dimension $4g-4 + n$, and $\Picrel_{g,n}$ is smooth over $\field$ of pure dimension $4g - 3 + n$. 

\begin{remark} \label{Rmk:genus1n0}
We will moreover assume  $(g,n) \neq (1,0)$. Then, $\frak M_{g,n}$ and hence $\Picrel$, $\Picabs$, and anything of Deligne-Mumford type over them, has affine stabilisers, and so is therefore stratified by global quotient stacks in the sense of \cite{Kresch1999Cycle-groups-fo}. The latter property 
will be important for some intersection-theoretic computations, in particular the proof of Proposition \ref{prop:invarianceproperbirat}. 
\end{remark}

\subsection{Operational Chow groups of algebraic stacks}\label{ssec:Chopstacks}
Our goal here is to define the operational Chow group of $\Picabs_{g,n}$, following \cite[Chapter 17]{Fulton1984Intersection-th}. 
In fact, we construct an operational Chow group for any algebraic stack locally of finite type over a field. 
The definition is a simple generalisation of  \cite{Fulton1984Intersection-th}. 

\begin{definition}
    Let $\frak Y$ be a locally finite type algebraic stack over ${\field}$. Let $p$ be an integer. A bivariant class $c$ in the $p$th operational Chow group $\CHop^p(\frak Y)$ is a collection of  homomorphisms
\[c(\phi)^{m}\colon \Chow_m(B) \to \Chow_{m-p}(B)\]
for all maps $\phi \colon B \to \frak Y$ where $B$ is a scheme of finite type over ${\field}$, and for all integers $m$, compatible with proper pushforward, flat pullback, and Gysin homomorphisms for regular embeddings (conditions (C1)-(C3) in \cite[Section 17.1]{Fulton1984Intersection-th}). 
\end{definition}
For a class $\alpha \in \Chow_m(B)$, we will sometimes write $c(\alpha)$ in place of $c(\phi)^m(\alpha)$, if the morphism $\phi$ is clear. 


Such a definition for the operational Chow group of a Deligne-Mumford stack is given in \cite{Edidin2017Towards-an-inter}. To be able to use Chow groups on algebraic stacks as defined in \cite{Kresch1999Cycle-groups-fo} for algebraic stacks of finite type over a field, we will use the following result. 
\begin{lemma}\label{lem:factorisationopenquasi}
Let $f\colon \frak A \to \frak B$ be a morphism over a field ${\field}$ where $\frak B$ is an algebraic stack 
locally of finite type over ${\field}$ and $\frak A$ is an algebraic stack of finite type over ${\field}$. Then there exists a factorisation of $f$ via a commutative diagram

\begin{center}
\begin{tikzcd}
\frak A \arrow[dr] \arrow[rr, "f"]
&~
& \frak B \\
~
&\frak B' \arrow[ur]
&~
\end{tikzcd}
\end{center}
where $\frak B' \to \frak B$ is an open immersion and $\frak B'$ is quasi-compact (and hence of finite type).  
\end{lemma}

\begin{proof}
We cover $\frak B$ by affine flat finite presentation morphisms $\{V_i \to \frak B\}_{i \in I}$.  
Let $U_i$ be the image of the $V_i$ in $\frak B$. The $U_i$ are open, and the $f^{-1}(U_i)$ cover $\frak A$. As $\frak A$ is quasi-compact, there is a finite subset $J \subset I$ such that $\{f^{-1}(U_i)\}_{i \in J}$ covers $\frak A$. Then $\frak B' = \cup_{i \in J} U_i$ 
defines the required factorisation.
\end{proof}

For representable morphisms (representable by algebraic spaces) $f\colon \frak X \to \frak Y$, we can also define the operational Chow group $\CHop^p(\frak X \to \frak Y)$ as a collection of morphisms 
\[c(\phi)^{m}\colon \Chow_{m}(B) \to  \Chow_{m-p}(B \times_{\frak Y} \frak X)\]
for all maps $\phi \colon B \to \frak Y$ where $B$ is a scheme of finite type over ${\field}$, and for all integers $m$, compatible with proper pushforward, flat pullback, and Gysin homomorphisms. We have
$$\CHop^p(\frak X) = \CHop^p(\on{id}\colon \frak X \to \frak X)\,. $$

We have products, pullbacks, and proper representable pushforwards on these operational Chow groups of algebraic stacks as described in \cite[chapter 17.2]{Fulton1984Intersection-th} satisfying the properties described there.

\vspace{8pt}                              
\begin{remark}\label{remark:opclassactonalgspaces}         
Even for $B$ a scheme and $\pi: \frak X \to \frak Y$ representable, the fibre product $B \times_{\frak Y}\frak X$ can be an algebraic space. Therefore, some care is needed when generalizing classical constructions such as the product
\[\CHop^{a}(\frak X) \times \CHop^{b}(\pi\colon \frak X \to \frak Y) \to \CHop^{a+b}(\pi\colon \frak X \to \frak Y).\]
Indeed, for $c \in \CHop^{a}(\frak X)$ and  $d \in \CHop^{b}(\pi\colon \frak X \to \frak Y)$  we want to define $$c \cdot d \in \CHop^{a+b}(\pi\colon \frak X \to \frak Y)\, .$$
For a map $\phi: B \to \frak Y$ with $B$ a finite type scheme fitting in a pullback diagram
\[
\begin{tikzcd}
B \times_{\frak Y}\arrow[d] \arrow[r,"\phi'"] \frak X  & \frak X \arrow[d,"\pi"]\\
B \arrow[r,"\phi"] & \frak Y
\end{tikzcd}
\]
we want to define the induced map
\[(c \cdot d)(\phi)^m : \Chow_m(B) \to \Chow_{m-a-b}(B \times_{\frak Y} \frak X)\]
as the composition
\[\Chow_m(B) \xrightarrow{d(\phi)^m} \Chow_{m-b}(B \times_{\frak Y} \frak X) \xrightarrow{c(\phi')^{m-b}} \Chow_{m-a-b}(B \times_{\frak Y} \frak X)\, .\]
But, a priori, the map $c(\phi')^{m-b}$ does not make sense since the domain $B \times_{\frak Y} \frak X$ of $\phi'$ is an algebraic space.
However, given a collection 
$$c= c(\phi')^{n}\colon \Chow_{n}(B') \to  \Chow_{n-a}(B') $$ for all finite type schemes $B'$ with $\phi'\colon B' \to \frak X$, we can construct  a collection of maps $${c}(\phi')^n\colon \Chow_{n}(B') \to \Chow_{n-a}(B')$$ for $\phi'\colon B' \to \frak X$ with $B'$ a finite type algebraic space via \cite[Section 5.1]{Vistoli1989Intersection-th}. Indeed, for each integral closed substack $Z \subset B'$, \cite[Section 5.1]{Vistoli1989Intersection-th} defines an action of $c$ on $[Z]$ which is independent of the chosen cover of the algebraic space by a scheme. This action induces a map 
$\Chow_{n}(B') \to \Chow_{n-a}(B')$
which commutes with proper pushforward and flat morphisms via \cite[Lemma 5.3]{Vistoli1989Intersection-th} and is compatible with Gysin homomorphisms. Applying this to $$\phi' : B' = B \times_{\frak Y} \frak X \to \frak X$$
with $n=m-b$ gives the desired map ${c}(\phi')^{m-b}$.
\end{remark}
\vspace{8pt}

For a proper representable morphism $\pi\colon \frak X \to \frak Y$ which is flat of relative dimension $q$, the pushforward 
\[\CHop^*(\pi\colon \frak X \to \frak Y) \to \CHop^*(\frak Y)\]
can be extended to a pushforward
$$\CHop^{p}(\frak X)\to \CHop^{p-q} (\frak Y) $$
as follows. Because $\pi$ is flat, the pullback $\pi^*$ gives a natural element in $$\CHop^{-q}(\pi\colon \frak X \to \frak Y)$$ 
and then we can compose
\[\CHop^{p}(\frak X)\to \CHop^{p}(\frak X) \times \CHop^{-q}(\pi\colon \frak X \to \frak Y) \]
given by $c \mapsto  (c,\pi^*)$ with the product and the pushforward maps 
\[\CHop^{p}(\frak X) \times \CHop^{-q}(\pi\colon \frak X \to \frak Y) \to \CHop^{p-q}(\pi\colon \frak X \to \frak Y) \to \CHop^{p-q}(\frak Y),\]
yielding the desired pushforward map $\pi_*\colon \CHop^{p}(\frak X)\to \CHop^{p-q} (\frak Y)$. This may for example be applied to the universal curve $\pi\colon \frak C_{g,n} \to \Picabs_{g,n}$. A similar construction also works for $\pi$ proper representable and lci. This pushforward map commutes with pullback of operational classes.

\subsection{Relationship to usual Chow groups}

Let $\frak Y$ be a locally finite type algebraic stack over ${\field}$, and $(\frak U_i)_{i \in \bb N}$ an increasing sequence of finite type open substacks of $\mathfrak Y$ with 
$$\bigcup_i \frak U_i =  \frak Y\, .$$
In particular, for a finite type scheme $B/{\field}$, every
map $B \to \frak Y$ factors via some $\frak U_i$. We have pullback maps 
\begin{equation}
\CHop^*(\frak Y) \to \CHop^*(\frak U_i)\, 
\end{equation}
which induce a map 
\begin{equation}\label{eq:chop_lim}
\Phi\colon \CHop^*(\frak Y) \to \lim_i \CHop^*(\frak U_i)
\end{equation}
to the inverse limit of the $\CHop^*(\frak U_i)$, with transition maps given by pullback of operational classes along open immersions. 

\begin{lemma}\label{lem:chop_and_limits}
The map $\Phi\colon \CHop^*(\frak Y) \to \lim_i \CHop^*(\frak U_i)$ is an isomorphism of abelian groups. 
\end{lemma}
\begin{proof}
We first show injectivity. Let $c$ in $\CHop^*(\frak Y)$ with $\Phi(c) = 0$. For every $B \to \frak Y$ with $B/{\field}$ of
finite type, we get a map $$c(B/\frak Y):\Chow_*(B) \to \Chow_*(B)\, . $$
There exists $i$ such that the map $B \to \frak Y$ factors via $\frak U_i$. Then $\Phi(c) = 0$ implies that
$$c(B/\frak U_i):\Chow_*(B) \to \Chow_*(B)$$ is the zero map. By definition of the pullback, $c(B/\frak Y) = c(B/\frak U_i)$.

Next we show surjectivity. Suppose we have a compatible collection 
$$c_i \in \CHop^*(\frak U_i)\, .$$
We will build $c \in \CHop^*(\frak Y)$ as follows.
Let $B \to  \frak Y$ with $B/{\field}$ of finite type. There exists $N$ such that for all $i\ge N$,
the map $B \to \frak Y$ factors via $\frak U_i$. Then for all $i \ge N$,
we have maps $$c(B/\frak U_i):\Chow_*(B) \to \Chow_*(B)\, ,$$
and the compatibly means $c_i(B/\frak U_i) = c_j(B/\frak U_i)$ for all $i$, $j \ge N$. 
We define $c = c_N$, which clearly is sent by $\Phi$ to
the $c_i$.

To conclude, we must check that $c$ satisfies the axioms of an operational class. This follows easily from the fact that each $B \to \frak Y$ factors via some $\frak U_i$. 
\end{proof}

\begin{lemma}[Proposition 5.6 of \cite{Vistoli1989Intersection-th}]\label{lem:op_eq_ch_for_smooth}
Let $\frak Y$ be a smooth finite type Deligne-Mumford stack over ${\field}$ of pure dimension $d$, and 
let $\iota\colon \frak Y \to \frak Y$ be the identity. Then for $m \geq 0$ the map 
\begin{equation}\label{eq_chop_to_chow}
\CHop^m(\frak Y) \to \Chow_{d-m}(\frak Y)\,, \ \ \alpha \mapsto \alpha(\iota)([\frak Y])
\end{equation}
is an isomorphism.
\end{lemma}

Combining \cref{lem:chop_and_limits,lem:op_eq_ch_for_smooth}, we immediately obtain the following
result.
\begin{corollary}
Let $\frak Y$ be a smooth Deligne-Mumford stack over ${\field}$ of pure dimension $d$, and let $\frak U_i$ be a sequence of finite type open substacks with $\bigcup_I \frak U_i = \frak Y$. Then the natural map 
\begin{equation}
\Phi\colon \CHop^m(\frak Y) \to \lim_I \Chow_{d-m}(\frak U_i)
\end{equation}
obtained by combining \eqref{eq:chop_lim} and \eqref{eq_chop_to_chow} is an isomorphism. 
\end{corollary}

As a final remark{\footnote{We thank A. Kresch for related
discussion.}}, we note that there exists a map of Chow groups in the opposite direction of \eqref{eq_chop_to_chow} in greater generality. Let $\frak Y$ be a smooth algebraic stack of finite type over $\field$ and of pure dimension $d$
which has a stratification\footnote{See \cite{Kresch1999Cycle-groups-fo} for the precise definition. The property is always satisfied for Deligne-Mumford stacks.}
by quotient stacks. Then, there exists a map
\[\Psi: \Chow_{d-m}(\frak Y) \to \CHop^m(\frak Y)\]
from the Chow group $\Chow_*(\frak Y)$ constructed in \cite{Kresch1999Cycle-groups-fo} to the operational Chow group of $\frak Y$ 
defined as follows. Given $\phi : B \to \frak Y$ with $B$ a finite type scheme, let $$\phi_B : B \to B \times \frak Y$$ be 
the diagonal morphism. Since $\frak Y$ is smooth, $\phi_B$
is representable and  is a local complete intersection of codimension $d$. For $\beta \in \Chow_{d-m}(\frak Y)$, we define
\begin{equation*}
    \Psi(\beta)(\phi) : \Chow_*(B) \to \Chow_{*-m}(B)\, \ \
    \alpha \mapsto \phi_B^!(\alpha \times \beta)\, ,
\end{equation*}
where $\alpha \times \beta \in \Chow_{*+d-m}(B \times \frak Y)$ is the exterior product of $\alpha$ and $\beta$ as defined in \cite[Section 3.2]{Kresch1999Cycle-groups-fo}. The collection of maps $\Psi(\beta)(\phi)$ defines an element $$\Psi(\beta) \in \CHop^m(\frak Y)\, .$$
For $\frak Y$ a Deligne-Mumford stack, the map $\Psi$ is the inverse of the map \eqref{eq_chop_to_chow}. 
However, for an arbitrary smooth algebraic stack $\frak Y$,
we do {\em not} know whether  $\Psi$ is injective or surjective.

\subsection{Constructing an operational Chow class} \label{sec:constructing_chow_class}
\subsubsection{Overview}
Given a vector $A \in \bb Z^n$ of ramification data satisfying
$$\sum_{i=1}^n {a_i} =d\, ,$$
we will construct in \cref{sec:log_def} 
a stack $\cat{Div}_{g,A}$ together with a proper representable  Abel-Jacobi map 
$$\cat{Div}_{g,A} \to \Picabs_{g,n,d}\, .$$
We wish to define the twisted universal double ramification cycle $\DRop_{g,A}$ as the pushforward of the fundamental class of $\cat{Div}_{g,A}$ to $\Picabs_{g,n,d}$. 
However, two basic issues must be settled
to carry out the construction:
\begin{itemize}
\item[$\bullet$] The stack $\cat{Div}_{g,A}$ is not Deligne-Mumford and is not quasi-compact, so the existence of a well-behaved fundamental class in the operational Chow group is not clear. 
\item[$\bullet$] The proper representable pushforward of \cite[Appendix B]{BaeSchmitt1} is only defined between finite-type stacks, and so cannot be applied directly. \footnote{Chow theory of non-finite type algebraic stacks will be developed in \cite[Appendix A]{BaeSchmitt1}.}
\end{itemize}

To solve these problems, we provide here a very general construction which associates to a suitable proper morphism $$a\colon \frak X \to \frak Y$$ an operational Chow class on $\frak Y$ which plays the role of the pushforward of the fundamental class of $\frak X$. 
In \cref{sec:log_def}, we will apply the result to construct the universal twisted double ramification cycle $\DRop_{g,A}$. We also verify certain basic properties such as invariance of the class under proper birational maps which will be important in \cref{sec:comparing_stacks}. 

\subsubsection{Construction}\label{c44c}
Let $\frak X$ and $\frak Y$ be algebraic stacks locally of finite type over a field ${\field}$ and suppose we have a proper morphism $$a\colon \frak X \to \frak Y$$
of Deligne-Mumford type. 
Suppose further that $\frak Y$ is smooth of pure dimension $\dim \frak Y$ over the field, and $\frak X$ is of pure dimension $\dim \frak X$. Let 
$$e = \dim \frak Y-  \dim \frak X\, .$$
We  will construct an operational Chow class associated to $\frak X$ in $\CHop^e(\frak Y)$. 

For all finite type schemes $B$ with a morphism $\phi\colon B \to \frak Y$, and for all integers $m$, we must define maps
\[c(\phi)^{m}\colon \Chow_m(B) \to \Chow_{m-e}(B) \]
which are compatible under proper pushforward and flat pullback and satisfy commutativity (properties (C1), (C2), (C3) of \cite[Section 17.1]{Fulton1984Intersection-th}).


Let $[V] \in \Chow_m(B)$ be an irreducible cycle, and let $i_V\colon V \to B$ be the inclusion. Let $V \to B \to \frak Y$ be factored as in \cref{lem:factorisationopenquasi} into $V \to \frak Y' \to \frak Y$ where $\frak Y'$ is of finite type and $\frak Y' \to \frak Y$ is an open immersion. 

We form the diagram
\begin{equation} \label{eq:definingdiagram}
 \begin{tikzcd}
  \frak X' \times_{\frak Y'} V \arrow[r] \arrow[d, "\psi_V"] & \frak X' \times V \arrow[d, "a \times \on{id}"]  \\
  V \arrow[r, "\phi_V"] & \frak Y' \times V 
\end{tikzcd}
\end{equation}
where $\frak X'$ is the inverse image of $\frak Y'$ under $a$. Each stack in this diagram is of finite type, and therefore has a Chow group in the sense of \cite{Kresch1999Cycle-groups-fo}. Since $\frak Y$ is  smooth over ${\field}$, the map $\phi_V$ is a regular embedding of codimension $\dim \frak Y$. Also $\phi_V$ is unramified and hence a regular local embedding, so Kresch's contruction yields a map $$\phi_V^!\colon \Chow_m(\frak X' \times V) \to \Chow_{m-\dim \frak Y}(\frak X' \times_{\frak Y'} V)\, .$$ 
In particular, $[\frak X' \times V]$ is a class in dimension $\dim \frak X + m$, so the class $\phi_V^! ([\frak X \times V])$ lies in $\Chow_{m-e} (\frak X \times_{\frak Y} V)$. 
The morphism $\psi_V$ is proper and of Deligne-Mumford type, and so by \cite[Appendix B]{BaeSchmitt1} we have a pushforward ${\psi_V}_*$. \begin{definition}\label{defbivariantclass}
We define a class $a_{\on{op}}[\frak X] \in \CHop(\frak Y)$ 
via the formula
\begin{align}
    c(\phi)^{m}\colon Z_m(B) &\to \Chow_{m-e}(B) \\
    [V] &\mapsto {i_V}_* {\psi_V}_* \phi_V^! ([\frak X' \times V])\, . \nonumber
\end{align}
\end{definition}
We must verify that this construction passes to rational equivalence, is independent of the choices made, and satisfies the properties (C1), (C2), and (C3). 
After verifying in Lemma \ref{lem:classindepfact} independence on the choice of factorisation, we follow the logic in \cite{Fulton1984Intersection-th}: we verify in Lemmas \ref{lem:C1proper}, \ref{lem:C2flat}, and \ref{lem:C3regimb} of Section \ref{c45c} that the properties (C1), (C2), and (C3) hold on the level of cycles, and finally in \cref{lem:ratequiv} we use these to show that the construction passes to rational equivalence. 



\begin{lemma}\label{lem:classindepfact}
The class $c(\phi)^{m}([V])$ defined above is independent of the chosen factorisation $V \to \frak Y' \to \frak Y$. 
\end{lemma}
\begin{proof}
Let $V \to \frak Y' \to \frak Y$ and $V \to \frak Y'' \to \frak Y$ be two such factorisations. By considering $\frak Y' \cap \frak Y''$ inside $\frak Y$, we may restrict to the case where one is contained in the other. So we 
suppose there is an open immersion $$\iota\colon \frak Y'' \to \frak Y'\, .$$
Let $j\colon \frak X'' \to \frak X'$ be the induced map. Consider the diagram
\begin{equation}
\begin{tikzcd}[column sep=small, row sep = small]
 ~ & \frak X'' \times_{\frak Y''} V 	\arrow[rr] \arrow[dd, dashed, near start, "j \times \on{id}"] \arrow[dl, swap, "\psi_V''"] & ~ & \frak X'' \times V \arrow[dl, "a \times \on{id}"] \arrow[dd, "j \times \on{id}"]\\
 V \arrow[rr, near end, swap, "\phi_V''"] \arrow[dd, "\on{id}"]& ~ & \frak Y'' \times V \arrow[dd, near start, "\iota \times \on{id}"] & ~ \\
~ & \frak X' \times_{\frak Y'} V \arrow[dl, swap, "\psi_V'"] 	\arrow[rr] & ~ & \frak X' \times V \arrow[dl, "a \times \on{id}"] \\
 V  \arrow[rr, near end, swap, "\phi_V'"] &~ & \frak Y' \times V & ~ 
\end{tikzcd}
\end{equation}
where  $\frak X'' \times_{\frak Y''} V \to \frak X' \times_{\frak Y'} V$ is an isomorphism. 
We must show 
$${\psi_V'}_* {\phi_V'}^! 
([\frak X' \times V]) = {\psi_V''}_* {\phi_V''}^! ([\frak X'' \times V])\, .$$

Because $\phi_V''$ and $\phi_V'$ are both regular embeddings of the same codimension and the front square commutes, by the same proof as for \cite[Theorem 6.2(c)]{Fulton1984Intersection-th}, we obtain  ${\phi_V''}^!([\frak X '' \times V]) = {\phi_V'}^!([\frak X '' \times V])$. Therefore,
\begin{equation}\label{kk992}
{\psi_V''}_* {\phi_V''}^! ([\frak X'' \times V]) = {\psi_V''}_* {\phi_V'}^! ([\frak X'' \times V]) = {\psi_V''}_* {\phi_V'}^! ( (j \times \on{id})^*[\frak X' \times V])
\end{equation}
as $j \times \on{id}$ is an open immersion so in particular flat, and the flat pullback of the fundamental class is the fundamental class. By the compatibility of the flat pullback and Gysin maps  \cite[Section 3.1]{Kresch1999Cycle-groups-fo}, we obtain 
that \eqref{kk992} is equal to 
\[{\psi_V''}_*(j \times \on{id})^* {\phi_V'}^! ([\frak X' \times V])\, .\]
By commutativity of the left side of the cube and because of the pullback pushforward formula, we obtain  
\[{\psi_V''}_* {\phi_V''}^! ([\frak X'' \times V]) =  {\psi_V''}_*(j \times \on{id})^* {\phi_V'}^! ([\frak X' \times V]) = {\psi_V'}_* {\phi_V'}^! ([\frak X' \times V])\, \]
as required.
\end{proof}

\subsubsection{Compatibility} \label{c45c}
We will now check that the maps $c(\phi)^{m}$ 
defined in Section \ref{c44c} are
 compatible under proper pushforward and flat pullback and satisfy commutativity (properties (C1), (C2), (C3) of \cite[Section 17.1]{Fulton1984Intersection-th}).

The proper pushforward along DM-type maps between finite type algebraic stacks over a field which are stratified by global quotient stacks is defined in  \cite[Appendix B]{BaeSchmitt1}.  We cannot use the results of \cite{Kresch1999Cycle-groups-fo} for  compatibility with the proper  pushforward
since Kresch discusses only projective pushforward. 
Nevertheless, we now show that the proper pushforward for DM-type maps of finite type algebraic stacks over a field as defined in \cite[Appendix B]{BaeSchmitt1} is compatible with the Gysin maps of \cite{Kresch1999Cycle-groups-fo}.


\begin{proposition}\label{prop:Gysinproperpush}
 For a pullback diagram of algebraic stacks of finite type over $K$,
 \begin{center}
 \begin{tikzcd}
 \frak X'' \arrow[r, "i''"] \arrow[d, "q"]
 &\frak Y'' \arrow[d, "p"] \\
 \frak X' \arrow[r, "i'"] \arrow[d, "g"] 
 &\frak Y' \arrow[d,"f"] \\
 \frak X \arrow[r, "i"]
 &\, \, \frak Y\, ,
 \end{tikzcd}
 \end{center}
 where $i$ is a regular local embedding of codimension $d$, $p$ is a proper DM-type morphism, and $\frak Y'$ is stratified by global quotient stacks, we have 
 \[i^!p_*(\alpha) = q_* (i^! \alpha)\]
 for all  $\alpha \in \Chow_*(\frak Y'')$.
\end{proposition}

\begin{proof}
 A class $\alpha$ on $\frak Y''$ is represented by a projective map $z''\colon Z'' \to \frak Y''$, a vector bundle $E'' \to Z''$ and a class $[V]$
 in the naive Chow group of $E''$ represented  by $V \subset E$. 

We want to push $\alpha$ forward via the construction of \cite[Appendix B]{BaeSchmitt1}: by \cite[Theorem B.17]{BaeSchmitt1}, it suffices to treat the case where $E'' \to Z'' \to \frak Y''$ fits in a pullback diagram of the form 
\begin{equation}\label{diagramskowerapushf}
\begin{tikzcd}
E'' \arrow[r] \arrow[d, "p''"]
&Z'' \arrow[r, "z''"] \arrow[d, "p'"]
&\frak Y'' \arrow[d,"p"] \\
E' \arrow[r] 
&Z' \arrow[r, "z'"]
&\frak Y' 
\end{tikzcd}
\end{equation}
where $z'$ is projective. 
Then the pushforward is defined by simply pushing forward on the level of bundles, so 
$$p_*(z'',[V]) = (z',p''_*([V]))\, .$$
If $W = p''(V)$, then $p''_*([V]) = \deg(V/W)[W]$.  

Let $N$ be the normal bundle $N_{\frak X} \frak Y$.
The Gysin maps constructed in \cite[Section 3.1]{Kresch1999Cycle-groups-fo} are described explicitly on level of representatives as follows: $i^!(z',[W])$ is represented by $[C_{W\times_{\frak Y} \frak X} W]$ as a class on the bundle $N_{\mid Z' \times_{\frak Y} \frak X} \oplus E'_{\mid Z' \times_{\frak Y} \frak X}$ with the projective map $\tilde{z}'\colon Z' \times_{\frak Y} \frak X \to \frak X'$ induced by $z'$. Hence, $$i^!p_*(z'',[V]) =(\tilde{z}', \deg(V/W) [C_{W\times_{\frak Y} \frak X} W])\, .$$

Next, we study $q_*i^!(z'',[V])$. 
To start, $i^!(z'',[V])$ is represented by $[C_{V\times_{\frak Y} \frak X} V]$ as class on the bundle $N_{\mid Z'' \times_{\frak Y} \frak X} \oplus E''_{\mid Z'' \times_{\frak Y} \frak X}$   with the projective map $Z'' \times_{\frak Y} \frak X \to \frak X''$ induced by $z''$.
We push $i^!(z'',[V])$
forward via the construction of \cite[Appendix B]{BaeSchmitt1}. We complete the diagram
\begin{center}
\begin{tikzcd}
N_{\mid Z'' \times_{\frak Y} \frak X} \oplus E''_{\mid Z'' \times_{\frak Y} \frak X} \arrow[r] 
&Z'' \times_{\frak Y} \frak X \arrow[r, "\tilde{z}''"]
&\frak X'' \arrow[d,"q"] \\
&~
&~ \frak X'
\end{tikzcd}
\end{center}
to a pullback diagram so that we can pushforward on the levels of bundles: 
\begin{center}
    \begin{tikzcd}[column sep = small, row sep = small]
    ~ 
    &E'' \arrow[rr] \arrow[dd,near start, "p''"]
    &~ 
    &Z'' \arrow[rr, near end,"z''"] \arrow[dd,near start, "p'"]
    &~
    &\frak Y'' \arrow[dd, near start, "p"]\\
    N_{\mid Z'' \times_{\frak Y} \frak X} \oplus E''_{\mid Z'' \times_{\frak Y} \frak X} \arrow[rr]
    &~
    &Z'' \times_{\frak Y} \frak X \arrow[rr, near end, "\tilde{z}''"]  \arrow[dd,near start, "\tilde{p}'"] \arrow[ur]
    &~ 
    &\frak X'' \arrow[ur, "i''"] \arrow[dd,near start, "q"] \\
    ~
    &E' \arrow[rr]
    &~
    &Z' \arrow[rr, near end, "z'"] 
    &~
    &\frak Y' \arrow[dd, "f"] \\
    N_{\mid Z' \times_{\frak Y} \frak X} \oplus E'_{\mid Z' \times_{\frak Y} \frak X} \arrow[rr]
    &~ 
    & Z' \times_{\frak Y} \frak X \arrow[rr, near end, "\tilde{z}'"] \arrow[ur]
    &~ 
    &\frak X' \arrow[dd, "g"] \arrow[ur, "i'"]
    &~ \\
    ~
    &~
    &~
    &~
    &~
    & \frak Y \\
    ~
    &~
    &~
    &~
    & \frak X \arrow[ur, "i"]
    &~ 
    \end{tikzcd}
\end{center}
The map $p'\colon Z'' \to Z'$ induces a map $\tilde{p}'\colon Z'' \times_{\frak Y} \frak X  \to Z' \times_{\frak Y} \frak X $, and the square with $q, \tilde{p}', \tilde{z}'$ and $\tilde{z}''$ is a pullback (as the pullback of a pullback square). 

There is a map $q''\colon N_{\mid Z'' \times_{\frak Y} \frak X} \oplus E''_{\mid Z'' \times_{\frak Y} \frak X} \to  N_{\mid Z' \times_{\frak Y} \frak X} \oplus E'_{\mid Z' \times_{\frak Y} \frak X}$ induced by $p''$ such that it forms a pullback square, and so 
 the pushforward along $q$ is simply 
$$q_*(i^!(z'', [V])) = q_*(\tilde{z}'', [C_{V\times_{\frak Y} \frak X} V]) = (\tilde{z}', q''_*([C_{V\times_{\frak Y} \frak X} V]))\, . $$ 

The proof then reduces to comparing $q''_*([C_{V\times_{\frak Y} \frak X} V])$ and $\deg(V/W)([C_{W\times_{\frak Y} \frak X} W])$, which follows from \cite[Lemma 3.15]{Vistoli1989Intersection-th}. The result is a completely local statement and
therefore extends from the setting of schemes to the setting of Deligne-Mumford stacks which we need here. 
\end{proof}

Let $\phi\colon B \to \frak Y$ and $\phi'\colon B' \to \frak Y$ be morphisms from finite-type schemes, and let 
$$h\colon B' \to B$$ be a $\frak Y$-morphism. By \cref{defbivariantclass}, we have morphisms
\[c(\phi)^{m} \colon Z_m(B) \to \Chow_{m-e}(B)
\ \ \ \
\text{and}
\ \ \ \
c(\phi')^{m} \colon Z_m(B') \to \Chow_{m-e}(B')\, . 
\]
If $h$ is proper, we have a pushforward map{\footnote{We use the same notation for the proper pushforward on the Chow groups.}}
\[h_*\colon Z_m(B') \to Z_m(B)\, \]
which descends to Chow.
\begin{lemma}\label{lem:C1proper}
If $h$ is proper, 
 $$c(\phi)^m\circ h_* = h_*\circ c(\phi h)^m\, . $$
\end{lemma} 
\begin{proof}
Let $[V']\in Z_m(B')$ for 
an irreducible cycle $V'$. Let $V = h(V')$. By definition, $$h_*([V']) = \deg(V'/V) [V]\, .$$ Let $\frak Y'$ be a factorisation of $V' \to V \to \frak Y$. We have 
$$(\on{id}\times h)_*([\frak X' \times V']) = \deg(V'/V) [\frak X' \times V]\, . $$ 
Via the commutative diagram
\begin{equation}
\begin{tikzcd}[column sep=small, row sep = small]
~ & ~ & \frak X' \times_{\frak Y'} V' 	\arrow[rr] \arrow[dd, dashed, near start, "\on{id} \times h"] \arrow[dl, swap, "\psi_{V'}"] & ~ & \frak X' \times V' \arrow[dl] \arrow[dd, dashed, near start, "\on{id} \times h"]\\
~ & V' \arrow[dl, "i_{V'}"] \arrow[rr, near end,swap, "\phi_{V'}"] \arrow[dd, "h"]& ~ & \frak Y' \times V' \arrow[dd, dashed, near start, "\on{id} \times h"] & ~ \\
 B' \arrow[dd, "h"] & ~ & \frak X' \times_{\frak Y'} V \arrow[dl, swap,"\psi_V"] 	\arrow[rr] & ~ & \frak X' \times V \arrow[dl] \\
~ & V \arrow[dl, "i_V"] \arrow[rr, near end, swap, "\phi_V"] &~ & \frak Y' \times V & ~ \\
B \arrow[d, "\phi"] &~ & ~ &~ & ~ \\
\frak Y &~ & ~ &~ & ~ 
\end{tikzcd}
\end{equation}
we compute
\begin{eqnarray*}
    c(\phi)^{m} (h_*(\alpha)) &=& c(\phi)^{m} (\deg(V'/V) [V])\\
    &=& {i_V}_* {\psi_V}_* \phi_V^! (\deg(V'/V) [\frak X' \times V]) \\
    &=& {i_V}_* {\psi_V}_* \phi_V^! ( (\on{id} \times h)_* [\frak X' \times V']) \\
    &=& h_* {i_{V'}}_* {\psi_{V'}}_* \phi_{V'}^! ( [\frak X' \times V'])    \\
    &=& h_* c(\phi h)^m(\alpha)\, ,
\end{eqnarray*}
where the final line follows from compatibility of the Gysin map with proper representable pushforward (Proposition \ref{prop:Gysinproperpush}). 
\end{proof}

If $h:B'\to B$  is flat of relative dimension $n$, we have a pullback map
\[h^*\colon Z_m(B) \to Z_{m+n}(B') \]
which descends to Chow.

\begin{lemma}\label{lem:C2flat}
If $h$ is flat of relative dimension $n$,
$$c(\phi h)^{m+n} \circ h^* = h^* \circ c(\phi)^m. $$
\end{lemma} 

\begin{proof}
Let $[V] \in Z_m(B)$ for an
 irreducible cycle $V$. Let $$V' = h^{-1}(V)\, ,$$ so
$[V'] = h^*([V])$. Let $\frak Y'$ be a factorisation of $V' \to V \to \frak Y$. We have 
$$(\on{id}\times h)^*([\frak X' \times V]) = [\frak X' \times V']\, .$$ 
Via the commutative diagram
\begin{equation}
    \begin{tikzcd}[column sep=small, row sep = small]
~ & ~ & \frak X' \times_{\frak Y'} V' 	\arrow[rr] \arrow[dd, dashed, near start, "\on{id} \times h"] \arrow[dl, swap, "\psi_{V'}"] & ~ & \frak X' \times V' \arrow[dl] \arrow[dd, dashed, near start, "\on{id} \times h"]\\
~ & V' \arrow[dl, "i_{V'}"] \arrow[rr, near end, , swap, "\phi_{V'}"] \arrow[dd, "h"]& ~ & \frak Y' \times V' \arrow[dd, dashed, near start, "\on{id} \times h"] & ~ \\
 B' \arrow[dd, "h"] & ~ & \frak X' \times_{\frak Y'} V \arrow[dl, swap, "\psi_V"] 	\arrow[rr] & ~ & \frak X' \times V \arrow[dl] \\
~ & V \arrow[dl, "i_V"] \arrow[rr,swap, near end, "\phi_V"] &~ & \frak Y' \times V & ~ \\
B \arrow[d, "\phi"] &~ & ~ &~ & ~ \\
\frak Y &~ & ~ &~ & ~ 
\end{tikzcd}
\end{equation}
we compute
\begin{align*}
     c(\phi h)^{m+n} ([V']) &= {i_{V'}}_* {\psi_{V'}}_* \phi_{V'}^! ([\frak X' \times V']) \\
    &= {i_{V'}}_* {\psi_{V'}}_* \phi_{V'}^! ( (\on{id} \times h)^* [\frak X' \times V]) \\
    &= h^* {i_V}_* {\psi_V}_* \phi_V^! ( [\frak X' \times V]) \\
    &=  h^* c(\phi)^m([V])\, , 
\end{align*}
where the final line follows from the compatibility of the Gysin map with flat pullback for a morphism of finite type algebraic stacks (\cite[Section 3.1]{Kresch1999Cycle-groups-fo}) and the pullback and pushforward formulas. 
\end{proof}

Let $\phi\colon B \to \frak Y$ be a morphism from a finite type scheme
as above. Let $$g\colon B \to Z$$ be a morphism of finite type schemes, and let $i\colon Z' \to Z$ be a regular embedding of codimension $f$. 
Form the fiber square 
\begin{center}
\begin{tikzcd}
B' \arrow[r] \arrow[d, "i'"] & Z' \arrow[d,  "i"] \\
B \arrow[r, "g"] \arrow[d, "\phi"] & Z \\
\, \frak Y.  &~ 
\end{tikzcd}
\end{center} 
Let $V$ be an irreducible cycle in $B$ with inverse image $V' = (i')^{-1}(V)$. 
We choose a representative $i^![V] = \sum_j n_j [V'_j]$ in $Z_{m-f}(V')$. 
\begin{lemma}\label{lem:C3regimb}
We have 
\[i^! c(\phi)^m ([V]) = c(\phi i')^{m-f}\left(\sum n_j [V'_j]\right). \]
\end{lemma}
\begin{remark}
In particular, once we have shown that the maps $c(\phi)^m$ pass to rational equivalence, 
Lemma \ref{lem:C3regimb} will imply 
\[i^! c(\phi)^m(\alpha) = c(\phi i')^{m-f}(i^! \alpha) \]
for  $\alpha \in \Chow_m(B)$. 
\end{remark}
\begin{proof}  
From the equality $i^! [V] = \sum n_j [V'_j]$, we deduce 
$$i^![\frak X' \times V] = \sum n_j [\frak X' \times V'_j]\, .$$ 
Via the commutative diagram\footnote{We should also add another layer of the diagram for the $V'_j$. 
}


\begin{equation}
\begin{tikzcd}[column sep = small, row sep = small]
~ & ~ & ~  & \frak X' \times_{\frak Y'} {V'} 	\arrow[rr] \arrow[dd, dashed, near start, "\on{id} \times i'"] \arrow[dl, swap, "\psi_{V'}"] & ~ & \frak X' \times  {V'} \arrow[dl] \arrow[dd, dashed, near start, "\on{id} \times i'"]\\
~ & ~ & {V'}  \arrow[dll] \arrow[dl, "i_{V'}"] \arrow[rr, swap,near end, "\phi_{V'}"] \arrow[dd, "i'"]& ~ & \frak Y' \times {V'} \arrow[dd, dashed, near start, "\on{id} \times i'"] & ~ \\
Z' \arrow[dd,"i"] & B' \arrow[l] \arrow[dd, "i'"] & ~ & \frak X' \times_{\frak Y'} V \arrow[dl, swap, "\psi_V"] 	\arrow[rr] & ~ & \frak X' \times V \arrow[dl] \\
~ & ~ & V \arrow[dll] \arrow[dl, "i_V"] \arrow[rr, swap, near end, "\phi_V"] &~ & \frak Y' \times V & ~ \\
 Z & B \arrow[l, "g"] \arrow[d, "\phi"] &~ & ~ &~ & ~ \\
~ & \frak Y,  &~ & ~ &~ & ~ 
\end{tikzcd}
\end{equation}
we compute
\[i^! c(\phi)^m ([V]) = i^! {i_V}_* {\psi_V}_* \phi_V^! ( [\frak X' \times V]) \]
and
\[c(\phi i')^{m-f}(\sum n_j [V'_j]) = \sum n_j  {i_{V'_j}}_* {\psi_{V'_j}}_* \phi_{V'_j}^! ( [\frak X' \times V'_j]) =  {i_{V'}}_* {\psi_{V'}}_* \phi_{V'}^! ( i^![\frak X' \times V])\, . \]
We deduce equality of these expressions by using the compatibility of Gysin maps with proper pushforward (\cref{prop:Gysinproperpush}) and then the commutativity of Gysin maps from \cite{Kresch1999Cycle-groups-fo} to obtain $\phi_{V'}^!  i^! = i^! \phi_V^!$. 
\end{proof}

\begin{lemma}\label{lem:ratequiv}
The morphisms $c(\phi)^{m}$ from \cref{defbivariantclass} pass to rational equivalence, 
\[c(\phi)^{m}\colon \Chow_m(B) \to \Chow_{m-e}(B)\, .\]
\end{lemma}

\begin{proof}
The proof is now completely analogous to \cite[Theorem 17.1]{Fulton1984Intersection-th}. 
\end{proof}

\subsubsection{Properties}
The class constructed in \cref{defbivariantclass} is invariant under proper birational maps in the following sense.
\begin{proposition}\label{prop:invarianceproperbirat}
Let $f\colon \frak W \to \frak X$ be a proper DM-type birational morphism of locally finite type algebraic stacks over ${\field}$ of pure dimension, with $\frak X$ stratified by global quotient stacks. Then,
\[(a\circ f)_{\on{op}}[\frak W] = a_{\on{op}}[\frak X] \in \CHop^e(\frak Y)\]
where $(a\circ f)_{\on{op}}[\frak W]$ and $a_{\on{op}}[\frak X]$ are the operational classes constructed in \cref{defbivariantclass} with respect to $a\circ f\colon \frak W \to \frak Y$ and $a\colon \frak X \to \frak Y$.
\end{proposition}

\begin{proof}
The proper pushforward of the fundamental class along $f$ is the fundamental class, as the map $f$ is birational and hence of degree 1. 

Choose a factorisation $V \to\frak Y' \to\frak Y$. Denote by $\frak X'$ the pullback of $\frak X$ along $a$ and by $\frak W'$ the the pullback of $\frak X'$ along $f$. As in \eqref{eq:definingdiagram}, we
form a pullback diagram 
\begin{equation} 
 \begin{tikzcd}
   \frak W' \times_{\frak Y'} V \arrow[r] \arrow[d, "\tilde{f}"] & \frak W' \times V \arrow[d, "f \times \on{id}"]  \\
  \frak X' \times_{\frak Y'} V \arrow[r] \arrow[d, "\psi_V"] & \frak X' \times V \arrow[d, "\on{a} \times \on{id}"]  \\
  V \arrow[r, "\phi_V"] & \ \frak Y' \times V\, . 
\end{tikzcd}
\end{equation}
 \cref{prop:Gysinproperpush} then yields
\begin{eqnarray*}
{\psi_V}_* \phi_V^!([\frak X' \times V]) &=&  {\psi_V}_* \phi_V^!((f \times \on{id})_*[\frak W' \times V])\\
&=& {\psi_V}_* \tilde{f}_* \phi_V^! ([\frak W' \times V])\\
&= &{(\psi_V \tilde{f})}_* \phi_V^! ([\frak W' \times V])\, ,
\end{eqnarray*}
which is the required equality. 
\end{proof}

We will also require a flat pullback property.
Let $\frak X, \frak Y, \frak Z$ be pure dimensional algebraic stacks of locally finite type over ${\field}$  with $\frak Y$ and $\frak Z$ smooth. 
Suppose we have a fibre diagram
\begin{center}
\begin{tikzcd}
\frak X \times_{\frak Y} \frak Z \arrow[r, "\tilde{a}"] \arrow[d, "\tilde{f}"] 
& \frak Z \arrow[d, "f"] \\
\frak X \arrow[r, "a"] 
&\frak Y
\end{tikzcd}
\end{center}
where  $a: \frak X \to \frak Y$ is a proper DM-type morphism and $f: \frak Z \to \frak Y$ is flat {\em and} lci\footnote{A flat and lci map is called {\em syntomic}.}, with $\frak Z$ stratified by global quotient stacks. 
\begin{lemma}\label{lem:pushpullcommutes}
In $\CHop(\frak Z)$, we have
$$\tilde{a}_{\on{op}} [\frak X \times_{\frak Y} \frak Z] = f^* a_{\on{op}}[\frak X]\, .$$
\end{lemma}

\begin{proof}
Let $\frak X', \frak Y', \frak Z'$ denote appropriate finite-type factorisations as in \cref{defbivariantclass}. Then we compare the two operational classes via the following diagram:
\begin{center}
\begin{tikzcd}
~
&(\frak X' \times_{\frak Y'} \frak Z') \times_{\frak Z'} V \arrow[d, "\sim" labl] \arrow[r]
& \frak X' \times_{\frak Y'} \frak Z' \times V \arrow[dr, "\tilde{f} \times \on{id}"] \arrow[dd, near start, "\tilde{a} \times \on{id}"] \\
~
&\frak X' \times_{\frak Y'} V \arrow[rr] \arrow[d, "\psi_V"]
&~
&\frak X' \times V \arrow[d, "a \times \on{id}"] \\
~
&V \arrow[dl, "i_V"] \arrow[r, "\phi_V"] \arrow[rr, bend right, "(f \circ \phi)_V"]
& \frak Z' \times V \arrow[r, "f \times \on{id}"]
& \frak Y' \times V \\
B &~ &~ &~ .
\end{tikzcd}
\end{center}
By definition,
\[\tilde{a}_{\on{op}} [\frak X \times_{\frak Y} \frak Z] (\phi)([V]) = {i_V}_* {\psi_V}_* \phi_V^! ([\frak X'\times_{\frak Y'} \frak Z' \times V]) \]
and
\[f^* a_{\on{op}}[\frak X](\phi)([V]) = a_{\on{op}}[\frak X](f \circ \phi)([V])  = {i_V}_* {\psi_V}_* (f \circ \phi)_V^! ([\frak X'\times V])\, .\]
Since $f$ is lci, the above expression is equal to 
\[{i_V}_* {\psi_V}_* \phi_V^! (f \times \on{id})^! ([\frak X'\times V])\, .\]
Because $f$ is also flat, we see as in  \cite[Prop 6.6(b)]{Fulton1984Intersection-th} that we obtain
\[{i_V}_* {\psi_V}_* \phi_V^! (\tilde{f} \times \on{id})^* ([\frak X'\times V]) = {i_V}_* {\psi_V}_* \phi_V^! ([\frak X' \times_{\frak Y'} \frak Z' \times V])\, . \]
which yields the required equality. 
\end{proof}


\section{The universal double ramification cycle}\label{sec:uni_DR_defs}
\subsection{Overview}
We fix a genus $g$, a number of markings $n$, and a vector $A=(a_1,\ldots,a_n) \in \bb Z^n$
of ramification data satisfying
$$\sum_{i=1}^{n}a_i=d\, .$$
We define here 
the associated universal twisted double ramification cycle class
in the operational Chow group of the universal Picard stack $\Picabs_{g,n,d}$.
The operational class is the class associated to a certain proper
representable
morphism 
$$\cat{Div}_{g,A} \to \Picabs_{g,n,d}$$
using the theory of Section \ref{sec:constructing_chow_class}.
Our goal here is to define the stack $\cat{Div}_{g,A}$ over $\Picabs_{g,n,d}$. 


We will present three essentially equivalent definitions of the universal
twisted double ramification cycle
in Sections \ref{sec:naive_def}--\ref{sec:rub_log_def} which yield the same operational class: 
\begin{enumerate}
\item[$\bullet$] a definition in Section \ref{sec:naive_def} by closing
the Abel-Jacobi section
which is simple to state but difficult to
handle,
\item[$\bullet$] an intrinsic logarithmic definition in Section
\ref{sec:log_def}
following Marcus-Wise \cite{Marcus2017Logarithmic-com},
 
\item[$\bullet$] a slight variation of the log definition in
Section \ref{sec:rub_log_def}
which facilitates  comparison to the spaces of rubber maps.
\end{enumerate}
After analyzing the set-theoretic closure of the Abel-Jacobi section in
Section \ref{sec:image_of_aj},
the equality of the three resulting classes will be shown in \cref{sec:proof_of_equiv_defs}. 
In \cref{Sect:AJextension}, we briefly discuss 
the lift of universal twisted
 double ramification cycle to operational b-Chow.

\subsection{$\mathsf{DR}^{\mathsf{op}}_{g,A}$ by closure}
\label{sec:naive_def}

We define the Abel-Jacobi section $\sigma$ of $\Picabs_{g,n,d} \to \frak M_{g,n}$  by
\begin{equation}
\sigma\colon \frak M_{g,n} \to \Picabs_{g,n,d}\, ,\ \ \ \ (C, p_1, \dots, p_n) \mapsto \ca O_C\Big(\sum_{i=1}^n a_i p_i\Big)\, . 
\end{equation}
The section $\sigma$  is not a closed immersion (both because 
of the $\bb G_m$-automorphism groups of line bundles and because the
image is not closed).
However, $\sigma$ is quasi-compact and relatively representable by schemes, and hence admits a well-defined \emph{schematic image} (we use that the formation of the schematic image is compatible with flat base-change, see \cite[\href{https://stacks.math.columbia.edu/tag/081I}{Tag 081I}]{stacks-project}). 
The schematic image is the smallest closed reduced substack through which
$\sigma$ factors.

Since the schematic image $\bar\sigma$ is a closed substack of pure dimension, 
$$\iota:\bar\sigma\to \Picabs_{g,n,d}\, ,$$
we obtain an operational class $\iota_{\mathsf{op}}[\bar\sigma]$ by \cref{defbivariantclass}.  
Our first definition of the universal twisted double ramification
cycles is via the schematic image of $\sigma$:
\begin{equation}\label{ddd111}
\DRop_{g,A} = \iota_{\mathsf{op}}[\bar\sigma] \, \in\, 
\mathsf{CH}^g_{\mathsf{op}}(\Picabs_{g,n,d})
\, . 
\end{equation}

Let $\Picabs_{\ul 0} \hra \Picabs$ be the open substack consisting of line bundles having degree 0 on every irreducible component of every geometric fibre (\emph{multidegree $\ul 0$}), and let $\Picrel_{\ul 0} \hra \Picrel$
be defined analogously. We have a commutative diagram in which all squares are pullbacks: 
\begin{equation}\label{eq:multidegree0}
\begin{tikzcd}
(B \bb G_m)_{\frak M_{g,n}} \arrow[r] \arrow[d, equals] & \frak M_{g,n} \arrow[d, equals] \arrow[ddrr, "e = \ca O_C"]\\
\bar \sigma^{\ul 0} \arrow[r]\arrow[dd]  \arrow[drr]& \bar\sigma_{\mathsf{rel}}^{\ul 0}\arrow[dd]\arrow[drr]\\
& & \Picabs_{g,n, \ul 0} \arrow[r]\arrow[dd] & \Picrel_{g,n,\ul 0}\arrow[dd]\\
\bar\sigma \arrow[r]\arrow[drr] & \bar\sigma_{\mathsf{rel}} \arrow[drr] &&\\
& & \Picabs_{g,n,0} \arrow[r] & \Picrel_{g,n,0}\, .\\
\end{tikzcd}
\end{equation}

Let $(C/B, p_1, \dots, p_n)$ be a prestable curve over
a scheme $B$ of finite type over $\field$.
Let $L$ be a line bundle on $C$ such that 
$L(-\sum_{i=1}^n a_i p_i)$ is of multidegree $\ul 0$ 
for $A=(a_1,\ldots, a_n) \in \bb Z^n$.
The 
data 
$$C\to B\, , \ \ \ \ L\Big(-\sum_{i=1}^n a_i p_i\Big)\to C$$
determine a map $$\phi\colon B \to \Picabs_{g,n,\ul 0}\, ,$$
and we form a pullback diagram
\begin{equation}\label{dia:easy}
    \begin{tikzcd}
    B' \arrow[r] \arrow[d,"\psi"] & \frak M_{g,n}\arrow[d, "e"]\\
    B \arrow[r, "\phi^{\rel}"] & \Picrel_{g,n,\ul 0}\, . 
    \end{tikzcd}
\end{equation}
Since $\frak M_{g,n}$ is smooth and $\Picrel_{g,n,\ul 0}$ is separated, the map $e$ is a regular embedding.

\begin{lemma} In the multidegree $\ul 0$ case, we have
\begin{equation}
    \DRop_{g,A}(\phi)([B]) = \psi_*e^![B]\, .
\end{equation}
\end{lemma}
\begin{proof}
We begin by expanding the diagram \cref{dia:easy} to
\begin{equation}
    \begin{tikzcd}
    B' \arrow[r] \arrow[d,"\psi"] & \frak M_{g,n} \times B \arrow[r,"f'"] \arrow[d] & \frak M_{g,n}\arrow[d, "e"]\\
    B \arrow[r, "\phi' "]& \Picrel_{g,n,\ul 0} \times B \arrow[r,"f"] & \Picrel_{g,n,\ul 0}\, . 
    \end{tikzcd}
\end{equation}
Since $\Picabs_{g,n} \to \Picrel_{g,n}$ is smooth, we deduce from \cref{lem:pushpullcommutes} and  diagram \cref{eq:multidegree0} that 
\begin{equation}
    \DRop_{g,A}(\phi)([B]) = \psi_*\phi'^![\frak M_{g,n} \times B]\, .
\end{equation}
We then compute
\begin{equation}
    \begin{split}
        \DRop_{g,A}(\phi)([B]) & = \psi_*\phi'^![\frak M_{g,n} \times B]\\
        & = \psi_*\phi'^! f'^*[\frak M_{g,n}]\\
        & = \psi_*\phi'^! f'^*e^![\Picrel_{g,n,\ul 0}]\\
        & = \psi_*\phi'^! e^! f^*[\Picrel_{g,n,\ul 0}]\\
        & = \psi_*\phi'^! e^! [\Picrel_{g,n,\ul 0}\times B]\\
        & = \psi_*e^! \phi'^!  [\Picrel_{g,n,\ul 0}\times B]\\
        & = \psi_*e^! [B]\, . \\
    \end{split}
\end{equation}
\end{proof}
In particular, if the intersection of $B$ with the unit section in $\Picrel_{g,n,\ul 0}$ is transversal, then we simply take the naive intersection in $\Picrel_{g,n,\ul 0}$ and push it down to $B$.

\subsection{Logarithmic definition of \texorpdfstring{$\DRop$}{DR\^{}op}}\label{sec:log_def}


\subsubsection{Overview of log divisors}

We begin by recalling various results and definitions from log geometry. We refer the reader to \cite{Kato1989Logarithmic-str} for basics on log geometry and \cite{Marcus2017Logarithmic-com} for the details of what we do here. 
While log geometry will not play a substantial role elsewhere in the paper, 
it will reappear in \cref{sec:comparing_stacks}.

Given a log scheme $S = (S, M_S)$, we write
\begin{equation*}
\mathbb G_m^{log}(S) = \Gamma(S, M_S^{gp})
\ \ \ \ 
\text{and} \ \ \ \  
\mathbb G_m^{trop}(S) = \Gamma(S, \bar M_S^{gp}), 
\end{equation*}
which we call the logarithmic and tropical multiplicative groups.  Both can naturally be extended to presheaves on the category $\cat{LSch}_S$ of log schemes over $S$,
and both
admit log smooth covers by log schemes (with subdivisions $\bb P^1$ and $[\bb P^1/\bb G_m]$ respectively). A \emph{log (tropical)} line on $S$ is a $\bb G_m^{log}$ ($\bb G_m^{trop}$) torsor on $S$, for the strict \'etale topology.

\begin{definition}(See \cite[Def. 4.6]{Marcus2017Logarithmic-com}) 
\label{Def:logdivisor}
Let $C$ be a logarithmic curve over a logarithmic scheme $S$. A \emph{logarithmic divisor} on $C$ is a tropical line $P$ over $S$ and an $S$-morphism $C \to P$. 

Let $\cat{Div}^{\rel}_g$ be the stack{\footnote{The stack $\cat{Div}^{\rel}_g$ was denoted $\cat{Div}_g$ in \cite{Marcus2017Logarithmic-com}, but we wish to reserve the
latter notation for a certain $\bb G_m$-gerbe over $\cat{Div}^{\rel}_g$ which will play a much more prominent role in our paper.}} 
in the strict \'etale topology on logarithmic schemes whose $S$-points are triples $(C, P, \alpha)$ where $C$ is a logarithmic curve of genus $g$ over $S$, $P$ is a tropical line over $S$, and 
$$\alpha \colon  C \to P$$ is an $S$-morphism. 
\end{definition}

If $S$ is a geometric log point and $C/S$ a log curve, then the set of isomorphism classes of $\cat{Div}_g^{\rel}(S)$ is given by $\pi_*\bar M_C^{gp}/\bar M_S^{gp}$. At the markings, an element of $\pi_*\bar M_C^{gp}/\bar M_S^{gp}$ determines an element of the groupified relative characteristic monoid $\bb Z$ (for those who prefer a tropical perspective, this can be viewed as the outgoing slope at the marking). 

\begin{definition}
Let $\cat{Div}^{\rel}_{g,A}$ be the (open and closed) substack of $\cat{Div}^{\rel}_g$ consisting of those triples where the curve carries exactly $n$ markings and where on each geometric fibre the outgoing slopes at the markings correspond to $A$ (our
log curves come with an ordering of their markings as explained in \cref{sec:prestable_vs_log_curves}). 
\end{definition}



\begin{remark}
It is natural to ask for a description of the functor of points of the underlying (non-logarithmic) stack of $\cat{Div}^{\rel}_{g,A}$ as a fibred category over $\frak M_{g,n}$. However, we expect that such a description will not be simple. A closely related problem is solved in \cite{Biesel2016Fine-compactifi}, and the intricacy of the resulting definition suggests that the path will not be easy. 
\end{remark}


\subsubsection{Abel-Jacobi map}\label{sec:abel-jacobi-map}

Given a log curve $\pi\colon C \to S$ of genus $g$,  the right-derived pushforward to $S$ of the standard exact sequence
\begin{equation}\label{eq:log_exact_seq}
1 \to \ca O_C^\times \to M_C^{gp} \to \bar M_C^{gp} \to 1\, ,
\end{equation}
yields a natural map 
$$\pi_*\bar M_C^{gp} \to R^1\pi_*\ca O_C^\times\, ,$$
which factors via the quotient $$\pi_*\bar M_C^{gp}/\bar M_S^{gp} = \cat{Div}^{\rel}_g(S)\, .$$ We therefore obtain
a \emph{relative Abel-Jacobi} map $$\aj^{\rel}\colon \cat{Div}^{\rel}_{g} \to \Picrel_{g}\, ,$$ 
which restricts to maps 
$$\aj^{\rel}\colon \cat{Div}^{\rel}_{g,A} \to \Picrel_{g,n,d}\, .$$ 

For a first example, suppose $S$ is a geometric log point with $\bar M_S = \bb N$. 
The data then determines
to first order a deformation of the curve over a DVR (which we take generically smooth),  and the section of $\pi_*\bar M_C^{gp}/\bar M_S^{gp}$ gives the multiplicities of components in the special fibre and the twists by the markings. 

For another example, consider what happens over the locus of (strict) smooth curves. Writing $N/\ca M^{\mathsf{log}}_{g,n}$ for the stack of markings (finite \'etale), we see $\cat{Div}^{\rel}_g$ is just the category of locally constant functions from $N$ to $\bb Z$ -- in other words, choices of outgoing slope/weight on each leg. 
The Abel-Jacobi map yields  $\ca O_C(\sum_{i=1}^n a_ip_i)$ where 
the $p_i$ are the markings and $a_i$ are the weights. In particular, we see  $$\cat{Div}^{\rel}_{g,A} \to \ca M^{\mathsf{log}}_{g,n}$$
is birational (as we fixed an ordering of the markings) and log \'etale. 

\begin{definition}\label{def:div}
Let  $\cat{Div}_{g}$ be the fibre product 
\begin{equation}
\cat{Div}^{\rel}_{g} \times_{\Picrel_{g}}\Picabs_{g}\, .
\end{equation}
\end{definition}

More concretely, an $S$-point of $\cat{Div}_{g}$ is a quadruple $(C,P,\alpha,\ca L)$ where $(C, P, \alpha)$ is an $S$-point of $\cat{Div}^{\rel}_{g}$ and $\ca L$ is a line bundle on $C$ satisfying{\footnote{Here, $[\,]$ denotes the
equivalence class under the relations of isomorphism and tensoring with the pullback of a line bundle from the base. }}
$$[\ca L] = \aj^{\rel}(C, P, \alpha)\in \Picrel_{g}(S)\, .$$ 
We will denote by $\aj$ the resulting Abel-Jacobi map 
$$\cat{Div}_g \to \Picabs_{g}\, .$$
Observe that $\cat{Div}_g$ is a $\bb G_m$-gerbe over $\cat{Div}^{\rel}_{g}$, just as $\Picabs_{g,n,d}$ is a $\bb G_m$-gerbe over $\Picrel_{g,n, d}$. 
Analogously,
we define  
\begin{equation}
\cat{Div}_{g,A} = \cat{Div}^{\rel}_{g,A} \times_{\Picrel_{g,n,d}}\Picabs_{g,n,d}
\ \ \ \ 
\text{and} 
\ \ \ \ \aj\colon \cat{Div}_{g,A} \to \Picabs_{g,n,d}\, .
\end{equation}



We summarise the key properties of the Abel-Jacobi map. These are proven
in \cite{Marcus2017Logarithmic-com} for $\aj^{\rel}$, and are stable under base-change. 
\begin{proposition}
The Abel-Jacobi map 
$$\aj\colon \cat{Div}_{g,A} \to \Picabs_{g,n,d}$$
is proper, relatively representable by algebraic spaces, and is a monomorphism of log stacks. 
\end{proposition}

We obtain an operational class $\aj_{\mathsf{op}}[\cat{Div}_{g,A}]$ associated by \cref{defbivariantclass} to the Abel-Jacobi map $\aj$.
Our second definition of the universal twisted double ramification
cycles is via $\aj$:
\begin{equation}\label{ddd222}
\DRop_{g,A} = \aj_{\mathsf{op}}[\cat{Div}_{g,A}] \, \in\, 
\mathsf{CH}^g_{\mathsf{op}}(\Picabs_{g,n,d})
\, . 
\end{equation}
The equivalence of 
definitions \eqref{ddd111} and \eqref{ddd222} will be proven in \cref{sec:proof_of_equiv_defs}.



\subsubsection{Description of \texorpdfstring{$\cat{Div}_g$}{Div\_g} with log line bundles}
Our approach to $\cat{Div}_g$ in \cref{def:div} via
a fiber product is indirect.
While it will not be used in the paper, a more conceptual path is to consider the stack $\cat{Div}'_g$ whose objects are tuples
\begin{equation*}
(C/S, \ca P, \alpha)
\end{equation*}
where $C/S$ is a log curve, $\ca P$ is a \emph{logarithmic} line on $S$  (a $\bb G_m^{log}$ torsor), and $\alpha$ is a map from $C$ to the \emph{tropical} line $P$ on $S$ induced from $\ca P$ by the exact sequence \eqref{eq:log_exact_seq}. 
There is a natural map 
$$\cat{Div}'_g \to \cat{Div}^{\rel}_g\, .$$ We can see $\alpha$ as a section of the tropicalisation of the pullback of $\ca P$ to $C$. As such, by the sequence \eqref{eq:log_exact_seq}, $\alpha$ induces a $\bb G_m$-torsor on $C$, giving us an Abel-Jacobi map $\cat{Div}_g'\to \Picabs_g$. Together these maps induce a map 
\begin{equation*}
\cat{Div}_g' \to \cat{Div}_g
\end{equation*}
to the fibre product, and a local computation verifies that this is an isomorphism. 

 The above discussion points\footnote{The {\it unit section} of $\Picabs$ is given by the stack of $\bb G_m$ torsors on the base. Similarly, the {\it unit section} of the \emph{logarithmic} Picard stack $\frak{LogPic}_g$ is given by the stack of $\bb G_m^{log}$ torsors on the base. The natural map $\Picabs_g \to \frak{LogPic}_g$ is neither injective nor surjective: a logarithmic line bundle comes from a line bundle if and only if the associated tropical line bundle is trivial, and a choice of trivialisation of that tropical line bundle then determines a lift to a line bundle. Hence, we see that $\cat{Div}'_g$ is precisely the pullback of the unit section of $\frak{LogPic}_g$ to $\Picabs_g$.} towards a definition of the double ramification cycle via the \emph{logarithmic Picard functor} of \cite{Molcho2018The-logarithmic}
which we hope will be pursued in future.


\subsection{Logarithmic rubber definition of \texorpdfstring{$\DRop$}{DR\^{}op}}\label{sec:rub_log_def}
Marcus and Wise introduce a slight variant $\cat{Rub}^{\mathsf{rel}}_{g}$ of the stack $\cat{Div}^{\mathsf{rel}}_{g}$
which
parametrises pairs $(C,P, \alpha)$ where $P$ is a tropical line on $S$ and $\alpha\colon C \to P$ is an $S$-morphism \emph{such that on each geometric fibre over $S$ the values taken by $\alpha$ on the irreducible components of $C$ are totally ordered in $(\bar{M}_S^{gp})_s$}, with the ordering given by declaring the elements of $(\bar{M}_S)_s$ to be the non-negative elements. 

The space $\cat{Rub}_{g,A}$,
defined via pullback
$$
\cat{Rub}_{g,A} \, \stackrel{\sim}{=}\, 
\cat{Div}_{g,A}
\times_{\cat{Div}^{\mathsf{rel}}
_{g,A}} 
\cat{Rub}^{\mathsf{rel}}_{g,A}\, ,$$
is pure dimensional and comes with a proper birational map 
\begin{equation}\label{rrrbbb}
\cat{Rub}_{g,A} \to \cat{Div}_{g,A}\, .
\end{equation}
The stack $\cat{Rub}_{g,A}$  will 
play an important role in the comparison to classes coming from stable map spaces in \cref{sec:comparing_stacks}. 


We obtain an operational class $\aj^{\mathsf{rub}}_{\mathsf{op}}[\cat{Rub}_{g,A}]$ associated by \cref{defbivariantclass} to 
$$\aj^{\mathsf{rub}}: \cat{Rub}_{g,A} \to \Picabs_{g,n,d}$$
obtained by composing \eqref{rrrbbb} with $\aj$.
Our third defintion of the universal twisted double ramification
cycles is via $\aj^{\mathsf{rub}}$:
\begin{equation}
\DRop_{g,A} = \aj^{\mathsf{rub}}_{\mathsf{op}}[\cat{Rub}_{g,A}] \, \in\, 
\mathsf{CH}^g_{\mathsf{op}}(\Picabs_{g,n,d})
\, . 
\end{equation}
The equivalence with the first two 
definitions will be proven in \cref{sec:proof_of_equiv_defs}.

\subsection{The image of the Abel-Jacobi map}\label{sec:image_of_aj}
The set theoretic image of the Abel-Jacobi map 
$$\aj\colon \cat{Div}_{g, A} \to \Picabs_{g,n,d}$$ 
can be characterized 
in terms of a condition on twisted divisors similar to the 
conditions of \cite{Farkas2016The-moduli-spac} for the moduli spaces $\widetilde{\HH}_g(A)$
twisted
canonical divisors.

Given a prestable graph $\Gamma_{\delta}$ of degree $d$, a \emph{twist} on $\Gamma_\delta$ is a function $I: H(\Gamma) \to \mathbb{Z}$ satisfying
\begin{enumerate}
\item[(i)] $\forall j\in \L(\Gamma_\delta)$, corresponding to
 the marking $j\in \{1,\ldots, n\}$,
$$I(j)=a_j\, ,$$
\item[(ii)] $\forall e \in \E(\Gamma_\delta)$, corresponding to two half-edges
$h,h' \in \H(\Gamma_\delta)$,
$$I(h)+I(h')=0 \, ,$$
\item[(iii)] $\forall v\in \V(\Gamma_\delta)$,
$$\sum_{v(h)= v} I(h)= \delta(v) \, ,$$ 
where the sum is taken over {\em all} $\mathsf{n}(v)$ half-edges incident 
to $v$.
\item[(iv)] There is no \emph{strict cycle}\footnote{A strict cycle 
 is a sequence $\vec e_i = (h_i, h_i')$, $i=1, \ldots, \ell$ of directed edges in $\Gamma$ forming a closed path in $\Gamma$ such that $I(h_i) \geq 0$ for all $i$ and there exists at least one $i$ with $I(h_i)>0$. Condition (iv)
corresponds to the combination of the Vanishing, Sign, and Transitivity conditions for twists in \cite[Section 0.3]{Farkas2016The-moduli-spac}.} in $\Gamma$.
\end{enumerate}

Let $(C,p_1, \ldots, p_n)$ together with a line bundle $\mathcal{L} \to C$ of degree $d$ be a geometric point of $\Picabs_{g,n,d}$. Let $\Gamma_\delta$ be the prestable graph of $C$ decorated with the degrees $\delta(v)$ of the line bundle $\mathcal{L}$ restricted to the components $C_v$ of $C$. Given a twist $I$ on $\Gamma_\delta$, let
\[\eta_I : C_I \to C\]
be the partial normalization of $C$ at all nodes $q \in C$ corresponding to edges $e=(h,h')$ of $\Gamma$ with $$I(h) = -I(h') \neq 0\, .$$
Denote by $q_h, q_{h'} \in C_I$ the preimages of $q$ corresponding to the half-edges $h,h'$. Denote by $\widehat{p}_i \in C_I$ the unique preimage of the $i$th marking $p_i \in C$.

We say the point $(C,p_1, \ldots, p_n, \mathcal{L})$ of $\Picabs_{g,n,d}$ satisfies the \emph{twisted divisor condition} for the integer vector $A$ if and only if there exists a twist $I$ on $\Gamma_\delta$ such that on the partial normalization $C_I$ of $C$ there exists an isomorphism of line bundles
\begin{equation}\label{eq:twisted_diff_condition}
    \eta_I^* \mathcal{L} \cong  \mathcal{O}_C\left(\sum_{i=1}^n a_i \widehat{p}_i + \sum_{h \in H(\Gamma)} I(h) q_h \right).
\end{equation}
For $\mathcal{L}=\omega_C$ this exactly corresponds to the notion \cite[Definition 1]{Farkas2016The-moduli-spac} of a twisted canonical divisor. 

\begin{proposition}\label{prop:aj_image}
A geometric point $(C,p_1, \ldots, p_n, \mathcal{L})$ of $\Picabs_{g,n,d}$ lies in the image of the Abel-Jacobi map $\aj \colon \cat{Div}_{g,A} \to \Picabs_{g,n,d}$ if and only if the twisted divisor condition for the vector $A$ is satisfied.
\end{proposition}




\begin{proof}
We may suppose that ${\field}$ is a separably closed field and $(C/S, p_1, \dots, p_n)$ is a prestable curve over ${\field}$. We must show
that the twisted divisor condition is equivalent to the existence of a log structure on $C/S$ satisfying the following property: {\em $C/S$  is a log curve with markings given by the $p_i$  which admits
a global section $\alpha$ of $\bar M_C^{gp}$ with outgoing slope at $p_i$ given by $a_i$}. 

Suppose that such a log structure exists. From the log structure,
we can determine a twist. To each leg we associate the outgoing 
slope of $\alpha$ on the corresponding leg. For an edge $\{ h, h'\}$,
we define $\ell(\{ h, h'\})$ to be the element of $\bar M_S$ associated via the data of the log morphism $C \to S$ to the node of $C$ corresponding to $\{h,h'\}$. If $\{ h, h'\}$ is an edge with half-edge $h$ attached to a vertex $u$ and the opposite half-edge $h'$ attached to $v$, we set the integer $I(h)$ to be the unique integer such that 
\begin{equation}\label{eq:slope}
\alpha(u) + I(h)\cdot \ell(\{h,h'\})  = \alpha(v) \in \bar M_S^{gp}\, .
\end{equation}
That such an $I(h)$ exists follows from the structure of $\bar M_C^{gp}$. 

Next, we verify that $I$ is a twist. Conditions (i) and (ii) are immediate from the construction.  We deduce condition (iv) because by following
a strict cycle starting at some vertex $u$ and applying \eqref{eq:slope} along each edge would yield 
$\alpha(u) < \alpha(u)$, 
which is impossible. Condition (iii) is immediate from the twisted divisor condition \eqref{eq:twisted_diff_condition} and the fact that isomorphic line bundles have the same degree, so this will be proven once we have checked the latter condition. 

For the latter condition,
we must work a little harder. To start,
we claim that there exists a morphism $\bar M_S \to \bb N$ which does not send the label of any edge to $0$. Indeed, $\bar M_s^{gp}$ injects into its groupification which is a finitely generated torsion-free abelian group, hence isomorphic to $\bb Z^m$. Since $\bar M_S$ is sharp\footnote{A monoid is sharp if $0$ is indecomposable: $a+b=0$ implies $a=b=0$.} and finitely generated, the non-zero elements of its image in $\bb Z^m$ land in some strict half-space of $\bb Z^m$ cut out by a linear equation with integral coefficients.  Such a half-space admits a map to $\bb N$ such that the only preimage of $0$ is $0$. 

After base changing over $S$ along such a map, we may assume 
that $\bar M_S^{gp} = \bb N$ and that all edges have non-zero label. 
We  obtain a first order map passing through our given point,
$$S = \on{Spec} {\field}[[t]] \to \cat{Div}_{g,A}\, $$
for which the induced prestable curve $C/S$ is generically smooth. On the curve $C$, we define a Weil divisor $D$ by assigning to an irreducible component $v$ the integer $\alpha(v)$. The divisor $D$ is then Cartier by \eqref{eq:slope}, which still applies after base-change, and hence determines a line bundle $\ca O_C(-D)$, which is exactly the image of the Abel-Jacobi map. In particular, the bundle $\ca L$ is (up to isomorphism) given by restricting $\ca O_C(-D)$ to the central fibre, so it suffices to verify \eqref{eq:twisted_diff_condition} for the latter bundle, which is a standard local calculation on a prestable curve over a discrete valuation ring.

%

Conversely, suppose the twisted divisor condition is satisfied. We must build a log structure and a suitable section $\alpha \in \bar M_C^{gp}(C)$. We could try equipping $(C/S, p_1, \dots, p_n)$ with its minimal log structure (see \cref{sec:prestable_vs_log_curves}), but then the section $\alpha$ is
unlikely to exist -- if there are no separating edges then there are no non-constant sections of $\bar M_C^{gp}$. Instead, we will construct a log structure by deforming the curve. 

First, we claim that there exists an assignment of a positive integer $\ell(e) \in \bb Z_{>0}$ to each edge and of an integer $d(v) \in \bb Z$ to each vertex such that the following condition is satisfied: 
  \begin{equation}
    \begin{array}{c}
         \text{\emph{if $e = \{ h, h'\}$ is an edge with $h$ attached to $u$ and $h'$ to $v$, then}} \\
          d(u) + I(h)\cdot \ell(e)  = d(v)\, .
    \end{array}\tag{*}
  \end{equation}
A twist $I$ on $\Gamma$ induces a binary relation $\preceq$ on $V(\Gamma)$ by 
\[u \preceq v \iff \text{there is an edge }e=\{h,h'\}\text{ with $h$ at $u$, $h'$ at $v$ and } I(h) \geq 0\, .\]
The fact that $\Gamma$ contains no strict cycles is equivalent by \cite{Suzumura1976Remarks-on-the-} to the existence of an extension of $\preceq$ to a total preorder on $V(\Gamma)$ (a reflexive, total, and transitive binary relation). Hence there exists a level function $d_0 : \V(\Gamma) \to \mathbb{Z}$ such that 
\begin{equation} \label{eqn:levelfunct} u \preceq v \iff u, v\text{ connected by an edge and }d_0(u) \leq d_0(v)\, .\end{equation}
We define
\[L = \mathrm{lcm}(I(h) : h \in H(\Gamma), I(h) > 0)\, .\]
Then, $d(v) = L d_0(v)$ still has property (\ref{eqn:levelfunct}), and,
for any edge $e = \{ h, h'\}$  with $h$ attached to $u$ and $h'$ to $v$, we have two cases:
\begin{itemize}
    \item $I(h) = 0$, in which case all edges $\{\tilde h, \tilde h'\}$ connecting $u,v$ must satisfy $I(\tilde h)=0$ (due to the strict cycle condition), so we can set $\ell(e)=1$,
    \item $I(h) \neq 0$, in which case the number $\ell(e) = (d(v) - d(u))/I(h)$ is indeed a positive integer (since $d$ has values in $L \cdot \mathbb{Z}$).
\end{itemize}
Clearly the functions $d$ and $\ell$ thus constructed satisfy the condition above.

Such $d$ and $\ell$ are far from unique, but we pick them. Consider then the space of all smoothings $\ca C$ of $C$ over ${\field}[[t]]$ such that the thickness\footnote{The local equation of the node is $xy = t^r$ for some positive integer $r$ which we call the \emph{thickness} of the node. } of $\ca C$ at the node corresponding to edge $e$ is $\ell(e)$. Given such a smoothing $\ca C$, we construct a vertical Weil divisor $D$ by assigning to the irreducible component corresponding to vertex $v$ the weight $d(v)$. The divisor $D$ is then Cartier by the condition (*). Set $$\ca L_{\ca C} = \ca O_{\ca C}(D)|_{C}\otimes \ca O_C(\sum_i a_i p_i)\, .$$
The smoothing $\ca C$ also induces a log structure on $C$ by  taking the divisorial log structure of the special fibre. The twist $I$ then determines an element $\alpha$ of $\pi_*\bar M_C^{gp}/\bar M_S^{gp}$. Applying the Abel-Jacobi map to $\alpha$ recovers $\ca L_{\ca C}$. 

One can readily verify that $\ca L_{\ca C}$ satisfies \eqref{eq:twisted_diff_condition} by a local computation, but we need to show more:  the smoothing $\ca C$ can be chosen so that $\ca L_{\ca C}$ is isomorphic to the line bundle $\ca L$ that we started with. The space of such smoothings $\ca C$ naturally surjects onto 
\begin{equation}
\bigoplus_{e = \{ h, h'\} : I(h) \neq 0} \left( \ca O_{C_I}(h) \otimes \ca O_{C_I}(h')\right)^{\otimes \ell(e)}, 
\end{equation}
where the 1-dimensional $\field$-vector space $\left( \ca O_{C_I}(h) \otimes \ca O_{C_I}(h')\right)^{\otimes \ell(e)}$ corresponds exactly to the ways to glueing the two branches of $\eta_I^*\ca L$ together at the points $h$ and $h'$. In other words, by moving over the space $\bigoplus_{e = \{ h, h'\} : I(h) \neq 0} \left( \ca O_{C_I}(h) \otimes \ca O_{C_I}(h')\right)^{\otimes \ell(e)}, $ we can recover \emph{all} ways of glueing $\eta_I^*\ca L$ to a line bundle on $C$. In particular, we can recover $\ca L$, hence we can realise $\ca L$ as $\ca L_{\ca C}$ for some smoothing $\ca C$, as required. 
\end{proof}

\subsection{Proof of the equivalence of the definitions}\label{sec:proof_of_equiv_defs}

The equivalence of the classes coming from $\cat{Div}_{g,A}$ and from $\cat{Rub}_{g,A}$ is immediate by applying \cref{prop:invarianceproperbirat}. We must compare 
the latter two with the class defined by  \eqref{ddd111}
via the schematic image. We will require the following two easy results. 
\begin{lemma}\label{lem:smooth_dense}
Let $U \hra \cat{Div}_{g,A}$ denote the open locus where the log curve is classically smooth. Then $U$ is schematically dense in $\cat{Div}_{g,A}$. 
\end{lemma}
\begin{proof}
Since $\cat{Div}_{g,A} \to \ca M^{\mathsf{log}}_{g,n}$ is log \'etale, we deduce that $\cat{Div}_{g,A}$ is log regular. In particular, $\cat{Div}_{g,A}$
 is reduced,  and the locus where the log structure is trivial is dense. 
%
\end{proof}


\begin{lemma} \label{Lem:UschemdenseDiv}
The Abel-Jacobi map $\aj:\cat{Div}_{g,A} \to \Picabs_{g,n,d}$ factors through the
inclusion $\bar\sigma \to \Picabs_{g,n,d}$, and the induced map 
$$\cat{Div}_{g,A} \to \bar \sigma$$ 
is proper and birational. 
\end{lemma}
\begin{proof}
That $\cat{Div}_{g,a} \to \Picabs_{g,n,d}$ factors through $\bar\sigma \to \Picabs$ is immediate from \cref{lem:smooth_dense} and the definition of the schematic image. The induced map $\cat{Div}_{g,A} \to \bar \sigma$ is proper since $\cat{Div}_{g,A}$ is proper over $\Picabs_{g,n,d}$ and is birational since it is an isomorphism over the locus of smooth curves. 
\end{proof}

By another application of \cref{prop:invarianceproperbirat}, the definitions of $\DRop$ via 
$\cat{Div}_{g,A}$ and the
schematic image are equivalent. \qed

\subsection{Proof of \cref{cccc}}\label{proof1111}

Let $k\geq 0$, and let
 $A=(a_1,\ldots,a_n)$ be a vector of integers satisfying
$$\sum_{i=1}^n a_i=k(2g-2)\, .$$
There are three definitions in the literature for the classical twisted double ramification cycle
$$\mathsf{DR}_{g,A, \omega^k}\in \mathsf{CH}_{2g-3+n}(\overline{\cM}_{g,n}) $$
on the moduli space of stable curves:
\begin{enumerate}
\item[$\bullet$]via birational modifications of $\Mbar_{g,n}$ \cite{Holmes2017Extending-the-d}, 
\item[$\bullet$] via the closure of the image of the Abel-Jacobi section \cite{Holmes2017Jacobian-extens},
\item[$\bullet$]
via logarithmic geometry and the stack $\cat{Div}_g^{\mathsf{rel}}$ \cite{Marcus2017Logarithmic-com}. 
\end{enumerate}
All three are shown to be equivalent
in \cite{Holmes2017Extending-the-d,Holmes2017Jacobian-extens}.
For the proof of \cref{cccc}, we choose the definition of \cite{Marcus2017Logarithmic-com}, as this will give the shortest path. 

For $d=k(2g-2)$, let $\phi: \Mbar_{g,n} \to \Picabs_{g,n,d}$ be the morphism
associated to the data
\begin{equation}\label{fftt66673}
\pi: \mathcal{C}_{g,n} \rightarrow \overline{\cM}_{g,n}\, , \ \ \ \ \omega_\pi^k \rightarrow \mathcal{C}_{g,n}\, .
\end{equation}
To prove \cref{cccc}, we must show
$$\DRop_{g,A}(\phi)([\Mbar_{g,n}]) = \mathsf{DR}_{g,A, \omega^k}\, ,$$
where $[\Mbar_{g,n}]$ is the fundamental class.

We form the pullback diagram
\begin{equation}
 \begin{tikzcd}
  \cat{Div}_{g,A} \times_{\Picabs_{g,n,d}} \Mbar_{g,n} \arrow[r] \arrow[d, "\psi"] & \cat{Div}_{g,A} \times \Mbar_{g,n} \arrow[d, "a \times \on{id}"]  \\
  \Mbar_{g,n} \arrow[r, "\phi'=\phi\times \on{id}"] & \Picabs_{g,n,d} \times \Mbar_{g,n}\, .
\end{tikzcd}
\end{equation}
Following \cref{defbivariantclass}, we have 
$$\DRop_{g,A}(\phi)([\Mbar_{g,n}]) = \psi_*(\phi')^!([\cat{Div}_{g,A} \times \Mbar_{g,n}])\, .$$
The construction
is equivalent to the definition of the class $\mathsf{DR}_{g,A, \omega^k}$ in \cite{Marcus2017Logarithmic-com} after making the standard translation between the Gysin pullback and the virtual fundamental class as in \cite[Example 7.6]{Behrend1997The-intrinsic-n}.

\subsection{The double ramification cycle in b-Chow} 
\label{Sect:AJextension}

The construction of the double ramification cycle in \cite{Holmes2017Extending-the-d} naturally yielded a more refined object: a b-cycle\footnote{An element of the colimit of the Chow groups of smooth blowups of $\Mbar_{g,n}$ with transition maps given by pullback.} on $\Mbar_{g,n}$ which pushes down to the usual double ramification cycle on $\Mbar_{g,n}$. 
The refined cycle was shown in \cite{Holmes2017Multiplicativit} to have better properties with respect to intersection products than the usual double ramification cycle. By considering rational sections of the multidegree-zero relative Picard space over $\Picabs_{g,n,d}$, we can in an analogous way define a b-cycle on $\Picabs_{g,n,d}$ refining the universal twisted double ramification cycle introduced here. In future work, we will show that this refined universal cycle is compatible with intersection products in the sense of \cite{Holmes2017Multiplicativit} and that the \emph{toric contact cycles} of \cite{Ranganathan2019A-note-on-cycle} can be obtained by pulling back these products.

\section{Pixton's formula}\label{sec:univeral_pixton}
\subsection{Reformulation}

Recall the cycle $\P_{g,A,d}^{c} \in \mathsf{CH}^c_{\mathsf{op}}(\Picabs_{g,n,d})$
defined in Section \ref{Ssec:MainFormula}. We write
$$\P_{g,A,d}^{\bullet} =\sum_{c=0}^\infty \P_{g,A,d}^{c}
\in \prod_{c=0}^\infty \mathsf{CH}^c_{\mathsf{op}}(\Picabs_{g,n,d})\, $$
for the associated mixed dimensional class.
We will rewrite the formula for $\P_{g,A,d}^{\bullet}$ in a 
more convenient form for computation.

Several factors in the formula of Section \ref{Ssec:MainFormula} can be pulled out of the sum over graphs and weightings. 
We require the following four definitions:
\begin{enumerate}
\item[$\bullet$] Let $\G_{g,n,d}^\textup{se}$ be the set of graphs in $\G_{g,n,d}$ having exactly two vertices connected by a single edge. 
Such graphs are thus described by a partition 
\begin{equation}\label{cc993}
(g_1, I_1, \delta_1\, |\,  g_2, I_2, \delta_2)
\end{equation}
of the genus, the marking set, and the degree of the universal line bundle. 
\item[$\bullet$]
Given a vector $A =(a_1,\ldots,a_n) \in \mathbb{Z}^n$ satisfying 
$$\sum_{i=1}^n a_i = d$$ and 
$\Gamma_{\delta} \in \G_{g,n,d}^\textup{se}$ corresponding to the partition \eqref{cc993}, we define
\[c_A(\Gamma_\delta) = - (\delta_1 - \sum_{i \in I_1} a_i)^2 =  -(\delta_2 - \sum_{i \in I_2} a_i)^2\, .\]
\item[$\bullet$] For $\Gamma_\delta \in \G_{g,n,d}^\textup{se}$, we write $$[\Gamma_\delta] = \frac{1}{|\Aut(\Gamma_\delta)|} j_{\Gamma_\delta*} [\Picabs_{\Gamma_{\delta}}]$$ for the class of the boundary divisor of $\Picabs_{g,n,d}$ associated to $\Gamma_\delta$. 
\item[$\bullet$]
Let $\G_{g,n,d}^\textup{nse}$ be the set of graphs in $\G_{g,n,d}$ such that every edge is non-separating. 
\end{enumerate}

\begin{proposition} \label{Lem:Pixformulafactorization}
The class $\P_{g,A,d}^{\bullet}$ is the constant term in $r$ of
\begin{align} \label{eqn:Pixformulafactorization}
&\exp\left(\frac12 \left( -\eta + \sum_{i=1}^n 2 a_i \xi_i+ a_i^2 \psi_i  + \sum_{\Gamma_\delta \in \G_{g,n,d}^\textup{se}} c_A(\Gamma_\delta) [\Gamma_\delta] \right) \right) \\ &\sum_{
\substack{\Gamma_\delta\in \G_{g,n,d}^\textup{nse} \\
w\in \mathsf{W}_{\Gamma_\delta,r}}
}
\frac{r^{-h^1(\Gamma_\delta)}}{|\Aut(\Gamma_\delta)| }
\;
j_{\Gamma_\delta*}\Bigg[
\prod_{e=(h,h')\in \E(\Gamma_\delta)}
\frac{1-\exp\left(-\frac{w(h)w(h')}2(\psi_h+\psi_{h'})\right)}{\psi_h + \psi_{h'}} \Bigg]\, , \nonumber
\end{align}
for $r \gg 0$.
\end{proposition}
In the proof of Proposition \ref{Lem:Pixformulafactorization}, we will use the following computation which provides an interpretation for parts of the formula \eqref{eqn:Pixformulafactorization} and which will be used again in \cref{sec:applications}. 
\begin{lemma} \label{Lem:vertterm}
Let $A = (a_1,\ldots,a_n) \in \mathbb{Z}^n$
with $\sum_{i=1}^n a_i=d$.
 For the line bundle $\mathcal{L}$ on the universal curve $$\pi:\mathfrak{C}_{g,n,d} \to \Picabs_{g,n,d}$$
 with universal sections $p_1, \ldots, p_n$, we define $$\mathcal{L}_A = \mathcal{L}\Big(- \sum_{i=1}^n a_i [p_i]\Big)\, .$$ Then, we have
    \[- \pi_* c_1(\mathcal{L}_A)^2=
    -\eta + \sum_{i=1}^n 2 a_i \xi_i+ a_i^2 \psi_i\, .\]
\end{lemma}
\begin{proof} The result follows from the definitions of the classes $\eta$ and $\xi_i$:
\begin{align*}
- \pi_* c_1(\mathcal{L}_A)^2 &= - \pi_* \left(c_1(\ca L)^2 +  \sum_{i=1}^n - 2 a_i c_1(\ca L)|_{[p_i]} + a_i^2 [p_i]^2  \right) \\
&=-\eta + \sum_{i=1}^n 2 a_i \xi_i+ a_i^2 \psi_i\,,
\end{align*}
where, for the self-intersection $[p_i]^2$, we have used that the first Chern class of the normal bundle of $p_i$ is given by $- \psi_i$.
\end{proof}

\begin{proof}[Proof of \cref{Lem:Pixformulafactorization}]
We denote by
\begin{align*} 
\Phi_{a}(x)&= \frac{1-\exp({-\frac{a}{2}}x)}{x}\\ &= \sum_{m=0}^\infty (-1)^m (\frac{a}{2})^{m+1} \frac{1}{(m+1)!}x^m = \frac{a}{2} - \frac{a^2}{8} x + \ldots 
\end{align*}
the power series appearing in the edge-terms of Pixton's formula. 

As a first step, we show that the constant term in $r$ of
\begin{equation} \label{eqn:Graphfactorization1}
    \exp\left(\frac12  \sum_{\Gamma_\delta \in \G_{g,n,d}^\textup{se}} c_A(\Gamma_\delta) [\Gamma_\delta]  \right) \cdot \sum_{
\substack{\Gamma_\delta\in \G_{g,n,d}^\textup{nse} \\
w\in \mathsf{W}_{\Gamma_\delta,r}}
}
\frac{r^{-h^1(\Gamma_\delta)}}{|\Aut(\Gamma_\delta)| }
\;
j_{\Gamma_\delta*} \prod_{e=(h,h')\in \E(\Gamma_\delta)}
 \Phi_{w(h)w(h')}(\psi_h+\psi_{h'})
\end{equation}
and the constant term in $r$ of
\begin{equation} \label{eqn:Graphfactorization2}
\sum_{
\substack{\Gamma_\delta\in \G_{g,n,d} \\
w\in \mathsf{W}_{\Gamma_\delta,r}}
}
\frac{r^{-h^1(\Gamma_\delta)}}{|\Aut(\Gamma_\delta)| }
\;
j_{\Gamma_\delta*} \prod_{e=(h,h')\in \E(\Gamma_\delta)}
 \Phi_{w(h)w(h')}(\psi_h+\psi_{h'})
\end{equation}
are equal. The formula \eqref{eqn:Graphfactorization2} is a linear combination of boundary strata decorated by edge-terms $(\psi_h + \psi_{h'})^{m(e)}$ for nonnegative integers $m(e)$, $e \in \E(\Gamma_\delta)$ -- terms of the form
\begin{equation} \label{eqn:Graphfactorization3}
 j_{\Gamma_\delta*} \prod_{e=(h,h')\in \E(\Gamma_\delta)}
 (\psi_h+\psi_{h'})^{m(e)}\, .
\end{equation}
A first consequence of the combinatorial rules for computing intersections in the tautological ring\footnote{See \cite{GraberPandharipande} for the original treatment of the tautological ring of $\overline{\mathcal{M}}_{g,n}$. A corresponding treatment for $\mathfrak{M}_{g,n}$ will be given in \cite{BaeSchmitt1, BaeSchmitt2}. See also \cite[Sections 1.1, 1.7]{Janda2018Double-ramifica}.} of $\Picabs_{g,n,d}$ is that \eqref{eqn:Graphfactorization1} is also a linear combination of such terms. The decorations $(\psi_h+\psi_{h'})^{m(e)}$ on separating edges $e=(h,h')$ appear naturally in the self-intersection formula for the boundary divisors $[\Gamma_\delta]$ since, for $\Gamma_\delta \in \G_{g,n,d}^\textup{se}$, the Chern class of the normal bundle of $j_{\Gamma_\delta}$ is given by $- (\psi_h + \psi_{h'})$.

We show that the coefficients of the term \eqref{eqn:Graphfactorization3} in \eqref{eqn:Graphfactorization1} and \eqref{eqn:Graphfactorization2} have the same constant term in $r$. In \eqref{eqn:Graphfactorization2}, the coefficient is given by
\begin{equation} \label{eqn:Graphfaccoeff1}
    \sum_{
w\in \mathsf{W}_{\Gamma_\delta,r}}
\frac{r^{-h^1(\Gamma_\delta)}}{|\Aut(\Gamma_\delta)| }
\prod_{e=(h,h')\in \E(\Gamma_\delta)}
 (-1)^{m(e)} \left(\frac{w(h)w(h')}{2}\right)^{m(e)+1} \frac{1}{(m(e)+1)!}\, .
\end{equation}
On the other hand, let $e_1, \ldots, e_\ell \in E(\Gamma_\delta)$ be the separating edges of $\Gamma_\delta$, and let $\overline{\Gamma}_{\delta} \in \G_{g,n,d}^\textup{nse}$ be the graph obtained from $\Gamma_\delta$ by contracting these separating edges. Each separating edge $e_j$ corresponds to a unique graph $(\Gamma_j)_{\delta_j} \in \G_{g,n,d}^\textup{se}$ obtained by contracting all edges of $\Gamma_\delta$ except for $e_j$. 

In the product \eqref{eqn:Graphfactorization1}, the intersection rules of the tautological ring of $\Picabs_{g,n,d}$ imply that we obtain multiples of the term \eqref{eqn:Graphfactorization3} by combining 
\begin{itemize}
    \item for $j=1, \ldots, \ell$, a total of $m(e_j)+1$ terms $[(\Gamma_j)_{\delta_j}]$ from expanding the power series \[\exp\left(\frac12  \sum_{\Gamma_\delta \in \G_{g,n,d}^\textup{se}} c_A(\Gamma_\delta) [\Gamma_\delta]  \right),\]
    \item the terms associated to $\overline{\Gamma}_{\delta} \in \G_{g,n,d}^\textup{nse}$ in the second factor.
\end{itemize}
Let $M=\sum_{j=1}^\ell (m(j)+1)$, then \eqref{eqn:Graphfactorization3} appears in \eqref{eqn:Graphfactorization1} with coefficient
\begin{align} \label{eqn:Graphfaccoeff2}
    \frac{1}{M!}\binom{M}{m(e_1)+1, \ldots, m(e_\ell)+1} \cdot \left( \prod_{j=1}^\ell \left(\frac{c_A((\Gamma_j)_{\delta_j})}{2}\right)^{m(e_j)+1} (-1)^{m(e_j)} \right)\\ \nonumber \cdot \frac{|\Aut(\overline{\Gamma}_\delta)|}{|\Aut({\Gamma}_\delta)|} \sum_{
w\in \mathsf{W}_{\overline{\Gamma}_\delta,r}}
\frac{r^{-h^1(\overline{\Gamma}_\delta)}}{|\Aut(\overline{\Gamma}_\delta)| }
\prod_{e=(h,h')\in \E(\overline{\Gamma}_\delta)}
 (-1)^{m(e)} (\frac{w(h)w(h')}{2})^{m(e)+1} \frac{1}{(m(e)+1)!}\, .
\end{align}
To show the equality of \eqref{eqn:Graphfaccoeff1} and \eqref{eqn:Graphfaccoeff2}, we combine a number of observations. First, for the multinomial coefficients, we have
\[\frac{1}{M!}\binom{M}{m(e_1)+1, \ldots, m(e_\ell)+1} = \prod_{j=1}^\ell \frac{1}{(m(e_j)+1)!}\, .\]
Second, for the graph morphism $\Gamma_\delta \to \overline{\Gamma}_\delta$ contracting the separating edges:
\begin{itemize}
    \item we have an equality of Betti numbers $h^1(\Gamma_\delta) = h^1(\overline{\Gamma}_\delta)$,
    \item for the separating edges $e_j=(h_j, h_j')$ of  $\Gamma_\delta$, splitting the graph according to the partition $(g_1, I_1, \delta_1\,  | \, g_2, I_2, \delta_2)$, the value of every weighting $w\in \mathsf{W}_{\Gamma_\delta,r}$ is uniquely determined on $h_j, h_j'$ since 
    \[w(h_j) = \delta_1 - \sum_{i \in I_1} a_i   \mod r\, ,\ \ \  w(h_j') = \delta_2 - \sum_{i \in I_2} a_i   \mod r.\] Hence, the constant term in $r$ of $w(h_j) w(h_j')$ is precisely given by $c_A((\Gamma_j)_{\delta_j})$. 
    \item concerning the non-separating edges for fixed $\Gamma_\delta$ with contraction $\Gamma_\delta \to \overline{\Gamma}_\delta$, the map $W_{\Gamma_\delta,r} \to W_{\overline{\Gamma}_\delta,r}$ given by restricting weightings $w \in W_{\Gamma_\delta,r}$ to the remaining half-edges $H(\overline{\Gamma}_\delta) \subset H({\Gamma}_\delta)$ is a bijection.
\end{itemize}
The combination of these facts proves equality of \eqref{eqn:Graphfaccoeff1} and \eqref{eqn:Graphfaccoeff2} and hence the equality of \eqref{eqn:Graphfactorization1} and \eqref{eqn:Graphfactorization2}. 

To conclude the proof, we must show that the remaining part of the exponential term of \eqref{eqn:Pixformulafactorization} can be drawn into the graph sum. Using the projection formula, this identity is equivalent to showing
\begin{multline*}
(j_{\Gamma_\delta})^* \exp \left( \frac{1}{2} \left(-\eta + \sum_{i=1}^n 2 a_i \xi_i+ a_i^2 \psi_i \right) \right) = \\ \prod_{v \in \V(\Gamma_\delta)} \exp\left(-\frac12 \eta(v) \right) \prod_{i=1}^n \exp\left(\frac12 a_i^2 \psi_i + a_i \xi_i \right)
,
\end{multline*}
which immediately reduces to showing
\[(j_{\Gamma_\delta})^*  \left(-\eta + \sum_{i=1}^n (2 a_i \xi_i+ a_i^2 \psi_i) \right) = -\sum_{v \in \V(\Gamma_\delta)} \eta(v) + \sum_{i=1}^n (2 a_i \xi_i + a_i^2 \psi_i) 
\, .\]
By \cref{Lem:vertterm},
    \[-\eta + \sum_{i=1}^n (2 a_i \xi_i+ a_i^2 \psi_i) = - \pi_* c_1(\mathcal{L}_A)^2.\]
Now consider the diagram of universal curves
\[
\begin{tikzcd}
\coprod_{v \in \V(\Gamma)}\frak C_{\g(v),\n(v),\delta(v)}  \arrow[d] &\frak C'_{\Gamma_{\delta}} \arrow[l] \arrow[r,"G"] \arrow[d,"\pi'_{\Gamma_\delta}"] &\frak C_{\Gamma_{\delta}} \arrow[r, "J_{\Gamma_{\delta}}"] \arrow[d,"\pi_{\Gamma_\delta}"] & \frak C_{g,n,d} \arrow[d, "\pi"]\\
\prod_{v \in \V(\Gamma)}\Picabs_{\g(v),\n(v),\delta(v)} & \Picabs_{\Gamma_{\delta}}\arrow[l]&\Picabs_{\Gamma_{\delta}} \ar[equal]{l}\arrow[r, "j_{\Gamma_{\delta}}"] & \Picabs_{g,n,d}
\end{tikzcd}    
\]
where the left and right square are cartesian and the map $G$ is the gluing map identifying sections of $\frak C'_{\Gamma_{\delta}} \to \Picabs_{\Gamma_\delta}$ corresponding to pairs of half-edges forming an edge. This map $G$ is proper, representable, and birational. 

The space $\frak C'_{\Gamma_{\delta}}$ is a disjoint union of universal curves 
$$\pi'_{\Gamma_\delta,v} : \frak C'_{\Gamma_{\delta},v} \to \Picabs_{\Gamma_\delta}$$
for $v \in \V(\Gamma)$ and the bundle  $G^* J_{\Gamma_\delta}^* \ca L_A$ restricted to the component $\frak C'_{\Gamma_{\delta},v}$ is equal to the pullback of the line bundle $\ca L_{v, A_v}$ from the factor $C_{\g(v),\n(v),\delta(v)}$ (where $A_v$ is the vector formed by numbers $a_i$ for $i$ a marking at $v$, extended by $0$ on the half-edges at $v$). Then using the projection formula together with \cref{prop:invarianceproperbirat}, we have
    $$(j_{\Gamma_\delta})^* \pi_* c_1(\ca L_A)^2 = (\pi_{\Gamma_\delta})_* J_{\Gamma_\delta}^* c_1(\ca L_A)^2 = (G \circ \pi_{\Gamma_\delta})_* (G \circ \pi_{\Gamma_\delta})^*c_1(\ca L_A)^2$$
    $$= \sum_{v \in \V(\Gamma_\delta)}(\pi'_{\Gamma_\delta,v})_* c_1(\ca L_{v, A_v})^2 = \sum_{v \in \V(\Gamma_\delta)} \eta(v) + \sum_{i=1}^n  a_i^2 \psi_i + 2 a_i \xi_i\, ,$$
where for the last equality we again use \cref{Lem:vertterm}.
\end{proof}

In the case $n=0$ and $d=0$, the formula $\P_{g,\emptyset,0}^{\bullet}$ takes a slightly simpler shape: it is the $r=0$ term of the formula
\begin{align} \label{eqn:P_gformula2}
\exp\left(-\frac12 \eta \right) \sum_{
\substack{\Gamma_\delta\in \G_{g,0,0} \\
w\in \mathsf{W}_{\Gamma_\delta,r}}
}
\frac{r^{-h^1(\Gamma_\delta)}}{|\Aut(\Gamma_\delta)| }
\;
j_{\Gamma_\delta*}\Bigg[&
\prod_{e=(h,h')\in \E(\Gamma_\delta)}
\frac{1-\exp\left(-\frac{w(h)w(h')}2(\psi_h+\psi_{h'})\right)}{\psi_h + \psi_{h'}} \Bigg]\, . 
\end{align}
As explained in \cref{sec:intro_deg_0}, the full formula $\P_{g,A,d}^{\bullet}$ can be reconstructed from $\P_{g,\emptyset,0}^\bullet$.

\subsection{Comparison to Pixton's \texorpdfstring{$k$}{k}-twisted  formula}\label{sec:pixton_ktwist_vs_univ}
Given $k \geq 0$ and a vector $A =(a_1,\ldots,a_n) \in \mathbb{Z}^n$ satisfying $$\sum_i a_i = k(2g-2)\, ,$$
let $\tilde A = (\tilde{a}_1,\ldots,\tilde{a}_n)$ be the vector with entries $\tilde a_i = a_i + k$. Denote by  
\[P_g^{c ,k}(\tilde A) \in \Chow^c(\overline{\mathcal{M}}_{g,n})\] Pixton's original formula  defined in \cite[Section 1.1]{Janda2016Double-ramifica}.

In the $k=0$ case, $A=\tilde{A}$, and
$2^{-g} P_g^{g,0}(\tilde{A})$ is the
class originally conjectured by Pixton to
equal the double ramification cycle associated to the vector $\tilde A$. Compatibility with the formula for the universal twisted double ramification cycle is given by the
following result.

\begin{proposition} \label{Prop:PixtonPullbackCompatibility}
Via the map $\phi_{\omega_\pi^k}: \Mbar_{g,n} \to \Picabs_{g,n,k(2g-2)}$ associated to the universal data
\[\pi: \mathcal{C}_{g,n} \rightarrow \overline{\cM}_{g,n}\, , \ \ \ \ \omega_\pi^k \rightarrow \mathcal{C}_{g,n}\, ,\]
the class $\P_{g,A,k(2g-2)}^c$ acts as
\begin{equation} \label{eqn:PixtonPullbackCompatibility}
\P_{g,A,k(2g-2)}^c(\phi_{\omega_\pi^k})([\overline{\mathcal{M}}_{g,n}])=2^{-c} P_g^{c,k}(\tilde A)
\end{equation}
for every $c \geq 0$.
\end{proposition}
\begin{proof}
The left-hand side of \eqref{eqn:PixtonPullbackCompatibility} is obtained from $\P_{g,A,k(2g-2)}^c$ by substituting 
\begin{equation}\label{cc78}
    \mathcal{L} = \omega_\pi^{\otimes k}
    \end{equation}
in
the formula and taking the action. 
A factor $2^{-c}$ arises on the left side
since all terms in $\P_{g,A,k(2g-2)}^c$ increasing the codimension of the cycle naturally come with corresponding negative powers of $2$ (which is placed as
a prefactor on the right side in \cite[Section 1.1]{Janda2016Double-ramifica}).

Under the substitution \eqref{cc78}, 
the edge terms and weightings modulo $r$ in the two formulas 
naturally correspond to each other. 
Using \cref{Lem:Pixformulafactorization} and \cref{Lem:vertterm}, we must
show 
\[\exp \left(-\frac{1}{2} \pi_* c_1(\omega_\pi^{\otimes k}(- \sum_{i=1}^n a_i [p_i]) )^2 \right) = \exp \left(-\frac{1}{2} \left( k^2 \kappa_1 -  \sum_{i=1}^n \tilde a_i^2 \psi_i \right)\right),\]
where again $[p_i]$ denotes the class of the image of the section $p_i: \overline{\mathcal{M}}_{g,n} \to \mathcal{C}_{g,n}$. Defining $\omega_\pi^\textup{log} = \omega_\pi(\sum_i p_i)$, we see
\begin{align*}
c_1(\omega_\pi^{\otimes k}(- \sum_{i=1}^n a_i [p_i]) )^2
& = ( k c_1(\omega_{\log}) - \sum_{i=1}^n \tilde a_i [p_i] ) ^2 \\
& = k^2 c_1(\omega_{\log})^2 - 2 k \sum_{i=1}^n \tilde a_i c_1(\omega_{\log})|_{[p_i]} + \sum_{i=1}^n \tilde a_i^2 [p_i]^2
\end{align*}
After pushing forward, the first term gives $k^2 \kappa_1$, the second vanishes (since $\omega_{\log}$ restricts to zero on the section $p_i$), and the third gives $- \sum_i \tilde a_i^2 \psi_i$, as desired.
\end{proof}




\subsection{Comparison to Pixton's formula with targets}\label{sec:pixton_targets_vs_univ}
Let $X$ be a nonsingular projective variety over ${\field}$. The moduli space 
$\MbarX$ parametrizes stable maps 
$$f\colon (C,p_1,\ldots,p_n)\to X$$
from genus $g$, $n$-pointed curves $C$ to $X$ of degree $\beta \in H_2(X,\mathbb{Z})$. The moduli space carries a virtual fundamental class 
$$[\MbarX]^\textup{vir} \in \Chow_{\vdim(g,n,\beta)}(\MbarX)$$
where 
$$\vdim(g,n,\beta) = (\dim \,X -3)(1-g) + \int_{\beta} c_{1}(X) + n\, .$$
See \cite{Behrend1997The-intrinsic-n} for the construction of virtual fundamental classes.

Given the data of a line bundle $\ca L$ on $X$ and a vector $A=(a_1,\ldots, a_n) \in \mathbb{Z}^n$ satisfying
\[\int_{\beta}c_1(\ca L)=\sum_{i=1}^{n}a_i \, ,\]
a double ramification cycle 
\[\DR_{g,A}(X,\ca L) \in \Chow_{\vdim(g,n,\beta)-g}(\MbarX)\]
virtually compactifying the locus of maps $f\colon (C,p_1,\cdots,p_n)\to X$ with
\[f^* \mathcal{L} \cong \mathcal{O}_C \left(\sum_{i=1}^n a_i p_i \right)\]
is defined in \cite{Janda2018Double-ramifica}. 
Furthermore, the authors define the notion of tautological classes
inside the operational Chow ring $\CHop^*(\MbarX)$ of $\MbarX$.
The main result of \cite{Janda2018Double-ramifica} is a
Pixton  formula for a 
codimension $g$ tautological class   whose action on $[\MbarX]^\textup{vir}$ yields $\DR_{g,A}(X,\ca L)$.

We define a morphism 
\begin{equation*}
\phi_{\ca L}\colon \Mbar_{g,n,\beta}(X) \to \Picabs_{g,A,d}\, , \ \ \  f \mapsto \left(C, p_1, \ldots, p_n, f^*\ca L \right)\, . 
\end{equation*}
The compatibility result here is
\begin{equation}\label{eqn:pixton_vs_uni_pixton}
\DR_{g,A}(X,\ca L) = \phi_{\ca L}^* \P^g_{g,A,d} \Big( [\MbarX]^\textup{vir}\Big)\, .
\end{equation}
The equality follows by an exact matching of the definition of 
$\P^g_{g,\bd A,d}$ in Section \ref{Ssec:MainFormula}
(after pullback by $\phi_{\ca L}^*$) with the
Pixton formula in the main result of \cite{Janda2018Double-ramifica}.

In fact, the compatibility \eqref{eqn:pixton_vs_uni_pixton}
represented the starting point for our investigation of the 
universal twisted double ramification cycle here.

\section{Proof of \texorpdfstring{\cref{Thm:main}}{main theorem}}\label{sec:proof_of_main_theorem}
\subsection{Overview} \label{sec:proof_of_main_theorem_overview}
We prove here the main result of the paper:
for $A=(a_1,\ldots,a_n)\in \mathbb{Z}^n$  satisfying $$\sum_{i=1}^n a_i=d\, ,$$
the universal twisted double ramification cycle is calculated by Pixton's formula
$$\mathsf{DR}^{\mathsf{op}}_{g,A} =\P_{g,A,d}^g\
\in {\mathsf{CH}}^g_{\mathsf{op}}(\Picabs_{g,n,d})\, .$$

The result is an equality in the operational Chow group, and therefore
an equality on every finite type family of prestable curves. Given $C\to B$ a prestable curve and a line bundle $\ca L$ on $C$ of relative degree $d$, we obtain
a map $$\phi_{\ca L}\colon B \to \Picabs_{g,n,d}\, .$$
We must prove 
\begin{equation}\label{eq:main_conj}
\DRop_{g,A}(\phi_{\ca L})= {\P_{g,A,d}^g}(\phi_{\ca L}) : \Chow_*(B) \to \Chow_{*-g}(B)\, .
\end{equation}

As explained in \cref{sec:intro_deg_0},  the result for general $A \in \mathbb{Z}^n$ can be reduced  to the case $n=0, d=0$, though the case of arbitrary $A$ will be important in the proof as we proceed through a sequence of special cases. We recall that this reduction used the invariances \invnrn{} and \invnrA{} for the double ramification cycle and Pixton's formula. Note that these will be proved separately and independent of \cref{Thm:main} in \cref{Sect:proofInvariance}, so no circular reasoning occurs.

\subsection{On an open subset of $\Mbar_{g,n}(\mathbb{P}^l,\beta)$}\label{sec:projective_target}

\newcommand{\kk}{l}
As before, let $A=(a_1,\ldots,a_n) \in \bb Z^n$ with $$\sum_{i=1}^n a_i=d\,.$$ 
We consider here the target $X = \bb P^\kk$.
Let $\beta$ be the class of $d$ times a line in $\bb P^\kk$. 
Let 
$$\ca C \to \Mbar_{g,n}(\mathbb{P}^l,\beta)$$
be the universal curve over the moduli of stable maps to $\mathbb{P}^l$, let 
$$f\colon \ca C \to \mathbb{P}^l$$ be the universal map, and 
let $\ca L = f^*\ca O_X(1)$.

We have a tautological map
\begin{equation}
\phi_{\ca L} \colon \Mbar_{g,n}(\mathbb{P}^l,\beta) \to \Picabs_{g,n,d}. 
\end{equation}
We would like to prove an equality of operational classes 
$$\phi_{\ca L}^*\DRop_{g,A}=\phi_{\ca L}^*{\P_{g,A,d}^g}\, \in \mathsf{CH}_{\mathsf{op}}^g(\Mbar_{g,n}(\mathbb{P}^l,\beta))\, .$$
We will apply  the main result of \cite{Janda2018Double-ramifica} which
relates the double ramification cycle there to 
Pixton's formula. However, 
only the action of $\phi_{\ca L}^*{\P_{g,A,d}^g}$ on the
virtual fundamental class
$[ \Mbar_{g,n}(\mathbb{P}^l,\beta)]^\textup{vir}$ is computed in \cite{Janda2018Double-ramifica}. 
Since we are interested here in the full operational class $\phi_{\ca L}^*{\P_{g,A,d}^g}$,
our first idea is to restrict to the open locus 
$$\Mbar_{g,n}(\mathbb{P}^l,\beta)' \hra \Mbar_{g,n}(\mathbb{P}^l,\beta)$$ 
where (on each geometric fibre) we have $H^1(C, \ca L) = 0$. 

\begin{lemma}\label{lem:fun_eq_vfc}
On the smooth Deligne-Mumford stack $\Mbar_{g,n}(\mathbb{P}^l,\beta)'$, the fundamental and virtual fundamental classes coincide. 
\end{lemma}
\begin{proof}
It suffices to show that $H^1(C, f^*T_{\bb P^\kk}) = 0$ on $\Mbar_{g,n}(\mathbb{P}^l,\beta)'$. Pulling back the Euler exact sequence on $\bb P^\kk$ via $f$ yields 
\begin{equation}
0 \to \ca O_C \to \oplus_{1}^{\kk+1} f^*\ca O_{\bb P^\kk}(1) \to f^*T_{\bb P^\kk} \to 0\, .
\end{equation}

 Taking cohomology yields the exact sequence
\begin{equation}
\oplus_{1}^{\kk+1} H^1(C, f^*\ca O_{\bb P^\kk}(1)) \to H^1(C, f^*T_{\bb P^\kk}) \to H^2(C, \ca O_C)\, .
\end{equation}
But $H^1(C, f^*\ca O_{\bb P^\kk}(1)) = 0$ by assumption, and $H^2(C, \ca O_C) = 0$ for dimension reasons. 
\end{proof}

The next Lemma depends on a careful comparison of the logarithmic and rubber approaches to double ramification cycles, which will be postponed to Section \ref{sec:comparing_stacks}. 
\begin{lemma}\label{cor:operational_equality_projective}\label{lem:projective_space_comparison}
Let ${\phi'}_{\ca L}$ be the restriction of ${\phi}_{\ca L}$ to $\Mbar_{g,n}(\mathbb{P}^l,\beta)'$. We have an equality of operational classes 
\begin{equation} \label{eqn:phiprimepullbackDR}
{\phi'}_{\ca L}^*\DRop_{g,A}=
{\phi'}_{\ca L}^*{\P_{g,A,d}^g}  \, \in \, \CHop^g(\Mbar_{g,n}(\mathbb{P}^l,\beta)')\, .
\end{equation}
\end{lemma}
\begin{proof}

By \cref{lem:op_eq_ch_for_smooth}, the two sides of  \eqref{eqn:phiprimepullbackDR} are equal if and only if their actions on the fundamental class $[\Mbar_{g,n}(\mathbb{P}^l,\beta)']$ are equal in $\Chow_*(\Mbar_{g,n}(\mathbb{P}^l,\beta)')$. 
By  \eqref{eqn:pixton_vs_uni_pixton}, the action of the right side of \eqref{eqn:phiprimepullbackDR} on 
$$[\Mbar_{g,n}(\mathbb{P}^l,\beta)']=[\Mbar_{g,n}(\mathbb{P}^l,\beta)']^\textup{vir}$$ 
equals the restriction of $\DR_{g,A}(\mathbb{P}^l,\ca L)$ to 
$\Mbar_{g,n}(\mathbb{P}^l,\beta)'$.

The cycle $\DR_{g,A}(\mathbb{P}^l,\ca L)$
is defined in \cite{Janda2018Double-ramifica} as the pushforward of the virtual fundamental class of the space of rubber maps{\footnote{Rubber maps will
be discussed in \cref{sec:prestable_rubber_maps}.}}.  
By \cref{prop:vfc_comparison_summary} of 
\cref{sec:comparing_virtual_classes}, the restriction of
$\DR_{g,A}(X,\ca L)$ to $\Mbar_{g,n}(\mathbb{P}^l,\beta)'$ is equal to ${\phi'}_{\ca L}^*\DRop_{g,A}( [\Mbar_{g,n}(\mathbb{P}^l,\beta)'])$. 

\end{proof}

\subsection{For sufficiently positive line bundles}\label{sec:positive_line_bundle}

Let $\pi\colon C\to B$ be an $n$-pointed prestable curve over a scheme of finite type over ${\field}$. Let
$\ca L$ on $C$ be a line bundle 
of relative degree $d$. 
Let $$A = (a_1,\ldots, a_n) \in\mathbb{Z}^n$$ with $\sum_{i=1}^n a_i = d$.  
The line bundle $\ca L$ induces a map 
$$\phi_{\ca L} \colon B \to \Picabs_{g,n,d}\, .$$

We say $\ca L$ is \emph{relatively sufficiently positive} if $\ca L$ is relatively base-point free  and satisfies $R^1\pi_*\ca L = 0$.

\begin{lemma}\label{sec:very_ample}
Let $\ca L$ be a line bundle which is relatively sufficiently positive. 
Then we have an equality  
\begin{equation}\label{eq:very_ample}
\DRop_{g,A}(\phi_{\ca L})=
\P_{g,A,d}^g(\phi_{\ca L})  \colon \Chow_*(B) \to \Chow_{*-g}(B)\, .
\end{equation}

\end{lemma}
\begin{proof}

For any finite-type scheme $B$ the union of irreducible components of $B$ maps properly and surjectively to $B$. Thus the pushforward from the Chow groups of the irreducible components to that of $B$ is surjective, and hence it suffices to show the equality \eqref{eq:very_ample} of maps of Chow groups for $B$ irreducible. 

By relative sufficient positivity, 
$$R\pi_*\ca L = \pi_*\ca L$$ 
is a vector bundle on $B$ of rank $N$. 
For a positive integer $l$, we define
\begin{equation}\label{ppp12}
E_l =\bigoplus_1^{l+1} R\pi_*\ca L \, , 
\end{equation}
a vector bundle on $B$ of rank $r=N(l+1)$. 
Let $U_l \sub E_l$ denote the open locus of linear systems which are base-point free. Via pullback along $\psi\colon U_l \to B$, we obtain
 a map $$\psi^*\colon \mathsf{CH}_*(B) \to \mathsf{CH}_{*+r}(U_l)\, .$$

We claim that for $l>\dim B$, the pullback \eqref{ppp12}
is injective. 
To prove the injectivity, we show that the boundary $E_l \setminus U_l$ has codimension in $E_l$ greater than $\dim B$. Since $E_l \to B$ is flat with irreducible target, it suffices to bound the codimension  on each geometric fibre over $B$: for a  prestable curve $C/{\field}$ and a sufficiently
positive line bundle $\ca L$ on $C$, we must show that the locus in $\bigoplus_1^{l+1} H^0(C, \ca L)$ consisting of base point free linear systems has 
a complement of codimension 
greater than $\dim B$. 

Since $\ca L$ is base point free on $C$, the dimension of the locus in
$\bigoplus_1^{l+1} H^0(C, \ca L)$ where the linear system has
a base point at some given $p\in C$ is $(N-1)(l+1)$. Hence, as $p$ varies, the complement of the
base point free locus in $\bigoplus_1^ {l+1} H^0(C, \ca L)$ has dimension at
most $1 + (N-1)(l+1)$. So the codimension is at least
$$ N(l+1) - 1 - (N-1)(l+1)  = l \, .$$

We have a canonical map $g\colon C \times_B U_l \to \bb P^l$ with 
$g^*\ca O_{\bb P^l}(1) = \ca L$ 
which induces a map $$U_l \to \Mbar_{g,n}(\bb P^l,\beta)$$ which factors via the locus $$\Mbar_{g,n}(\bb P^l, \beta)' \subset \Mbar_{g,n}(\bb P^l, d)$$
where $H^1(C, f^*\ca O_{\bb P^l}(1)) = 0$. 
By construction, the composition $$U_l \xrightarrow{\psi} B \xrightarrow{\phi_{\ca L}} \Picabs_{g,n,d}$$ then factors through the map $ \Mbar_{g,n}(\bb P^l, \beta)' \to \Picabs_{g,n,d}$ induced by the line bundle $f^*\ca O_{\mathbb{P}^l}(1)$ as before. In other words, we have a commutative diagram
\[
\begin{tikzcd}
U_l \arrow[r] \arrow[d,"\psi"]& \Mbar_{g,n}(\bb P^l, \beta)' \arrow[d,"\phi_{f^*\ca O_{\mathbb{P}^l}(1)}"]\\
B \arrow[r,"\phi_{\ca L}"] & \Picabs_{g,n,d}\, . 
\end{tikzcd}
\]

\Cref{cor:operational_equality_projective} then implies that 
\begin{equation}
\left(\DRop_{g,A}-\P_{g,A,d}^g \right)(\phi_{\ca L}\circ \psi)\colon \Chow_*(U_l) \to \Chow_{* - g}(U_l)
\end{equation}
is the zero map, and we conclude the proof of the Lemma from the commutative diagram
\begin{equation}
 \begin{tikzcd}
  \Chow_{*+r}(U) \arrow[rrrr, "\left(\DRop_{g,A}-\P_{g,A,d}^g \right)(\phi_{\ca L}\circ \psi)"]&& && \Chow_{* + r-g}(U)\\
  \Chow_*(B) \arrow[rrrr, "\left(\DRop_{g,A}-\P_{g,A,d}^g\right)(\phi_{\ca L})"] \arrow[u,"\psi^*",hookrightarrow]& &&&\Chow_{* - g}(B) \arrow[u,"\psi^*",hookrightarrow]\, . 
\end{tikzcd}
\end{equation}

\end{proof}

\subsection{With sufficiently many sections}\label{sec:many_markings}

Let $\pi\colon C\to B$ be an $n$-pointed prestable curve 
with markings $p_1,\ldots, p_n$
over a scheme of finite type over ${\field}$. Let
$\ca L$ on $C$ be a line bundle 
of relative degree $d$. 
Let $$A = (a_1,\ldots, a_n) \in\mathbb{Z}^n$$ with $\sum_{i=1}^n a_i = d$.  
The line bundle $\ca L$ induces a map 
$$\phi_{\ca L} \colon B \to \Picabs_{g,n,d}\, .$$

\begin{lemma}\label{lem:many_markings}
For every geometric fibre of $C/B$, suppose the
complement of the union of irreducible components which carry markings
is a disjoint union of trees of nonsingular
rational curves on which $\ca L$ is trivial. Then we have an equality 
\begin{equation}\DRop_{g,A}(\phi_{\ca L})=
\P_{g,A,d}^g(\phi_{\ca L}) \colon \Chow_*(B) \to \Chow_{*-g}(B)\, .
\end{equation}
\end{lemma}

\begin{proof}
We can choose $A' = (a'_1,\dots, a'_n)$ with entries 
$$a'_i \gg 0\, , \ \ \ \sum_{i=1}^na_i' =d'\, $$
large enough so that $$\ca L' = \ca L\Big(\sum_{i=1}^n a'_ip_i\Big)$$
is relatively sufficiently positive (by Riemann-Roch for singular curves). 

We obtain an associated map 
$$\phi_{{\ca L}'} \colon B \to \Picabs_{g,n,d+d'}\, .$$
By \cref{sec:very_ample},
\begin{equation} \label{eqn:DRPequalityL42} \DRop_{g,A+A'}(\phi_{{\ca L}'})=
\P_{g,A+A',d+d'}^g(\phi_{{\ca L}'}) \colon \Chow_*(B) \to \Chow_{*-g}(B).
\end{equation}
Invariance \invnrA{} of \cref{Sect:introinvariance} (proven
in \cref{Sect:proofInvariance}) implies 
\begin{align*}
\DRop_{g,A+A'}(\phi_{{\ca L'}})
&=\DRop_{g,A}(\phi_{{\ca L}})\, ,\\ \P^g_{g,A+A',d+d'}(\phi_{{\ca L'}})
&=\P^g_{g,A,d}(\phi_{{\ca L}})\, ,     
\end{align*} 
which together with \eqref{eqn:DRPequalityL42} finishes the proof.
\end{proof}

\subsection{Proof in the general case}\label{sec:general_case}

To conclude the proof of \cref{Thm:main},
will use the invariances  of \cref{Sect:introinvariance}
(proven in \cref{Sect:proofInvariance}). As discussed in \cref{sec:proof_of_main_theorem_overview}, we can reduce to showing the result in the case $n=0, d=0$.\footnote{In genus $g=1$, we  follow a slightly modified strategy since there we must avoid the case $n=0$ for technical reasons (see Remark \ref{Rmk:genus1n0}). Instead, we can use the invariances to reduce to the case $g=1$, $n=1$, and $A=(0)$. All the proofs below generalize in a straightforward way since the vector $A=(0)$ does not affect the line bundles involved.}

Let $B$ be an irreducible scheme of finite type over ${\field}$.
Let $\pi\colon C\to B$ a prestable curve, and let
$\ca L$ on $C$ be a line bundle 
of relative degree $0$. 
The line bundle $\ca L$ induces a map 
$$\phi_{\ca L} \colon B \to \Picabs_{g,0,0}\, .$$

 \begin{lemma} There exists an \label{ALTt}
 alteration{\footnote{An alteration here is a proper, surjective, generically finite morphism between irreducible schemes.}}
 $B' \to B$ and a destabilisation
 \begin{equation}\label{kk3355}
 C' \to C \times_B B'
 \end{equation}
 such that $C'$ admits sections $p_1, \ldots, p_m$ and satisfies the
 following property:
 $$(C'/B', p_1, \ldots, p_m)$$ is a family of $m$-pointed prestable curves 
 and 
 for every geometric fibre of $C'/B'$,  the
complement of the union of irreducible components which carry markings
is a disjoint union of trees of nonsingular
rational curves which are contracted by the morphism \eqref{kk3355}.

\end{lemma}

\begin{proof}
We first claim, after an alteration 
$\widehat{B}\to B$,
there exists a multisection\footnote{By a multisection of $C_{\widehat{B}}\to \widehat{B}$,
we mean a closed substack $Z \subset C_{\widehat{B}}$ such that $Z \to \widehat{B}$ is finite and flat.} 
$$Z \subset \, C_{\widehat{B}}=C\times_B \widehat{B}\,  \to \widehat{B}$$ satisfying the following
two conditions:
\begin{enumerate}[label=\roman*)]
    \item[(i)] Over the generic point of $\widehat{B}$, $Z$ is contained in the smooth locus of $C_{\widehat{B}} \to \widehat{B}$.
    \item[(ii)] Every component of every geometric fibre of $C_{\widehat{B}}\to \widehat{B}$  carries at least two  distinct \'etale multisection points in the smooth locus. In other words, the \'etale locus of $Z \to \widehat{B}$ meets the smooth locus of every component of every geometric fibre of $C_{\widehat{B}} \to \widehat{B}$ in at least two points.
\end{enumerate}
To prove the above claim, we observe that for every geometric point $b$ of $B$ there exists an \'etale map $U_p \to B$ and a factorisation $U_p \to C$ whose image meets
the smooth locus of every irreducible component of every geometric fibre in some Zariski neighbourhood $V_b\sub B$ of $b$ at least twice. Choose a finite set of $b$ such that the $V_b$ cover $B$, define $U$ to be the union of the $U_b$, and define $Z'$ to be the closure of the image of $U$ in $C$.  Then $Z' \to B$ is proper and generically finite. 
Let $$\widehat B \to B$$ 
be a modification which flattens $Z'$ (see \cite{Raynaud-Gruson}), and let $Z$ be the strict transform of $Z'$ over $\widehat B$. Then $$Z \to \widehat B$$
is proper, flat, and generically finite, and hence  finite -- so condition (i) is
satisfied. Moreover, $U$ already satisfies condition (ii), and the strict transform of a flat map is just the fibre product, hence $Z$ also satisfies condition (ii). 

Let $\widetilde B \to \widehat B$ be an alteration such that over $\widetilde B$ the multisection $Z$ becomes a disjoint union of sections. 
In other words the pullback 
$$C_{\widetilde{B}}= C \times_B \widetilde B \to \widetilde B$$
has sections $\widetilde p_1, \ldots, \widetilde p_m$ such that, as a set, 
the preimage of $Z$ is given by the union of the images of sections $\widetilde p_1, \ldots, \widetilde p_m$. Such a $\widetilde B$ exists{\footnote{The base $\widehat B$ is excellent since it is finite type over a field.}}
by \cite[Lemma 5.6]{Jong1996Smoothness-semi}.
We can assume that the sections $\widetilde p_i$ are pairwise disjoint over the generic point of $\widetilde B$.

By assumption (i) above, the family $C_{\widetilde B} \to \widetilde B$ with sections $\widetilde p_1, \ldots, \widetilde p_m$ is generically a stable $m$-pointed curve (since 
every component has at least \emph{two} of the sections). We therefore obtain
a rational map $$\widetilde B \dashrightarrow \Mbar_{g,m}\, .$$
Let $B' \to \widetilde B$ be a blow-up resolving the indeterminacy of this map\footnote{As usual, this blowup is constructed by taking the closure of the graph and flattening. Then we check that this ensures the existence of the map to $C_{B'}$ as written below.}

\begin{equation}
\begin{tikzcd}
B' \arrow[d] \arrow[r] & \Mbar_{g,m} \\
\widetilde B \arrow[ru, dashed]&
\end{tikzcd}    
\end{equation}
 and let $C' \to B'$ with sections $p_1, \ldots, p_m : B' \to C'$
 be the pullback of the universal curve over $\Mbar_{g,n}$ to $B'$. Let 
 $$C_{B'} = C \times_B B'$$ 
 be the pullback of $C/B$ under $B' \to \widetilde B \to \widehat B \to B$. Then we have a map $f: C' \to C_{B'}$ fitting in a commutative diagram
 \begin{equation} \label{eqn:partstabilizationalteration}
 \begin{tikzcd}
 C' \arrow[rr,"f"] \arrow[rd] & & C_{B'} \arrow[ld]\\
 & B' \arrow[ul,"p_i", bend left] \arrow[ur, "\widetilde p_i", bend right, swap] &
 \end{tikzcd}
 \end{equation}
 such that $f$ is a partial destabilization. On geometric fibers of $C'\to B'$,
 $f$ collapses trees of rational curves to either nodes or coincident
  sections $\widetilde p_i$ on
 the geometric fibers of $C_{B'}$.
 
 To conclude, we must show that for every geometric point $b \in B'$ and every irreducible component $D \subset C'_b$ which is not contracted by $f$, we can 
 find a marking 
 $$p_i(b)\in D \, .$$ 
 The image of $D$ under $f$ is a component of $(C_{B'})_b$.
 By condition (ii) above, $f(D)$
 has at least one $\widetilde p_i(b)$ in the smooth locus of $f(D)$
 pairwise distinct from all other $\widetilde p_j(b)$.
 Since there are no components of $C'_b$ which collapse to $\widetilde p_i(b)$, we
 must have $p_i(b) \in D$.
 \end{proof}

\begin{lemma} \label{Lem:DR_alteration_invar}
We have 
$\DRop_{g,\emptyset}(\phi_{\ca L})=
\P_{g,\emptyset,0}^g(\phi_{\ca L}) :\Chow_*(B) \to \Chow_{*-g}(B)$. 
\end{lemma}
\begin{proof}
We apply Lemma \ref{ALTt} to the family $C/B$ to obtain
$$h\colon B' \to B\, , \ \ \ \ C' \to C_{B'}\, .$$ 
Let $\ca L'$ be the pullback of $\ca L$ to $C'$.
After applying \cref{lem:many_markings} with $A = \bd 0 \in \mathbb{Z}^m$, we obtain
\begin{equation}\label{eq:an_equality}
\DRop_{g,\bd 0}(\phi_{\ca L'})=\P_{g,\bd 0,d}^g(\phi_{\ca L'}):
\Chow_*(B') \to \Chow_{*-g}(B')\, .
\end{equation}

Since $h$ is proper and surjective, for any $\alpha \in \Chow_*(B)$ there exists $\alpha'\in \Chow_*(B')$ satisfying $h_*\alpha'=\alpha$. If any operational class maps $\alpha'$ to $0$, then it maps $\alpha$ to $0$ because the operation 
commutes with $h_*$. 

It therefore suffices to prove  
\begin{equation}
\left( \DRop_{g,\emptyset}-\P_{g,\emptyset,0}^g \right) (\phi_\ca L) \circ h_*
\end{equation}
is the zero map on $\Chow_*(B')$. By the compatibilities of operational classes we have
\begin{equation*}
\left( \DRop_{g,\emptyset}-\P_{g,\emptyset,0}^g \right) (\phi_\ca L) \circ h_* = h_* \left( \DRop_{g,\emptyset}-\P_{g,\emptyset,0}^g \right) (\phi_\ca L \circ h)
\end{equation*}
and the proof below will in fact show $\left( \DRop_{g,\emptyset}-\P_{g,\emptyset,0}^g \right) (\phi_\ca L \circ h)=0$.

By \eqref{eq:an_equality}, we need only show 
\begin{equation}\label{eq:pullback_Div}
\DRop_{g, \emptyset}(\phi_\ca L\circ h)  =  \DRop_{g,\bd 0}(\phi_{\ca L'})
:\Chow_*(B') \to \Chow_{*-g}(B'),
\end{equation}
\begin{equation}\label{eq:pullback_pixton}
\P_{g,\emptyset, 0}^g(\phi_\ca L \circ h) = \P_{g,\bd 0, 0}^g(\phi_{\ca L'})
:\Chow_*(B') \to \Chow_{*-g}(B')
\, .
\end{equation}
For the map $F : \Picabs_{g,n,0} \to \Picabs_{g,0,0}$ forgetting the markings, Invariance \invnrn{} from \cref{Sect:introinvariance} for the double ramification cycle and the Pixton formula shows that
we have
\begin{align*}
\DRop_{g,\bd 0}(\phi_{\ca L'})
&= \DRop_{g,\emptyset}(F \circ \phi_{\ca L'})  \\
\P_{g,\bd 0, 0}^g(\phi_{\ca L'}) 
&= \P_{g,\emptyset, 0}^g(F \circ \phi_{\ca L'})
\end{align*}
So we are reduced to showing
\begin{equation} \label{eqn:InvIVapplication}
    \DRop_{g, \emptyset}(\phi_\ca L\circ h) = \DRop_{g, \emptyset}(F \circ \phi_{\ca L'})\ \ \text{ and }\ \ \P_{g,\emptyset, 0}^g(\phi_\ca L\circ h) = \P_{g,\emptyset, 0}^g(F \circ \phi_{\ca L'})\ .
\end{equation}

The claims \eqref{eqn:InvIVapplication} follow from
 Invariance \invnrIV{} of \cref{Sect:introinvariance}. As before, let
 $C_{B'}$ be the pullback of $C$ under $h$,  and let
 $\ca L_{B'}$ be the pullback of $\ca L$ to $C_{B'}$. The map 
 $$\phi_\ca L\circ h  : B' \to \Picabs_{g,0,0}$$
 is induced by the data
\[C_{B'} \to B'\, ,\ \ \ \ca L_{B'} \to C_{B'}\, ,\]
whereas $F \circ \phi_{\ca L'} : B' \to \Picabs_{g,0,0}$ is induced by
\[C' \to B'\, ,\ \ \  \ca L'\to C'\,.\]
By construction, we have a partial destabilization $C' \to C_{B'}$ over $B'$, and the line bundle $\ca L'$ is the pullback of $\ca L_{B'}$ under this map. Hence the equalities \eqref{eqn:InvIVapplication} follow from Invariance \invnrIV{} of \cref{Sect:introinvariance}.
\end{proof}

\section{Comparing rubber and log spaces}\label{sec:comparing_stacks}

\subsection{Overview}
Our goal here is to compare the stack of stable rubber maps 
associated to a line bundle $\ca L$ on a target $X$   (introduced by Li \cite{Li2002A-degeneration-} and studied by Graber-Vakil \cite{Graber2005Relative-virtua}) to
the stack $\cat{Rub}_{g,A}$ of Marcus-Wise (see \cref{sec:log_def}) and our operational class $\DRop_{g,A}$. 
Rubber maps are reviewed in Section \ref{sec:prestable_rubber_maps} and connected to
the logarithmic space in Section \ref{clog}.
The relationship between  the construction
of Marcus-Wise and $\DRop_{g,A}$ is 
 \cref{lem:DRop_eq_MW} of Section \ref{kkkddd3}. The comparison to the class of Graber-Vakil is carried out in \cite{Marcus2017Logarithmic-com} in the case where the target $X$ is a point. We require the case where
$$X = \bb P^\kk  \ \ \ \text{and}\ \ \  \ca L = \ca O(1)\, ,$$
but only over the unobstructed locus 
$$\Mbar_{g,n}(\bb P^l, d)' \subset \Mbar_{g,n}(\bb P^l, d)\, ,$$
see \cref{sec:projective_target}. 
We will treat the case of a general nonsingular 
projective target $X$ since restricting to $\bb P^\kk$ provides no simplification (though the unobstructed locus may be rather small for general $X$). 
The final comparison result is Proposition \ref{prop:vfc_comparison_summary} in Section \ref{mwgv}.

\subsection{Refined definition of the logarithmic rubber space}

As described in Section \ref{sec:rub_log_def}, Marcus and Wise define $\cat{Rub}^{\mathsf{rel}}_{g}$ to be the moduli space of pairs $(C,P, \alpha)$ where $P$ is a tropical line on $S$ and $$\alpha\colon C \to P$$ is an $S$-morphism such that on each geometric fibre over $S$ the values taken by $\alpha$ on the irreducible components of $C$ are totally ordered in $(\bar{M}_S^{gp})_s$. However, with the above definition, certain key results of their paper (in particular concerning the comparison to spaces of rubber maps) are not correct as stated. 

To explain the problem, we restrict to the case where the base $C$ is a geometric log point. Subdividing $P$ at the images of the vertices of $C$ under the map $\alpha$ yields a \emph{divided tropical line} $Q$ (in the language of \cite{Marcus2017Logarithmic-com}). It is asserted in the discussion above \cite[\MWref{Proposition 5.5.2}\MWvone{Proposition 5.14}]{Marcus2017Logarithmic-com} that the fibre product $C \times_P Q$ is again a log curve over $S$, which, in general,
is not true. For example, take $\bar M_S$ to be the sub-monoid of $\mathbb Z^2$ generated by $(1,1)$, $(1,0)$, and $(1,-1)$, and $C$, $\alpha$ to be as illustrated in Figure \ref{fig:1}. In the fibre product,
the edge with length $(1,0)$ must be subdivided into two shorter edges, but $(1,0)$ 
is an irreducible element of $\bar M_S$. 
In fact, {\em failure of divisibility} is the only thing that can go wrong.

\begin{lemma}\label{lem:divisibility}
Let $(C, P, \alpha)$ be a point of $\cat{Rub}_g^{\mathsf{rel}}$ over a geometric log point $B$, and let $Q$ be obtained from $P$ by subdividing at the image of $\alpha$. Then the following are equivalent:
\begin{enumerate}
\item[(i)]
The fibre product $C \times_P Q$ is a log curve over $B$.
\item[(ii)] Let $e$ be an edge of $\Gamma_C$ between vertices $u$ and $v$ (satisfying $\alpha(v) \ge \alpha (u)$) with length $\ell_e\in \bar M_S$ and slope $\kappa_e = \frac{\alpha(v) - \alpha(u)}{\ell_e}$. Then, for every $y \in \on{Image}(\alpha)$ with $\alpha(u) < y < \alpha(v)$, the monoid $\bar M_{B}$ contains the element $\frac{y - \alpha(u)}{\kappa_e}$. 
\end{enumerate}
\end{lemma}
\begin{proof}
The characteristic monoid at a singular point with length $\ell_e$ is given by the monoid
\begin{equation}
\{(a,b) \in \bar M_b^2 : \ell_e \mid a-b\}\, . 
\end{equation}
Taking the fibre product over $P$ with $Q$ subdivides the characteristic monoid at the element $$\frac{y - \alpha(u)}{\kappa_e} \in \bar M_B^{\mathsf gp} \otimes_{\bb Z} \bb Q\, .$$
If $\frac{y - \alpha(u)}{\kappa_e}$ lies in $\bar M_B$, then the fiber product is easily seen to be a log curve. If not, then the subdivision is not even reduced. 
\end{proof}

\newcommand{\newRub}{\widetilde{\cat{Rub}}}
\newcommand{\newRubrel}{\newRub^{\mathsf{rel}}}
\begin{definition}
We define $\newRubrel$ to be the full subcategory of $\cat{Rub}^{\mathsf{rel}}$ consisting of objects $(C, P, \alpha)$ which, on each geometric fibre over $B$, satisfy the equivalent conditions of Lemma \ref{lem:divisibility}. We define $\newRub$ to be the fibre product of $\newRubrel$ over $\Picrel$ with $\Picabs$. 
\end{definition}

\begin{remark}
The double ramification cycle $\DRop$ can be defined as the operational class induced by the map $\cat{Rub} \to \Picabs$ following Definition \ref{defbivariantclass}. Applying the same definition to the composite map $\widetilde{\cat{Rub}} \to \Picabs$ yields the same operational class, by \Cref{prop:invarianceproperbirat}. 
\end{remark}



\begin{figure}
\begin{tikzpicture}
\begin{scope}[every node/.style={circle,thick,draw}]
    \node (A) at (0,0) {};
    \node (B) at (0,6) {};
    \node (C) at (3,3) {};
    \node (D) at (9,0) {};
    \node (E) at (9,3) {};
    \node (F) at (9,6) {} ;
\end{scope}
    \node (G) at (4,3) {} ;
        \node (H) at (7,3) {} ;
     \node at (10,0) {$(0,0)$} ;
\node at (10,3) {$(1,1)$} ;
\node at (10,6) {$(2,0)$} ;
\begin{scope}[>={Stealth[black]},
              every node/.style={fill=white,circle},
              every edge/.style={draw=black ,very thick}]
    \path [-] (A) edge node {$(1,0)$} (B);
    \path [-] (B) edge node {$(1,-1)$} (C);
    \path [-] (C) edge node {$(1,1)$} (A);
  \path [->] (G) edge node {$\alpha$} (H);
   \path [-] (D) edge  (E);
      \path [-] (F) edge  (E);
\end{scope}
\end{tikzpicture}
\caption{A point of $\cat{Rub}$}\label{fig:1}
\end{figure}

\subsection{The stack of prestable rubber maps}\label{sec:prestable_rubber_maps}
Let $\frak M(X)$ be the stack of maps from marked prestable curves to $X$. An $S$-point of $\frak M(X)$ is a pair 
$$(C/S,\, f\colon C \to X)  $$ 
where $C/S$ is prestable with markings. To simplify notation, we will often
suppress the markings.

The space of rubber maps associated to a line bundle $\ca L$ on $X$ is summarised in \cite{Janda2018Double-ramifica}:
\emph{a map to rubber with target $X$ is a map to a rubber chain of $\bb C\bb P^1$-bundles $\bb P(\ca O_X \oplus \ca L)$ over $X$ attached along their 0  and $\infty$ divisors.}

\newcommand{\R}{R}
To facilitate our comparison, we begin by writing the definition explicitly. 
Let $\mathbf{P}$ denote the projective bundle $\bb P(\ca O_X \oplus \ca L)$.
The map collapsing the fibers,
\begin{equation}\label{55772}
\rho: \mathbf{P} \to X\ 
\end{equation}
admits two sections $r_0,r_\infty:X\to \mathbf{P}$
corresponding to $\ca O_X$ and $\ca L$ respectively.

\begin{definition}\label{def:rubber_target} An $(X,\ca L)$-rubber target $(R/S,\rho, r_0,r_\infty)$ is flat,
proper, and finitely presented
$$ R \to S$$
and a collapsing map $\rho: R \to X_S$
with two sections 
$$r_0,r_\infty: X_S \to R\, $$
satisfying the following properties:
\begin{enumerate}
    \item [(i)]
    Every geometric fiber  $R_s$ is isomorphic over $X_s$ to
    a finite chain  
    \begin{equation}\label{v477}
    \mathbf{P}\cup \mathbf{P} \cup \ldots \cup \mathbf{P}
    \end{equation}
    with the components attached successively along the respective
    $0$ and $\infty$ divisors. The collapsing maps \eqref{55772} on the
    components together define $$\rho_s: R_s \to X_s\,.$$
    The $0$ and $\infty$ sections of $\rho_s$ are determined
    by the $0$ section of first component  and the $\infty$ section of
    last components of the chain \eqref{v477}.
    \item[(ii)] \'Etale locally near every point $s\in S$, the data 
    of $(R/S,\rho, r_0,r_\infty)$ is pulled back from
    a versal deformation space described by Li \cite{Li2001-Stable}
    with one dimension for every component of the singular locus of \eqref{v477}. 
\end{enumerate}
\end{definition}
\begin{definition}\label{def:prestable_rubber_maps}
The stack $\cat{Rub}^{\pre}(X, \ca L)$ of prestable rubber maps to $\ca L$ is a fibred category over $\frak M(X)$ whose fibre over a map $S \to \frak M(X)$ consists of 
three pieces of data: 
\begin{enumerate}
\item[(i)]
a prestable curve $\tilde C /S$ and a partial
stabilisation{\footnote{$C_S$ is not
necessarily a stable curve.}}
map $\tau\colon \tilde C \to C_S$ which is allowed to contract
genus 0 components with 2 special points,
\item[(ii)] an $(X,\ca L)$-rubber target $(R/S,\rho,r_0,r_\infty)$,

 

\item[(iii)] a map $\tilde{f}: \tilde C \to \R$ for which 
the following diagram commutes: 
\begin{equation}
\begin{tikzcd}
\tilde C \arrow[r,"\tilde{f}"] \arrow[d,"\tau"] & R \arrow[d, "\rho"]\\
C_S \arrow[r, "f"] & \, X_S\, .
\end{tikzcd}
\end{equation}
\end{enumerate}
The map $\tilde{f}$ in (iii) is finite over the singularities of $R/X_S$ and
predeformable{\footnote{See \cite{Li2001-Stable}.}}.
Moreover, over each geometric point $s\in S$, the image $\tilde{f}(\tilde{C}_s)$
meets
every component of $R_s$.

\end{definition}

An isomorphism between two objects
\begin{equation*}
(\tilde C \to C_S, R, r_0, r_\infty, \tilde C \to R) \,\,\, \text{ and } \,\,\, (\tilde C' \to C_S, R', r'_0, r'_\infty, \tilde C' \to R')
\end{equation*}
over $S \to \frak M(X)$ is given by the data of isomorphisms
\begin{equation*}
\tilde C' \isom \tilde C
\end{equation*}
over $C_S$ and
\begin{equation*}
R' \isom R
\end{equation*}
over $X_S$, compatible with the markings and such that the diagram
\begin{equation*}
\begin{tikzcd}
\tilde C' \arrow[r, "\sim"] \arrow[d] & \tilde C \arrow[d]\\
R' \arrow[r,"\sim"] & R
\end{tikzcd}
\end{equation*}
commutes. We leave the definition of the cartesian morphisms to the careful reader. 

Suppose now that we fix a genus $g$ and a vector of integers $A$ of length $n$. We define the stack $\cat{Rub}^{\pre}_{g,A}(X, \ca L)$ with objects being tuples 
\begin{equation}\label{vtt5}
(\tau\colon (\tilde C, p_1,\ldots,p_n) \to C_S\, , \R/X_S\, , \tilde{f}\colon \tilde C \to \R)
\end{equation}
where $(\tilde{C},p_1,\ldots,p_n)$ is a prestable curve of genus $g$ with $n$
markings.
The data \eqref{vtt5} are as for $\cat{Rub}^{\pre}(X, \ca L)$. Moreover, 
\begin{itemize}

\item[$\bullet$]
if $a_i >0$,  $p_i\in \tilde{C}$ is mapped to the $0$-divisor with
 ramification degree $a_i$,
 \item[$\bullet$]
if $a_i <0$,  $p_i\in \tilde{C}$ is mapped to the $\infty$-divisor with
 ramification degree $-a_i$,
\item[$\bullet$] if $a_i=0$, $p_i\in \tilde{C}$ is mapped to the smooth locus of $R$
away from the $0$ and $\infty$-divisors.
\end{itemize}


\subsection{Comparison to the logarithmic space}\label{clog}


The pullback of $\ca L$ from $X$ to the universal curve over $\frak M(X)$ induces a map $\frak M(X) \to \frak{Pic}$. The key comparison result is the following. 


\begin{proposition}\label{lem:prestable_stack_comparison}
The stack $\cat{Rub}_{g,A}^{\pre}(X, \ca L)$ is naturally isomorphic to the fibre product of $\newRub_{g,A}$ over $\Picabs$ with $\frak M(X)$ along the map induced by $\ca L$,
$$\cat{Rub}_{g,A}^{\pre}(X, \ca L)\,  \stackrel{\sim}{=}\,  \newRub_{g,A} \times_{\Picabs} {\frak{M}(X)} \, .$$
\end{proposition}
\begin{proof}
The right hand side comes with a built-in log structure, but the left side
does not. Our isomorphism will be between the underlying stacks. Our proof is based on the discussion above \cite[\MWref{Proposition 5.5.2}\MWvone{Proposition 5.14}]{Marcus2017Logarithmic-com}, and we will use the language of \emph{divided tropical lines} of \cite{Marcus2017Logarithmic-com}.

We begin by building a map from the right to the left. We are given a log curve $C/S$, a tropical line $\ca P$ on $S$, a map $\alpha\colon C\to \ca P$ whose image is totally ordered, and a map $f\colon C \to X$, such that $f^*\ca L$ lies in the isomorphism class $\ca O_C(\alpha)$. 

The $\bb G_m^{\mathsf{trop}}$-torsor $\ca P$ is rigidified by the least element among the images of the irreducible components of $C$ (here we use the total ordering condition), and hence comes with a canonical $\bb G_m$-torsor $P \to \ca P$ (if we use the rigidification to identify $\ca P = \bb G_m^{\mathsf{trop}}$ then $P = \bb G_m^{\mathsf{log}}$). The pullback $\alpha^*P$ gives a \emph{canonical} $\bb G_m$-torsor on $C$, which is isomorphic to $f^*\ca L^*$ up to pullback from $S$. In other words, the bundle $\alpha^*P \otimes f^*\ca L^\vee$ descends to a line bundle on $S$ which we denote $\ca M$. 

The images of the irreducible components of $C$ yield a subdivision $\ca Q$ of $\ca P$, and we define a destabilisation $\tilde C = C \times_{\ca P}\ca Q$ of $C$, which is a log curve over $S$ by Lemma \ref{lem:divisibility}. This $\ca Q$ comes with a canonical $\bb G_m$-torsor $Q$ by pulling back $P$ from $\ca P$; this $Q$ is then a 2-marked semistable genus 0 curve by \cite[\MWref{Proposition 5.2.4}\MWvone{Proposition 5.7}]{Marcus2017Logarithmic-com}. We define an $(X, \ca L)$-rubber target $R$ over $S$ by the formula
$$ R = \on{Hom}((\ca L \otimes \ca M)^*, Q)\,. $$
Here, we pull back 
and take $\bb G_m$-equivariant homomorphisms over $X_S$. 


Write $\tilde f\colon \tilde C \to X$. 
We need a predeformable map $\tilde C \to R$, equivalently an equivariant logarithmic map $\tilde f^* {\ca L}^* \to Q$ over $X_S$. It is enough to give a map $f^*\ca L \to P$ (since then we can tensor over $\ca P$ with $\ca Q$), which reduces to writing down an element of 
\begin{equation}
\begin{split}
\on{Hom}_C(f^*(\ca L \otimes \ca M), \alpha^*P) &= 
\on{Hom}_C(f^*\ca L \otimes f^*\ca L^\vee \otimes \alpha^*P, \alpha^*P)\\
& = \on{Hom}_C(\alpha^*P, \alpha^*P), 
\end{split}
\end{equation}
which contains the identity. The scheme-theoretic map is predeformable as it comes from a logarithmic map, see \cite{BumsigKim2008LOGARITHMICSTABLEMAPS}. 


Finally we check that no component of $\tilde C$ is mapped to a non-smooth point of $R$ and that every component is hit. The target $R$ is  constructed by subdividing $f^*\ca L$ at images of components of $C$, and then $\tilde C$ is constructed by subdividing $C$ at points lying over these divisions, so both assertions are clear. 

Now we construct a map from left to right. Given a prestable rubber map to $\ca L$ over a base $S$, we first need to equip $S$ with a suitable log structure. 

The curve $R/X_S$ is a map $X_S \to \frak M_{0,2}^{ss}$, giving a (minimal) log structure on $X_S$ by pullback.
Lemma \ref{bbll} below shows that this log structure descends to $S$. The curve $R/X_S$ now carries the structure of a log curve, and similarly the quotient $[R/\bb G_m]$ descends to $S$ (again by the Lemma \ref{bbll}), determining our tropical line $\ca P$ --- which evidently satisfies the divisibility condition in Lemma \ref{lem:divisibility}.  

It remains to verify that the map $\tilde C \to R$ descends to a map $C \to \ca P$ 
and that the total ordering condition is satisfied. Write
$$\tau\colon \tilde C \to C\, .$$ 
By the proof of \cite[\MWref{Proposition 5.5.2}\MWvone{Proposition 5.14}]{Marcus2017Logarithmic-com},
we see $R \to \tau_*\tau^*R$ is an isomorphism, hence the map descends as required. The condition that no components are mapped to the nodes implies that the values of $\alpha$ on the irreducible of $C$ are a subset of the irreducible components of $R$, in particular are totally ordered. 
\end{proof}

\begin{lemma}\label{bbll}
Let $(R/S, \rho, r_0, r_\infty)$ be an $(X, \ca L)$-rubber target. Then there exists a (minimal) log structure on $S$ such that $R/X_S$ can be equipped with the structure of  a log curve making $X_S$ strict over $S$. The quotient log stack $[R/\bb G_m]$ descends to a divided tropical line on $S$.  
\end{lemma}
\begin{proof}
The curve $R/X_S$ with markings $r_i$ is prestable and hence admits a minimal log structure. We must verify that the resulting log structure on $X_S$ descends to $S$. After a finite extension of $\field$ we may assume $X$ has a $\field$ point, so that $$\pi\colon X_S \to S$$
admits a section $x\colon S \to X_S$, and we can equip $S$ with the pullback log structure. It remains to construct an isomorphism $\pi^*x^* M_{X_S} \to M_{X_S}$. We start by building a map from left to right. 

We first build a map on the level of characteristic monoids. The characteristic monoid at a geometric point $t \in X_S$ is given by $\bb N^\ell$, where $\ell$ is the length of the chain of projective lines of $R$ over $t$. Crucially, the irreducible elements of $\bb N^\ell$ come with a total order, given by proximity of the corresponding singularity to the $r_0$ marking. This rigidifies the characteristic monoid, so as we move along the fibre over $\pi(t)$ the characteristic monoids are \emph{canonically} identified. We
obtain canonical identifications
$$(\overline{x^*M_{X_S}})_{\pi(t)} \isom (\bar{M}_{X_S})_t\, ,$$
which give an isomorphism
$$\pi^*x^* \bar M_{X_S} \isom \bar{M}_{X_S}\, .$$

 To construct an isomorphism of log structures, we will use the perspective of \cite[Section 3.1]{borne-vistoli} that a log structure is a monoidal functor from the groupified characteristic monoid to the stack of line bundles. The rubber target is by definition pulled back from  Li's versal deformation spaces, so it suffices to construct our map in that setting. We can
 therefore assume that $S$ is regular and the locus of non-smooth curves is a reduced divisor in $X_S$.
 Since our map will be canonical, we may further shrink $S$ to be atomic\footnote{\cite[Definition 2.2.4]{Abram_wise_birational}.}. Then $\bar M_{X/S}$ is generated by its global sections, and there is a natural isomorphism of sheaves on $X_S$
 $$\phi\colon \bb N^\ell \isom \bar M_{X/S}$$
 where $\ell$ is the number of singular points in the fibre of $C$ over any point of $X_S$ lying over the closed stratum of $S$. Given $1 \le i \le \ell$, write $D_i$ for the Cartier divisor in $X_S$ where the singularity at distance $i$ from the first marking persists. Then $\phi$ sends the $i$th generator  of $\bb N^\ell$ to the section corresponding to the line bundle $\ca O_{X_S}(D_i)$. To build the required map of log structures $$\pi^*x^*M_{X_S} \isom M_{X_S}\, ,$$ we must construct an isomorphism
$$\pi^*x^*\ca O_{X_S}(D_i) \isom \ca O_{X_S}(D_i)\, .$$
Condition (i) of the \cref{def:rubber_target} implies that the underlying point set of $D_i$ is a union of fibres of $X_S/S$. Since $S$ is regular and $X_S$ is smooth over $S$, it follows that
$$D_i = \pi^*x^*D_i$$
giving the required isomorphism. 

The quotient log stack $[R/\bb G_m]$ is a divided tropical line on $X_S$ with divisions coming from the divisors $D_i$. We can identify the underlying tropical line with $\bb G_m^{trop}$ by specifying that the smallest element in the sequence of divisions is mapped to $0$. We have already established that these divisions $D_i$ descend to $S$, hence so does the divided tropical line. 
\end{proof}



After restriction to the locus 
where the infinitesimal automorphisms are trivial, we obtain
 a stable version of \cref{lem:prestable_stack_comparison}. 
Let $\MbarX$ denote the stack of stable maps from $n$-pointed curves to $X$ 
representing  the class $\beta$.
The line bundle $\ca L$ determines 
a map $$\MbarX \to \Picabs\, $$  
and we can pullback $\newRub_{g,A}$ as before.
Let $$\cat{Rub}_{g,A}(X, \ca L) \subset \cat{Rub}_{g,A}^{\pre}(X, \ca L)$$ 
be the locus where the infinitesimal automorphisms are trivial.

\begin{lemma}
The stack $\cat{Rub}_{g,A}(X, \ca L)$ is the fibre product of $\newRub_{g,A}$ over $\Picabs$ with $\MbarX$ along the map given by $\ca L$. 
$$\cat{Rub}_{g,A}(X, \ca L)\,  \stackrel{\sim}{=}\,  \newRub_{g,A} \times _\Picabs \MbarX\, .$$
\end{lemma}

Next, we will compare the virtual fundamental classes on these
spaces. We will carry out the comparison on a smaller open locus.  
We define 
\begin{enumerate}
\item[(i)]
$\MbarX'$ is be the open locus of maps
$(C, f\colon C \to X)$ in $\MbarX$ where $H^1(C, f^*\ca L) = 0$.
\item[(ii)]
$\cat{Rub}_{g,A}(X, \ca L)' = \cat{Rub}_{g,A}(X, \ca L) \times_{\MbarX} \MbarX'$. 
\end{enumerate}
In  \cref{sec:projective_target}, we considered the case $X = \bb P^\kk$
and showed that this unobstructed locus is large enough
to control the cycles relevant to Theorem \ref{Thm:main}. For general $X$,
the unobstructed  locus might be very small (and possibly empty).

\subsection{Comparing the virtual classes}\label{sec:comparing_virtual_classes}
\subsubsection{Overview}
We begin by briefly discussing of several spaces which will be relevant in setting up the obstruction theories. Let $$\frak M_{g, n}^{ss}\subset \frak{M}_{g,n}$$
be the semistable locus (where every rational curve has at least two distinguished points).
We write $\cl T$ for the algebraic stack with log structure which parametrises tropical lines with at least one division. 
There are natural maps $$\frak M_{0,2}^{ss} \to \cl T\ \ \ \text{and}\ \ \  \frak M_{0,2}^{ss} \to B\bb G_m\, ,$$
the former defined by dividing $\bb G_m^{\mathsf{trop}}$ at $1$ and at the smoothing parameters of the nodes, and the latter defined by the normal bundle at the first marking. The induced map 
\begin{equation}
\frak M_{0,2}^{ss} \to \cl T\times B\bb G_m
\end{equation}
is an isomorphism by \cite[Proposition 3.3.3]{Abramovich2013Expanded-degene}. 

As $\cat{Rub}^{\rel}$ is the moduli stack of tuples $(C, \alpha\colon C \to \ca P)$ where $\ca P$ is a tropical line and the images of the irreducible components of $C$ are totally ordered, there is a natural map 
\begin{equation}\label{ff923}
\cat{Rub}^{\rel} \to \cl T
\end{equation}
sending $(C, \alpha\colon C \to \ca P)$ to the tropical line $\ca P$ with the division given by the images of the irreducible components of $C$. 

We will construct a map 
 \begin{equation}\label{vrt5}
 \cat{Rub}_{g,A}(X, \ca L)'\to \frak M_{0,2}^{ss}
 \end{equation}
 lifting the  morphism \eqref{ff923} by the following argument. 
A point of $\cat{Rub}_{g,A}(X, \ca L)'$ is a tuple $(C, \alpha\colon C \to \ca P, f\colon C \to X)$ where $f^*\ca L$ lies in the class{\footnote{Here,
$[\ca O_C(\alpha)]$
 is an equivalence class under isomorphisms and tensoring with pullbacks from $S$.
}}
$[\ca O_C(\alpha)]$.  However, as $\ca P$ is divided, there is a unique isomorphism $\ca P \isom \bb G_m^{trop}$ where the smallest division maps to $0$. The universal $\bb G_m$ torsor $\bb G_m^{log} \to \bb G_m^{trop}$ pulls back to a well-defined $\bb G_m$-torsor $\ca O_C^*(\alpha)$ on $C$, and the difference $f^*\ca L^* \otimes_{\ca O^*_C} \ca O^*_C(-\alpha)$ descends to a $\bb G_m$-torsor on $S$ by the construction of $\cat{Rub}_{g,A}(X, \ca L)'$ as a fibre product. The $\bb G_m$-torsor on $S$ induces a map $\cat{Rub}_{g,A}(X, \ca L)' \to B \bb G_m$. Combined with the map $\cat{Rub}_{g,A}(X, \ca L)' \to \cl T$ via \eqref{ff923},
we obtain the
map \eqref{vrt5}.



The space $\cat{Rub}_{g,A}(X, \ca L)'$ carries three virtual fundamental classes by
the following three constructions:
\begin{enumerate}
\item[(i)] The class $\DRop_{g,A}(\phi_\ca L)([\MbarX'])$ obtained by applying 
$\DRop_{g,A}$
to the (virtual) fundamental class of $\MbarX'$ via the map $\phi_\ca L$. 
\item[(ii)] The class obtained from a two-step obstruction theory described by Marcus and Wise \cite{Marcus2017Logarithmic-com} for the map $\cat{Rub}_{g,A}(X, \ca L)' \to \frak M_{g,n} \times \cl T$. 
\item[(iii)]  A class coming from a two-step obstruction theory studied by
Graber and Vakil \cite{Graber2005Relative-virtua} for the map 
$$\cat{Rub}_{g,A}(X, \ca L)' \to \MbarX' \times \frak M_{0,2}^{ss}\, .$$ 
\end{enumerate}
We will prove (i)-(iii) are all equal.





\subsubsection{$\DRop$ and the obstruction theory of Marcus-Wise}
\label{kkkddd3}

The obstruction theory of Marcus-Wise is a \emph{two-step} obstruction theory, a notion which we now recall. Unless otherwise stated, by \emph{perfect obstruction theory} we mean an obstruction theory which is perfect in amplitude $[-1,0]$. 

\begin{definition}A \emph{two-step} obstruction theory for a map $f\colon X \to S$ consists of a factorisation $$X \to Y \to S$$ together with 
prefect relative obstruction theories for $X/Y$ and for $Y/S$. 
\end{definition}
A two-step obstruction theory induces a virtual pullback by composition.{\footnote{A
two-step obstruction theory
also induces a perfect obstruction theory for $X/S$ in amplitude $[-2,0]$, but we
will not use the latter construction.}}
If $S$ has a fundamental class $[S]$, the virtual pullback of $[S]$
is the {\em virtual fundamental class} of $X$ associated to the
two-step obstruction theory.

We first recall the two-step obstruction theory of \cite{Marcus2017Logarithmic-com} in the case when $X$ is a point. We have a diagram
\begin{equation}
 \begin{tikzcd}
  \cat{Rub}_{g,A}(\mathsf{pt}, \ca O)' \arrow[dr] \arrow[r] & \widetilde{\cat{Rub}}_{g,A} \arrow[d]\\
& \frak M_{g,n} \times \cl T. 
\end{tikzcd}
\end{equation}
A perfect relative obstruction theory for the horizontal map is given in \cite[\MWref{Section 5.6.3}\MWvone{Section 5.6.2}]{Marcus2017Logarithmic-com}, and for the vertical map in \cite[Proposition \MWref{5.6.5.3}\MWvone{5.20}]{Marcus2017Logarithmic-com}; while the reader might expect that these arguments apply to $\cat{Rub}_{g,A}$ rather than the root stack $\widetilde{\cat{Rub}}_{g, A}$, Marcus and Wise in fact assume in both constructions the divisibility conditions of Lemma \ref{lem:divisibility}, hence their constructions in fact apply to $\widetilde{\cat{Rub}}_{g, A}$ (and not to $\cat{Rub}_{g, A}$). This two-step obstruction theory
coincides with the rubber theory of Graber-Vakil, as shown in \cite[Section \MWref{5.6.6}\MWvone{5.6.5}]{Marcus2017Logarithmic-com}.
Moreover, the virtual fundamental class obtained
equals the operational class of $\cat{Div}_{g,A}$, see \cite[\MWref{Theorem 5.6.1}\MWvone{Theorem 5.15}]{Marcus2017Logarithmic-com}; again, these results all assume the divisibility condition of Lemma \ref{lem:divisibility}, and hence apply to $\widetilde{\cat{Rub}}_{g,A}$ in place of $\cat{Rub}_{g, A}$. 


Returning to the case of arbitrary $(X,\ca L)$,
we can construct a similar commutative diagram 
\begin{equation}\label{eq:obs_factor_MW}
 \begin{tikzcd}
  \cat{Rub}_{g,A}(X, \ca L)' \arrow[dr] \arrow[r] & \widetilde{\cat{Rub}}_{g,A} \arrow[d]\\
& \frak M_{g,n} \times \cl T\, .
\end{tikzcd}
\end{equation}
The vertical map is unchanged and so again has a perfect relative obstruction theory by \cite[Proposition \MWref{5.6.5.3}\MWvone{5.20}]{Marcus2017Logarithmic-com}; in fact the morphism is a local complete intersection, and the obstruction theory of \cite[Proposition \MWref{5.6.5.3}\MWvone{5.20}]{Marcus2017Logarithmic-com} is just the relative tangent complex. 

We need to supply a perfect obstruction theory for the horizontal map $$\widetilde{\cat{Rub}}_{g,A}(X, \ca L)'  \to  \cat{Rub}_{g,A}\, ,$$
which we can factor as
\begin{equation}
\cat{Rub}_{g,A}(X, \ca L)'  \to \MbarX' \times_{\frak M_{g,n}}\widetilde{\cat{Rub}}_{g,A}  \to \widetilde{\cat{Rub}}_{g,A}\, . 
\end{equation}
The second map is a base change of the unobstructed map $\MbarX' \to \frak M_{g,n}$, hence is unobstructed. For the first map, consider the pullback square
\begin{equation}\label{eq:long_vertical_pullback}
 \begin{tikzcd}
\cat{Rub}_{g,A}(X, \ca L)'  \arrow[d] \arrow[r]   & \MbarX' \arrow[d]\\
\MbarX' \times_{\frak M_{g,n}} \widetilde{\cat{Rub}}_{g,A} \arrow[r]  & \MbarX' \times_{\frak M_{g,n}}\Picabs_{g,n}\, .\\
\end{tikzcd}
\end{equation}
The right vertical arrow is a section of a base change of the smooth morphism $\Picabs_{g,n} \to \frak M_{g,n}$, and as such is lci and has a perfect relative obstruction theory given by the relative tangent complex $R^1\pi_*\ca O_C$. Pullback yields a corresponding perfect obstruction theory for the left vertical arrow. This gives a two-step obstruction theory for the composite map $\cat{Rub}_{g,A}(X, \ca L)'  \to  \widetilde{\cat{Rub}}_{g,A}$, from which we obtain a virtual fundamental class following \cite{Manolache2012Pullbacks}. 

The discussion here
is a very slight generalisation of the obstruction theory constructed in \cite[\MWref{Proposition 5.6.3.1}\MWvone{Proposition 5.16}]{Marcus2017Logarithmic-com}.

\begin{definition}
The two-step obstruction theory for the diagonal map of \eqref{eq:obs_factor_MW},
$$
\cat{Rub}_{g,A}(X, \ca L)' \to \frak M_{g,n} \times \cl T\, $$ is the \emph{Marcus-Wise} obstruction theory.
\end{definition}

\begin{lemma}\label{lem:DRop_eq_MW}
The push forward along
$$\psi\colon \cat{Rub}_{g,A}(X, \ca L)' \to \MbarX'$$
of the virtual fundamental class of the Marcus-Wise theory
on $\cat{Rub}_{g,A}(X, \ca L)'$ 
equals the class $\DRop_{g,A}(\phi_\ca L)([\MbarX'])$
obtained via the map
$$\phi_\ca L\colon \MbarX' \to \Picabs_{g,n}\, .$$
\end{lemma}
\begin{proof}
From $\phi_{\ca L}$, we obtain maps
$$\phi'_\ca L\colon  \MbarX' \to \MbarX'\times_{\frak M_{g,n}}\Picabs_{g,n}\, ,$$ $$\phi''_\ca L\colon  \MbarX' \to \MbarX'\times \Picabs_{g,n}\, . $$
Both are lci morphisms because $\Picabs_{g,n}/\frak M_{g,n}$ is smooth. By 
Definition \ref{defbivariantclass} and Section \ref{sec:proof_of_equiv_defs} we have
\begin{align*}
    \DRop_{g,A}(\phi_\ca L)([\MbarX']) &=  \psi_*(\phi''_\ca L)^![\MbarX' \times \widetilde{\cat{Rub}}_{g,A}]\\
    &= \psi_*(\phi'_\ca L)^![\MbarX' \times_{\frak M_{g,n}} \widetilde{\cat{Rub}}_{g,A}]\, .
\end{align*}
The virtual fundamental class of the Marcus-Wise theory is the virtual pullback of the fundamental class of $\frak M_{g,n} \times \cl T$ along the composition
\begin{equation}
\cat{Rub}_{g,A}(X, \ca L)'  \stackrel{1}{\to} \MbarX' \times_{\frak M_{g,n}} \widetilde{\cat{Rub}}_{g,A}   \stackrel{2}{\to} \widetilde{\cat{Rub}}_{g,A}   \stackrel{3}{\to} \frak M_{g,n} \times \cl T\, . 
\end{equation}
The map (3) is lci and the obstruction theory is the relative tangent complex, so the pullback of the fundamental class is the fundamental class of $\widetilde{\cat{Rub}}_{g,A}$. The map (2) is unobstructed, so the (virtual) pullback of the fundamental class is again the fundamental class. The obstruction theory of the map (1) is defined by pulling back the relative tangent complex of the lci morphism $$\MbarX' \to \MbarX' \times_{\frak M_{g,n}}\Picabs_{g, n}$$ via the pullback square
\begin{equation}
\begin{tikzcd}
\cat{Rub}_{g,A}(X, \ca L)' \arrow[r] \arrow[d] & \MbarX' \times_{\frak M_{g,n}} \widetilde{\cat{Rub}}_{g,A}  \arrow[d] \\
\MbarX' \arrow[r, "\phi'_\ca L"] & \MbarX' \times_{\frak M_{g,n}} \Picabs_{g,n}\, , 
\end{tikzcd}
\end{equation}
so the virtual pullback of the fundamental class of $\MbarX' \times_{\frak M_{g,n}} \widetilde{\cat{Rub}}_{g,A}$ is  equal to the Gysin pullback $(\phi'_\ca L)^![\MbarX' \times_{\frak M_{g,n}} \widetilde{\cat{Rub}}_{g,A}]$. 
\end{proof}

\subsubsection{Marcus-Wise and Graber-Vakil}\label{mwgv}

As recalled above, Marcus-Wise define a two-step obstruction theory for the map 
\begin{equation*}
\cat{Rub}_{g,A}(X, \ca L)' \to \frak M_{g,n} \times \cl T\,. 
\end{equation*}
Graber and Vakil consider an obstruction theory for the map 
\begin{equation} \label{vrrt}
\cat{Rub}_{g,A}(X, \ca L)' \to \MbarX' \times \frak M_{0,2}^{ss}\, .
\end{equation}

We wish to show an equality of the corresponding virtual fundamental classes on $\cat{Rub}_{g,A}(X, \ca L)'$. Since $$\frak M_{0,2}^{ss} = \cl T \times B\bb G_m\, ,$$ and that the maps $\MbarX' \to \frak M_{g,n}$ and $B \bb G_m \to \on{Spec} \field$ are unobstructed, we have an unobstructed map
$$ \MbarX' \times \frak M_{0,2}^{ss} \to \frak M_{g,n} \times \frak M_{0,2}^{ss} \to \frak M_{g,n} \times \cl T\, .$$
Our final step is therefore to compare the obstruction theories (and thereby the corresponding virtual pullbacks) between Marcus-Wise and Graber-Vakil \cite{Graber2005Relative-virtua,Li2002A-degeneration-}.
We will match the obstruction spaces when the base $S$ is a point.
The full matching of deformation theories is similar and will be treated in \cite{PW}.
The claims are also required for \cite{Marcus2017Logarithmic-com}.

Suppose we are given the data of a point in $\cat{Rub}_{g,A}(X, \ca L)'(S)$,
\begin{equation}
(\tau\colon \tilde C \to C_S, \R/X_S, \phi\colon \tilde C \to \R, f\colon C_S \to X_S, p_1, \dots, p_n, f^*\ca L \isom \ca O_C(\alpha))\, .
\end{equation}

\begin{proposition}\label{prop:vfc_comparison_summary}
The restriction to $\MbarX'$ of the class $\DR_{g,A}(X, \ca L)$ of \cite{Janda2018Double-ramifica} is equal to the class obtained by letting $\DRop_{g,A}$ act on the fundamental class of $\MbarX'$ via the map induced by $\ca L$. 
\end{proposition}



\begin{proof}
The primary obstruction of Graber-Vakil lies in 
\begin{equation}\label{ggtt2}
H^0(\tilde C, \phi^{-1}\ca{E}xt^1(\Omega_{\R/X_S}(\log D), \ca O_{\R}))\, .
\end{equation}
Here, $D$ is the divisor on $\R$ given by the sum of the two markings $r_0$ and $r_\infty$, and $\Omega_{\R/X_S}(\log D)$ is the sheaf of relative 1-forms on $\R/X_S$ allowed logarithmic poles along $D$ (a coherent sheaf on $\R$). The obstruction space 
\eqref{ggtt2}
is isomorphic to the product of the deformation spaces of the nodes of $\R$
and coincides with the obstruction space
for the map $$\widetilde{\cat{Rub}}_{g,A} \to \frak M_{g,n} \times \cl T$$ the vertical arrow in \eqref{eq:obs_factor_MW}), coming from \cite[Proposition \MWref{5.6.5.3}\MWvone{5.20}]{Marcus2017Logarithmic-com} (where they assume the divisibility conditions of Lemma \ref{lem:divisibility}, hence the results apply to $\widetilde{\cat{Rub}}_{g, A}$ and not to $\cat{Rub}_{g, A}$). 

Suppose that the primary obstruction vanishes. Denote by $$T_{\R/X_S} = \sheafhom_{\ca O_{\R}}(\Omega_{\R/X_S} , \ca O_{\R})$$ the relative tangent sheaf. There is a secondary obstruction in 
\begin{equation}\label{pp39}
H^1(\tilde{C}, \phi^{\dagger}(T_{\R/X_S}))\,, 
\end{equation}
where the $\phi^{\dagger}$ is the torsion-free part of $\phi^*$, see \cite{Graber2005Relative-virtua} and \cite{Marcus2017Logarithmic-com}. The
obstruction space \eqref{pp39}
is the image of the obstruction space $H^1 (C_S, f^* T_\R ) = H^1(C_S, \ca O_{C_S})$ for the map $$\cat{Rub}_{g,A}(X, \ca L)' \to \widetilde{\cat{Rub}}_{g,A}\, $$
the horizontal arrow in \eqref{eq:obs_factor_MW}, coming from \cite[Proposition \MWref{5.6.3.1}\MWvone{5.17}]{Marcus2017Logarithmic-com}.

When both of these obstructions vanish, the deformations are a torsor under $H^0(\tilde{C}, \phi^\dagger T_{\R/X_S} )$, an extension of the first term of the obstruction complex in \cite[Proposition \MWref{5.6.5.3}\MWvone{5.20}]{Marcus2017Logarithmic-com} by the first term of the obstruction complex of \cite[Proposition \MWref{5.6.4.1}\MWvone{5.17}]{Marcus2017Logarithmic-com}. The comparison of the obstruction theories is complete. 
\end{proof}

\section{Invariance properties} \label{Sect:proofInvariance}
\subsection{Overview}
We prove here the invariance properties of the universal twisted double ramification cycle as presented in Section \ref{Sect:introinvariance}. 

We start with an object of $\phi_\ca L\colon \mathcal S \to \Picabs_{g,n,d}$ 
given by
a flat family of prestable $n$-pointed genus $g$ curves together with a line bundle of relative degree $d$,
    \begin{equation}
    \label{eqn:invSobj}    
    \pi:\mathcal{C} \rightarrow \mathcal{S}\, , \ \ \ \ p_1,\ldots,p_n: \mathcal{S} \rightarrow \mathcal{C}\, , \ \ \ \
        \mathcal{L} \rightarrow \mathcal{C}\, .
        \end{equation}
    Let $\mathsf{DR}^{\mathsf{op}}_{g,A,\mathcal{L}}= \phi_\ca L^*\DRop_{g,A}\in \mathsf{CH}^g_{\mathsf{op}}(\mathcal{S})$
    be the twisted double ramification cycle associated to the above
    family \eqref{eqn:invSobj} and
     the vector
    $$A=(a_1,\ldots, a_n)\, , \ \ \ \ d =\sum_{i=1}^n a_i\, .$$


Theorem \ref{Thm:main} asserts that  $\DRop_{g,A}$ is equal to
the tautological class
$$\P^g_{g,A,d} \in \CHop^g(\Picabs_{g,n,d})\, .$$ If we
assume Theorem \ref{Thm:main}, we have a choice of proving the invariance properties either for $\DRop_{g,A}$ or for the formula in tautological classes. In fact, since both
sides of Invariances \invnrn, \invnrA{} and
\invnrIV{} are used {\em in the proof} of Theorem \ref{Thm:main},
we will have to prove these two sides separately in each of these cases. 
In fact, we will do this for all the invariances, as each side yields interesting perspectives. Also, we will show the invariances of Pixton's formula hold
not just for 
the codimension $g$ part $\P^g_{g,A,d}$, but for the full mixed degree class
\[\P^\bullet_{g,A,d} \in \prod_{c=0}^\infty \CHop^c(\Picabs_{g,n,d})\, .\]


\subsection{Proof of Invariance \invnrI{}: {\rm Dualizing}}
We want to show the invariance
        $$\mathsf{DR}^{\mathsf{op}}_{g,-A, \mathcal{L}^*}
        =\epsilon^* \mathsf{DR}^{\mathsf{op}}_{g,A, \mathcal{L}}
        \, , $$
where $\epsilon: \Picabs_{g,n,-d} \rightarrow \Picabs_{g,n,d}$
is the natural map obtained via dualizing the line bundle.  It is enough to show the invariance
\begin{equation}\label{vv99q}
\mathsf{DR}^{\mathsf{op}}_{g,-A} = \epsilon^* \mathsf{DR}^{\mathsf{op}}_{g,A}
\end{equation}
of the universal twisted double ramification cycles.
The invariance \eqref{vv99q} can be deduced by applying \cref{lem:pushpullcommutes} to the following commutative diagram of morphisms, where the horizontal morphisms are the corresponding Abel-Jacobi maps and the vertical morphisms are isomorphisms
\[\begin{tikzcd}
\cat{Div}_{g,A} \arrow[r] \arrow[d, "\widehat \epsilon"] &\Picabs_{g,n,d} \arrow[d,"\epsilon"]\\
\cat{Div}_{g,-A} \arrow[r]  &\, \Picabs_{g,n,-d}\, .
\end{tikzcd}
\]
Here, in the language of Section \ref{sec:log_def}, the morphism $\widehat \epsilon$ is induced by the natural map $\pi_*\bar M_C^{gp}/\bar M_S^{gp} \to \pi_*\bar M_C^{gp}/\bar M_S^{gp}$ given by inversion in $\bar M_C^{gp}$. \qed

We now prove the invariance 
\begin{equation}\label{ww99}
\P_{g,-A,-d}^\bullet = \epsilon^*\P_{g,A,d}^\bullet
\end{equation}
using the formulas for these cycles from \cref{Lem:Pixformulafactorization}. The equality is then implied by the following observations:
\begin{itemize}
\item We write  $\ca L_A = \ca L(- \sum_{i=1}^n a_i p_i)$ for the twisted universal line bundle on the universal curve $\pi: \frak C \to \Picabs_{g,n,d}$, and we use \cref{Lem:vertterm} to obtain
\begin{align*}
    \epsilon^* \left(-\eta + \sum_{i=1}^n 2 a_i \xi_i+ a_i^2 \psi_i \right) &= - \epsilon^* \pi_* c_1(\mathcal{L}_A)^2 = - \pi_* c_1((\mathcal{L}_A)^*)^2\\ &= - \pi_* (-c_1(\mathcal{L}_A))^2= - \pi_* c_1(\mathcal{L}_A)^2\\
    &=-\eta + \sum_{i=1}^n 2 a_i \xi_i+ a_i^2 \psi_i\, .
\end{align*}
\item Given a prestable graph $\Gamma_\delta$ describing a stratum in $\Picabs_{g,n,-d}$, the map $\epsilon$ sends this stratum isomorphically to the stratum of $\Gamma_{-\delta}$ (with an associated commutative diagram of gluing morphisms over $\frak M_{g,n}$). Combined with the equality $c_A(\Gamma_\delta) = c_{-A}(\Gamma_{-\delta})$ for $\Gamma_\delta \in \G_{g,n,d}^\textup{se}$ , we see  that the first line of formula \eqref{eqn:Pixformulafactorization} for $\P_{g,A,d}^\bullet$ has the desired invariance.
\item For the sum over graphs and weightings, we clearly have $h^1(\Gamma_\delta) = h^1(\Gamma_{-\delta})$ and $\Aut(\Gamma_\delta)=\Aut(\Gamma_{-\delta})$. Moreover, we have a natural bijection of the admissible weightings modulo $r$
\[\mathsf{W}_{\Gamma_\delta,r} \to \mathsf{W}_{\Gamma_{-\delta},r}\, , \ \ \ w \mapsto (h \mapsto r-w(h) \ \text{mod}\ r).\]
The  map of weightings leaves the edge terms of the formula \eqref{eqn:Pixformulafactorization} invariant  since they only depend on products $w(h)w(h')$ for and edge $(h,h')$ --  which are of the form $a(r-a)$ and thus sent to $(r-a)a$.
\end{itemize}
Therefore, the formula of Proposition \ref{Lem:Pixformulafactorization} applied to 
the two sides of \eqref{ww99} yields the same result. \qed


\subsection{Proof of Invariance \invnrn{}: {\rm Unweighted markings}}
Assume we have an additional section $p_{n+1}: \mathcal{S} \rightarrow \mathcal{C}$ of $\pi$ which yields an object of $\Picabs_{g,n+1,d}$,
    \begin{equation}
    \label{ff44996}    
    \pi:\mathcal{C} \rightarrow \mathcal{S}\, , \ \ \ \ p_1,\ldots,p_n,p_{n+1}: \mathcal{S} \rightarrow \mathcal{C}\, , \ \ \ \
        \mathcal{L} \rightarrow \mathcal{C}\, .
        \end{equation}
 Then, for the vector $A_0 \in \mathbb{Z}^{n+1}$ obtained by appending $0$ (as the last coefficient) to $A$, we want to show the invariance
 \begin{equation} \label{eqn:n_invar_repeat}
     \mathsf{DR}^{\mathsf{op}}_{g,A_0, \mathcal{L}}
        =\mathsf{DR}^{\mathsf{op}}_{g,A, \mathcal{L}}
        \, .
 \end{equation}

For the map
$\frak M_{g,n+1} \to \frak M_{g,n}$
induced by forgetting the last marking, we have a diagram of cartesian squares
\begin{equation}
 \begin{tikzcd}
  \cat{Div}_{g,A_0} \arrow[r]\arrow[d]  & \cat{Div}_{g,A} \arrow[d]\\
  \Picabs_{g,n+1,d} \arrow[r,"F"] \arrow[d] & \Picabs_{g,n,d} \arrow[d]\\
  \frak M_{g, n+1} \arrow[r] & \frak M_{g,n}\, , 
\end{tikzcd}
\end{equation}
where the morphism $F$ is syntomic. In particular, for $\bd 0 \in \mathbb{Z}^n$ the zero vector, the stack $\cat{Div}_{g, \bd 0}$ can be obtained by pulling back
$\cat{Div}_{g,\emptyset}$ from $\Picabs_{g,0,0}$. 
Then, as a consequence of the above cartesian square and the
definition of the double ramification cycle, \cref{lem:pushpullcommutes} yields
\begin{equation} \label{eqn:DR_forget_pullback}
    F^* \mathsf{DR}^{\mathsf{op}}_{g,A} = \mathsf{DR}^{\mathsf{op}}_{g,A_0}\ .
\end{equation}
Since the morphisms $\ca S \to \Picabs_{g,n,d}$ and $\ca S \to \Picabs_{g,n+1,d}$ used to define $\mathsf{DR}^{\mathsf{op}}_{g,A, \mathcal{L}}$ and $\mathsf{DR}^{\mathsf{op}}_{g,A_0, \mathcal{L}}$ fit in a diagram
\[
\begin{tikzcd}
\ca S \arrow[d] \arrow[dr]& \\
\Picabs_{g,n+1,d} \arrow[r,"F"] & \Picabs_{g,n,d}
\end{tikzcd}
\]
the equation \eqref{eqn:DR_forget_pullback} immediately proves the invariance \eqref{eqn:n_invar_repeat}. \qed


We now prove the corresponding invariance
\begin{equation} \label{eqn:P_forget_pullback}
    F^* \P_{g,A,d}^\bullet = \P_{g,A_0,d}^\bullet\ 
\end{equation}
of Pixton's formula. First, since the map $F$ does not change the curve or the line bundle, we have a Cartesian diagram
\[
\begin{tikzcd}
\frak C_{g,n+1,d} \arrow[r,"\widehat{F}"] \arrow[d,"\pi'"]& \frak C_{g,n,d} \arrow[d,"\pi"]\\
\Picabs_{g,n+1,d} \arrow[r,"{F}"]& \Picabs_{g,n,d}\, .
\end{tikzcd}
\]
The universal line bundles $\ca L_{g,n,d}$ and $\ca L_{g,n+1,d}$ on $\frak C_{g,n,d}$ and $\frak C_{g,n+1,d}$ satisfy  $$\widehat{F}^* \ca L_{g,n,d} = \ca L_{g,n+1,d}\, .$$ Similarly, for the canonical line bundles of $\pi, \pi'$ we have $\widehat{F}^* \omega_\pi = \omega_{\pi'}$. Combining these facts, we see that $F$ pulls back the operational
classes $\eta, \xi_i, \psi_i$ on $\Picabs_{g,n,d}$ 
to the corresponding classes on $\Picabs_{g,n+1,d}$.

Next, given a graph $\Gamma_\delta \in \G_{g,n,d}$, we have a fibre diagram
\begin{equation} \label{eqn:forgetgluefibre}
\begin{tikzcd}
 \coprod_{v \in \V(\Gamma)} \Picabs_{\Gamma_{v,\delta}} \arrow[r,"\coprod_v j_{\Gamma_{v,\delta}}"] \arrow[d] & \Picabs_{g,n+1,d} \arrow[d,"\pi"]\\
 \Picabs_{\Gamma_{\delta}} \arrow[r,"j_{\Gamma_{\delta}}"] & \Picabs_{g,n,d}
\end{tikzcd}
\end{equation}
where, for $v \in \V(\Gamma)$, we denote by $\Gamma_{v,\delta} \in \G_{g,n+1,d}$ the graph obtained from $\Gamma_\delta$ by adding marking $n+1$ at $v$ and leaving the remaining data fixed. 

Using the expression \eqref{eqn:Pixformulafactorization} given in \cref{Lem:Pixformulafactorization}, we conclude the proof of \eqref{eqn:P_forget_pullback} by the following observations:
\begin{itemize}
    \item The invariance of $\eta, \xi_i, \psi_i$ under $F$ implies
    \[F^*(-\eta + \sum_{i=1}^n 2 a_i \xi_i+ a_i^2 \psi_i ) = -\eta + \sum_{i=1}^{n+1} 2 a_i \xi_i+ a_i^2 \psi_i, \]
    where we use $a_{n+1}=0$.
    \item From the fibre diagram \eqref{eqn:forgetgluefibre} and the equality $c_{A}(\Gamma_\delta) = c_{A_0}(\Gamma_{v,\delta})$ for $\Gamma_\delta \in \G_{g,n,d}^\textup{se}$ and any $v \in \V(\Gamma)$, the sum over the terms $c_A(\Gamma_\delta) [\Gamma_\delta]$ pulls back correctly.
    \item For all $\Gamma_\delta \in \G_{g,n,d}$ and $v \in \V(\Gamma)$, $h^1(\Gamma_\delta) = h^1(\Gamma_{v,\delta})$ since the Betti number is independent of the position of the markings. Moreover, the automorphism group $\Aut(\Gamma_\delta)$ acts on $\V(\Gamma_\delta)$ and by the orbit-stabilizer formula, the size of the orbit  $\Aut(\Gamma_\delta) \cdot v$ of $v$ and the size of its stabilizer $\Aut(\Gamma_\delta)_v$ satisfy
    \begin{equation} \label{eqn:orbitstabilizerGamdelt}|\Aut(\Gamma_\delta) \cdot v| = \frac{|\Aut(\Gamma_\delta)|}{|\Aut(\Gamma_\delta)_v|}\, .\end{equation}
    The stabilizer $\Aut(\Gamma_\delta)_v$ is exactly equal to the automorphism group $\Aut(\Gamma_{v,\delta})$ of the graph $\Gamma_{v,\delta}$, since the marking $n+1$ at $v$ forces this vertex to be fixed. 
    \item As $\Gamma_\delta$ runs through $\G_{g,n,d}^\textup{nse}$, the graphs $\Gamma_{v,\delta}$ run through $\G_{g,n+1,d}^\textup{nse}$. The equality \eqref{eqn:orbitstabilizerGamdelt} precisely implies that the corresponding graph sums (weighted by the inverse size of automorphism groups) correspond to each other under pullback by $F$.
    \item Finally, the weightings $\mathsf{W}_{\Gamma_\delta,r}$ and $\mathsf{W}_{\Gamma_{v,\delta},r}$ are naturally bijective.
\end{itemize}
Combining the above observations, we see that also the second line of \eqref{eqn:Pixformulafactorization} transforms under pullback of $F$ as expected. \qed

\subsection{Proof of Invariance \invnrA{}: {\rm Weight translation}}

Let $B=(b_1,\ldots,b_n) \in \bb Z^n$ satisfy $\sum_{i=1}^n b_i=e$, then for the family 
    \begin{equation}
    \label{ff44997}    
    \pi:\mathcal{C} \rightarrow \mathcal{S}\, , \ \ \ \ p_1,\ldots,p_n: \mathcal{S} \rightarrow \mathcal{C}\, , \ \ \ \
        \mathcal{L}\big(\sum_{i=1}^n b_i p_i\big) \rightarrow \mathcal{C}\, .
        \end{equation}
 defining an object of $\Picabs_{g,n,d+e}$ we want to show  the invariance
 \begin{equation} \label{eqn:A_invar_repeat}
 \mathsf{DR}^{\mathsf{op}}_{g,A+B, \mathcal{L}(\sum_i b_i p_i)}
        =\mathsf{DR}^{\mathsf{op}}_{g,A, \mathcal{L}}
        \, .
 \end{equation}
To show this, consider the smooth map 
\begin{equation}
\tau_B\colon \Picabs_{g,n,d} \to \Picabs_{g,n,d+e} \, , \ \ \ 
\ca L \mapsto \ca L \big(\sum_{i=1}^n b_i p_i\big)\, . 
\end{equation}
over $\frak M_{g,n}$. We have  a natural cartesian diagram
\begin{equation}
 \begin{tikzcd}
  \cat{Div}_{g,A} \arrow[r]\arrow[d]  & \cat{Div}_{g,A+B} \arrow[d]\\
  \Picabs_{g,n,d} \arrow[r, "\tau_B"]  & \Picabs_{g,n,d+e}\, .  
\end{tikzcd}
\end{equation}
In particular, for any ramification data
$A$, we can obtain $\cat{Div}_{g,A}$ from $\cat{Div}_{g,\bd 0}$ by such translations. The diagram above together with \cref{lem:pushpullcommutes} implies
\begin{equation} \label{eqn:varying_A}
    \tau_B^* \mathsf{DR}^{\mathsf{op}}_{g,A+B} = \mathsf{DR}^{\mathsf{op}}_{g,A}\ .
\end{equation}
Since the morphisms $\ca S \to \Picabs_{g,n,d}$ and $\ca S \to \Picabs_{g,n,d+e}$ used to define $\mathsf{DR}^{\mathsf{op}}_{g,A, \mathcal{L}}$ and $\mathsf{DR}^{\mathsf{op}}_{g,A+B, \mathcal{L}(\sum_i b_i p_i)}$ fit in a diagram
\[
\begin{tikzcd}
\ca S \arrow[d] \arrow[dr]& \\
\Picabs_{g,n,d} \arrow[r,"\tau_B"] & \Picabs_{g,n,d+e}
\end{tikzcd}
\]
the equation \eqref{eqn:varying_A} immediately proves the invariance \eqref{eqn:A_invar_repeat}. \qed


Now we prove the invariance
\begin{equation} \label{eqn:varying_A_for_P}
    \tau_B^* \P_{g,A+B,d+e}^\bullet = \P_{g,A,d}^\bullet\ .
\end{equation}
for Pixton's formula.
Recall the notation $\ca L_A = \ca L(- \sum_i a_i [p_i])$ for the twisted universal line bundles from  \cref{Lem:vertterm}, observe that in the Cartesian diagram
\[
\begin{tikzcd}
\frak C_{g,n,d} \arrow[r,"\widehat \tau_B"] \arrow[d,"\pi'"] & \frak C_{g,n,d+e} \arrow[d,"\pi"]\\
\Picabs_{g,n,d} \arrow[r,"\tau_B"] & \Picabs_{g,n,d+e}
\end{tikzcd}
\]
we have $\widehat \tau_B^* \ca L_{A+B} = \ca L_{A}$. By \cref{Lem:vertterm}, we have
\begin{align*}
\tau_B^*(-\eta + \sum_{i=1}^n 2 (a_i+b_i) \xi_i+ (a_i+b_i)^2 \psi_i) &= - \tau_B^*\pi_* c_1(\mathcal{L}_{A+B})^2\\
&=- \pi'_* c_1(\mathcal{L}_{A})^2 = -\eta + \sum_{i=1}^n 2 a_i \xi_i+ a_i^2 \psi_i\, ,
\end{align*}
which shows the compatibility of the first part of formula \eqref{eqn:Pixformulafactorization} for $\P_{g,A+B,d+e}^\bullet$.

For the second part, we can combine the exponential of the graph sum over $\G_{g,n,d+e}^\textup{se}$ with the graph sum over $\G_{g,n,d+e}^\textup{nse}$ 
in \eqref{eqn:Pixformulafactorization}
to recover the sum
\begin{equation} \label{eqn:taubfullsum}
    \sum_{
\substack{\Gamma_\delta\in \G_{g,n,d+e} \\
w\in \mathsf{W}_{\Gamma_\delta,r}}
}
\frac{r^{-h^1(\Gamma_\delta)}}{|\Aut(\Gamma_\delta)| }
\;
j_{\Gamma_\delta*}\Bigg[
\prod_{e=(h,h')\in \E(\Gamma_\delta)}
\frac{1-\exp\left(-\frac{w(h)w(h')}2(\psi_h+\psi_{h'})\right)}{\psi_h + \psi_{h'}} \Bigg]
\end{equation}
over all graphs in $\G_{g,n,d+e}$ as in the proof of \cref{Lem:Pixformulafactorization}. It will be more convenient to simply show the compatibility of the full graph sum \eqref{eqn:taubfullsum} under pullback by $\tau_B$.

Given a graph $\Gamma_\delta \in \G_{g,n,d+e}$, denote by $\delta^B : \V(\Gamma) \to \mathbb{Z}$ the map defined by
\[\delta^B(v) = \delta(v) - \sum_{\substack{i\text{ marking}\\\text{at }v}}b_i\, .\]
We have a fibre diagram
\begin{equation} \label{eqn:taubshiftgluing}
\begin{tikzcd}
\Picabs_{\Gamma_{\delta^B}} \arrow[r,"j_{\Gamma_{\delta^B}}"] \arrow[d] & \Picabs_{g,n,d} \arrow[d,"\tau_B"]\\
\Picabs_{\Gamma_{\delta}} \arrow[r,"j_{\Gamma_{\delta}}"] & \Picabs_{g,n,d+e}\, .
\end{tikzcd}    
\end{equation}
The proof that \eqref{eqn:taubfullsum} pulls back under  $\tau_B$ as desired follows from the following observations:
\begin{itemize}
    \item As $\Gamma_\delta$ runs through $\G_{g,n,d+e}$, the graphs $\Gamma_{\delta^B}$ run through $\G_{g,n,d}$. From the definitions, we verify that the conditions defining admissible weightings $w$ mod $r$ for $\Gamma_\delta$ and $\Gamma_{\delta^B}$ are identical (the shift from $\delta$ to $\delta^B$ cancels the shift from $A+B$ to $A$). 
    \item Since the underlying graphs of $\Gamma_\delta$ and $\Gamma_{\delta^B}$ agree, we have $h^1(\Gamma_\delta) = h^1(\Gamma_{\delta^B})$. Concerning the automorphisms,
    they appear to take into account the degree functions $\delta, \delta^B$ on the graphs. But any vertex $v$ such that $\delta(v) \neq \delta^B(v)$ must carry a marking and thus must anyway be fixed under an automorphism. Hence, $\Aut(\Gamma_\delta) = \Aut(\Gamma_{\delta^B})$.
    \item To conclude using the diagram \eqref{eqn:taubshiftgluing}, we observe that the map $$\Picabs_{\Gamma_{\delta^B}} \to \Picabs_{\Gamma_\delta}$$ appearing there is a map over $\frak M_{\Gamma}$ (since only the line bundle is changed). Hence, the 
    classes $\psi_h, \psi_{h'}$ appearing in the edge terms of \eqref{eqn:taubfullsum} are invariant.
\end{itemize}
\qed

\subsection{Proof of Invariance \invnrII{}: {\rm Twisting by pullback}}
Let $\mathcal{B} \rightarrow \mathcal{S}$ be a line bundle on the
base. We obtain a new object of $\Picabs_{g,n,d}$ over $\mathcal{S}$ by tensoring \eqref{eqn:invSobj} with $\pi^*\mathcal{B}$:
\begin{equation}
\label{eqn:invII}    
\pi:\mathcal{C} \rightarrow \mathcal{S}\, , \ \ \ \ p_1,\ldots,p_n: \mathcal{S} \rightarrow \mathcal{C}\, , \ \ \ \
\mathcal{L}\otimes \pi^*\mathcal{B} \rightarrow \mathcal{C}\, .
\end{equation}
We want to show the invariance
$$\mathsf{DR}^{\mathsf{op}}_{g,A, \mathcal{L}\otimes\pi^*\mathcal{B}}
=\mathsf{DR}^{\mathsf{op}}_{g,A, \mathcal{L}}
\, . $$

The universal twisted double ramification cycle $\DRop_{g,A}$ is the class associated to the Abel-Jacobi map $$\aj\colon \cat{Div}_{g,A} \to \Picabs_{g,n}\, ,$$
and the latter is constructed (\cref{def:div}) by pulling back the morphism $$\aj^{\rel}\colon \cat{Div}^{\rel}_{g,A} \to \Picrel_{g,n}\, .$$
Thus by \cref{lem:pushpullcommutes}, the corresponding cycle $\DRop_{g,A}$ is a pullback of $\aj^{\rel}_{\mathsf{op}}[\cat{Div}^{\rel}_{g,A}]$ from $\Picrel_{g,n,d}$.
But twisting the family \eqref{eqn:invSobj} by a line bundle pulled back from the base does not change the map to $\Picrel_{g,n,d}$, and so does not change the resulting operational class, proving the invariance.  \qed


We now prove the invariance
$$\phi_{\ca L\otimes \pi^*\ca B}^*\P_{g,A,d}^\bullet=
\phi_\ca L^*\P_{g,A,d}^\bullet  \,  . $$
Using that $\P_{g,A,d}^\bullet$ is a pullback of $\P_{g,\emptyset,0}^\bullet$ as described in \cref{sec:intro_deg_0}, it suffices to show that the cycle $\P_g^\bullet$ on $\Picabs_{g,0,0}$ pulls back to the same expression under the two maps 
\[\phi_{\ca L_A}\, , \,
\phi_{\ca L_A \otimes \pi^* \ca B}\,  : \ca S \to \Picabs_{g,0,0}\]
induced by $\mathcal{L}_{A} = \mathcal{L}(-\sum_{i=1}^n a_i p_i)$ and $\ca L_{A} \otimes \pi^* \ca B$ respectively.
 We will  use formula \eqref{eqn:P_gformula2}
 for $\P_g^\bullet$   and show that both parts of the formula are invariant separately.

\noindent $\bullet$ For the term $\exp\left(-\frac12 \eta \right)$,  we use \cref{Lem:vertterm} to obtain
\begin{align*}
\phi_{\ca L_A \otimes \pi^* \ca B}^* \eta &= \pi_* c_1( \mathcal{L}_{A} \otimes \pi^* \ca B)^2\\ &= \pi_* \left(c_1( \mathcal{L}_{A})^2 + 2 c_1(\mathcal{L}_{A}) \pi^* c_1(\ca B) + \pi^* c_1(\ca B)^2\right) \\
&= \pi_* \left(c_1( \mathcal{L}_{A})^2\right) + 2 c_1(\ca B) \underbrace{\pi_* c_1(\mathcal{L}_{A})}_{=0}  +  c_1(\ca B)^2 \underbrace{\pi_* 1}_{=0} = \phi_{\ca L_A}^* \eta,
\end{align*}
where 
$\pi_* c_1(\mathcal{L}_{A})$
vanishes since $\mathcal{L}_{A}$ has degree $0$, and $\pi_*1$ vanishes for dimension reasons. The vertex term is therefore
invariant under twisting by $\pi^* \ca B$.

\noindent $\bullet$ For the graph sum 
\begin{equation} \label{eqn:InvIIgraphsum}
\sum_{
\substack{\Gamma_\delta\in \G_{g,0,0} \\
w\in \mathsf{W}_{\Gamma_\delta,r}}
}
\frac{r^{-h^1(\Gamma_\delta)}}{|\Aut(\Gamma_\delta)| }
\;
j_{\Gamma_\delta*}\Bigg[
\prod_{e=(h,h')\in \E(\Gamma_\delta)}
\frac{1-\exp\left(-\frac{w(h)w(h')}2(\psi_h+\psi_{h'})\right)}{\psi_h + \psi_{h'}} \Bigg]    
\end{equation}
in \eqref{eqn:P_gformula2}, we show that it is a pullback under the morphism 
\begin{equation}\label{zz25}
\Picabs_{g,0,0} \to \Picrel_{g,0,0}\, ,
\end{equation}
which finishes the proof since the compositions of $\phi_{\ca L_A}$ and $\phi_{\ca L_A \otimes \pi^* \ca B}$ with the morphism
\eqref{zz25} agree. For a prestable graph $\Gamma$,
we have Cartesian diagrams
\begin{equation}
\begin{tikzcd}
\coprod_\delta \Picabs_{\Gamma_\delta} \arrow[r,"\coprod_\delta j_{\Gamma_\delta}"] \arrow[d] & \Picabs_{g,0,0} \arrow[d]\\
\coprod_\delta \Picrel_{\Gamma_\delta} \arrow[r,"\coprod_\delta j^\textup{rel}_{\Gamma_\delta}"] \arrow[d] & \Picrel_{g,0,0} \arrow[d]\\
\frak M_\Gamma \arrow[r,"j_{\Gamma}"] & \, \frak M_g\, .
\end{tikzcd}
\end{equation}
In formula  \eqref{eqn:InvIIgraphsum}, the edge terms use only the classes $\psi_h, \psi_{h'}$, which are pullbacks from $\frak M_\Gamma$. Therefore, \eqref{eqn:InvIIgraphsum} is the pullback under \eqref{zz25} of the identical formula with $j_{\Gamma_\delta}$ replaced with $j^\textup{rel}_{\Gamma_\delta}$. \qed



\subsection{Proof of Invariance \invnrIII{}: {\rm Vertical twisting}}
Consider the boundary divisor $\Delta$ of $\Picabs_{g,n,d}$ given by the partition 
\begin{equation*}
g_1+g_2=g\, , \ \ \ N_1\sqcup N_2 = \{1,\ldots,n\}\,, \ \ \ d_1+d_2=d\, 
\end{equation*}
which is not symmetric.
In  $\mathfrak{C}_{g,n,d} \to \Picabs_{g,n,d}$, let $\Delta_1$ and $\Delta_2$ be the $(g_1,N_1,d_1)$
and $(g_2,N_2,d_2)$ components respectively
of the universal curve over $\Delta$.
Then we have a morphism
\[\Phi_{\Delta_1} : \Picabs_{g,n,d} \to \Picabs_{g,n,d}\]
associated to the twisted line bundle $\mathcal{L}(\Delta_1)$ on the universal curve $\mathfrak{C}_{g,n,d} \to \Picabs_{g,n,d}$. 

\begin{remark}
The map $\Phi_{\Delta_1}$ is \emph{not} an isomorphism. Indeed, the map is equal to the identity away from $\Delta \subset \Picabs_{g,n,d}$, but it sends the generic point of $\Delta$ to the generic point of the boundary divisor 
$$\widetilde \Delta = \Delta(g_1, N_1, d_1-1|g_2, N_2, d_2+1)\, ,$$ which itself is fixed under $\Phi_{\Delta_1}$. Hence $\Phi_{\Delta_1}$ is not injective, though it is easily seen to be \'etale. 
\end{remark}



We want to show the invariance 
\begin{equation}\label{599e}
\Phi_{\Delta_1}^* \mathsf{DR}^{\mathsf{op}}_{g,A} = \mathsf{DR}^{\mathsf{op}}_{g,A}\, .
\end{equation}
Using the data 
\begin{equation*}
g_1+g_2=g\, , \ \ \ N_1\sqcup N_2 = \{1,\ldots,n\}\,, \ \ \ d_1+d_2=d\, ,
\end{equation*}
we will define a map $$\Phi'\colon\cat{Div}_{g,A}\to\cat{Div}_{g,A}\, .$$
In fact, we will define a map $\cat{Div}^{\rel}_{g,A}\to\cat{Div}^{\rel}_{g,A}$ and then lift it to $\cat{Div}_{g,A}$ by fibre product with $\Picabs$. 
The invariance \eqref{599e} will be deduced from $\Phi'$.

Suppose we are given a map $S \to \cat{Div}^{\rel}_{g,A}$ compatible with $C/S$ defined by a $\bar M_S^{gp}$ torsor $\ca P$ on $S$ and a map $\alpha\colon C \to \ca P$. 
The divisor $\Delta$ determines an element of the characteristic sheaf of the log structure on $\Picabs_{g,n,d}$, which pulls back under the composition $$S \to \cat{Div}^{\rel}_{g,A} \xrightarrow{\aj} \Picabs_{g,n,d}$$
to an element $\delta \in \bar M_S(S)$. All lifts of $\delta$ to $M_S(S)$ generate the same ideal sheaf on $S$, whose closed subscheme is exactly the pullback $\Delta_S$ of $\Delta$.
We write $$\Delta_1', \Delta_2' \subset C$$ for the two components of the universal curve over $\Delta_S \subset S$.



We define a new map $(\ca P', \alpha')\colon S\to \cat{Div}^{\rel}_{g,A}$ as follows. We take the same torsor $\ca P' = \ca P$. On the locus $C_1 \hookrightarrow C$ which is the complement of $\Delta'_1$, we define $\alpha' = \alpha$. On the locus $C_2 \hookrightarrow C$ which is the complement of $\Delta'_2$ we define $\alpha' = \alpha - \delta$. Since $\delta$ vanishes on the overlap $C_1 \cap C_2$, we just have to check that the
defined section extends from $C_1 \cup C_2$ to the whole of $C$ (across the separating node $\Delta'_1 \cap \Delta'_2$). The
extension
can be checked \'etale locally, and then the claim follows from the local description of the log structure.

We have defined a map $\Phi'\colon\cat{Div}_{g,A} \to \cat{Div}_{g,A}$, and we verify easily that the diagram 
\begin{equation}\label{inv_3_diagram}
 \begin{tikzcd}
  \cat{Div}_{g,A} \arrow[r,"\aj"]\arrow[d,"\Phi'"] & \Picabs_{g,n,d}\arrow[d,"\Phi_{\Delta_1}"]\\
    \cat{Div}_{g,A} \arrow[r, "\aj"] & \Picabs_{g,n,d} 
\end{tikzcd}
\end{equation}
commutes. We will prove \cref{inv_3_diagram} is a pullback square
which by \cref{lem:pushpullcommutes} yields the invariance \eqref{599e}.

To prove  \cref{inv_3_diagram} is a pullback square, since the horizontal arrows are monomorphisms, 
we need to show the following: given $(\ca P, \alpha) \in \cat{Div}^{\rel}_{g,A}(S)$, a line bundle $\ca L$ on $C$ and an isomorphism $\aj(\ca P, \alpha) \isom \Phi_{\Delta_1}(\ca L)$, there exists $(\ca P_0, \alpha_0) \in \cat{Div}^{\rel}_{g,A}(S)$ such that $\aj(\ca P_0, \alpha_0) \cong \ca L$.  If the element $(\ca P, \alpha)$ has only one preimage under $\Phi'$, we are done by commutativity. If there are two preimages (the only other case), the  Abel-Jacobi images differ by a twist by $\Delta'_1$, and the bundle $\ca L$ will determine which we choose. More formally, by uniqueness, we may assume $S$ to be strictly henselian local, then $\Phi'$ has exactly one preimage whenever $\Delta(g_1,N_1,d_1-1|g_2,N_2,d_2+1)$ does not meet $S$, and the result is clear as the diagram commutes. If on the other hand $\Delta(g_1,N_1,d_1-1|g_2,N_2,d_2+1)$ does meet $S$, then the two preimages under $\Phi'$ will have multidegrees $(d_1,d_2)$ and $(d_1-1,d_2+1)$, and only one of these can be sent by the Abel-Jacobi map to $\ca L$. \qed


We now prove the invariance 
\[\Phi_{\Delta_1}^* \P_{g,A,d}^\bullet = \P_{g,A,d}^\bullet\, .\]
We will use \cref{Lem:Pixformulafactorization} and prove that the two lines of formula \eqref{eqn:Pixformulafactorization} are separately
invariant.

\noindent $\bullet$ For the exponential term, we must show that the divisor
\[-\eta + \sum_{i=1}^n 2 a_i \xi_i+ a_i^2 \psi_i  + \sum_{\Gamma_\delta \in \G_{g,n,d}^\textup{se}} c_A(\Gamma_\delta) [\Gamma_\delta]\]
is invariant. By \cref{Lem:vertterm}, we see  
\[- \pi_* c_1(\mathcal{L}_A)^2 =
-\eta + \sum_{i=1}^n 2 a_i \xi_i+ a_i^2 \psi_i \, .\]
After pulling back via $\Phi_{\Delta_1}$, we obtain
\begin{align*}
- \pi_* c_1(\mathcal{L}_A(\Delta_1))^2 &= - \pi_* \left( c_1(\mathcal{L}_A)^2 + 2 c_1(\mathcal{L}_A) \Delta_1 + \Delta_1^2\right)\\
&= - \pi_* \left( c_1(\mathcal{L}_A)^2\right)) -2 \mathrm{deg}\left(\mathcal{L}_A|_{\Delta_1}\right) \pi_* \Delta_1 - \pi_* \left( \Delta_1 \cdot (\pi^* \Delta - \Delta_2)\right)\\
&= - \pi_* \left( c_1(\mathcal{L}_A)^2\right)) -2 \left(d_1 - \sum_{i \in N_1} a_i\right) \Delta + \Delta\, .
\end{align*}
We have used $\pi^* \Delta = \Delta_1 + \Delta_2$
and that the intersection of $\Delta_1$ and $\Delta_2$ has degree $1$ over $\Delta$. 


For the pullback of the linear combination 
\[\sum_{\Gamma_\delta \in \G_{g,n,d}^\textup{se}} c_A(\Gamma_\delta) [\Gamma_\delta]\]
of boundary divisors, recall the divisor $\widetilde \Delta = \Delta(g_1, N_1, d_1-1|g_2, N_2, d_2+1)$ which satisfies $$\widetilde \Delta = \Phi_{\Gamma_1}(\widetilde \Delta) = \Phi_{\Gamma_1}(\Delta)\, .$$
We see $\Phi_{\Delta_1}^* [\Gamma_\delta] = [\Gamma_\delta]$ for all boundary divisors $[\Gamma_\delta]$ different from $[\Delta]$ and $[\widetilde \Delta]$. Moreover,
$$\Phi_{\Delta_1}^* [\widetilde \Delta] = [\widetilde \Delta] + [\Delta]\, , \ \ \ \Phi_{\Delta_1}^* [\Delta] = 0\, .$$
Writing $\Gamma_{\Delta}, \Gamma_{\widetilde \Delta}$ for the graphs associated to $\Delta, \widetilde \Delta$, we see
\[\Phi_{\Delta_1}^* \sum_{\Gamma_\delta \in \G_{g,n,d}^\textup{se}} c_A(\Gamma_\delta) [\Gamma_\delta] = \sum_{\Gamma_\delta \in \G_{g,n,d}^\textup{se}} c_A(\Gamma_\delta) [\Gamma_\delta] + \left(c_A(\Gamma_{\widetilde \Delta}) - c_A(\Gamma_\Delta) \right) [\Delta]\,.\]
After expanding the last term further, we obtain the coefficient
\[c_A(\Gamma_{\widetilde \Delta}) - c_A(\Gamma_\Delta) = - (d_1 - 1 - \sum_{i \in I_1}a_i)^2 + (d_1  - \sum_{i \in I_1}a_i)^2 = 2 (d_1 - \sum_{i \in I_1} a_i) -1\,, \]
which exactly balances out the error term we obtained in the pullback of $- \pi_* c_1(\mathcal{L}_A(\Delta_1))^2$.
We have finished the proof of the invariance of the exponential term in \eqref{eqn:P_gformula2}.

\noindent $\bullet$ 
For the invariance of the sum over $\Gamma_\delta\in \G_{g,n,d}^\textup{nse}$, we claim that given any graph $\Gamma_\delta$, we have the diagram
\begin{equation}
\begin{tikzcd}
\coprod_{\sum_v \delta(v)=d} \Picabs_{\Gamma_\delta} \arrow[r, "\coprod_{\delta} j_{\Gamma_\delta}"] \arrow[d, "\Phi_{\Delta_1,\Gamma}"] &\Picabs_{g,n,d} \arrow[d, "\Phi_{\Delta_1}"]\\
\coprod_{\sum_v \delta(v)=d} \Picabs_{\Gamma_\delta} \arrow[r, "\coprod_{\delta} j_{\Gamma_\delta}"] \arrow[d] &\Picabs_{g,n,d} \arrow[d]\\
\frak M_\Gamma \arrow[r,"j_\Gamma"] & \, \frak M_{g,n}\, .
\end{tikzcd}
\end{equation}
The lower and outer diagrams are cartesian as we have seen in Section \ref{Sect:Prestgrwdeg}, thus the upper diagram is also cartesian. While for a general graph $\Gamma_\delta$ the map $\Phi_{\Delta_1,\Gamma}$ induces a nontrivial map on the set of components of $\coprod_{\delta} \Picabs_{\Gamma_\delta}$, for $\Gamma$ having only nonseparating edges, we obtain  $$\Phi_{\Delta_1, \Gamma} : \Picabs_{\Gamma_\delta} \to \Picabs_{\Gamma_\delta}$$ over $\frak M_{\Gamma}$. 
The classes $\psi_h$ (for $h \in H(\Gamma)$) are pullbacks from $\frak M_\Gamma$, in particular $\Phi_{\Delta_1,\Gamma}^* \psi_h = \psi_h$ for all such $h$.
As a result, each term
\[
j_{\Gamma_\delta*}\Bigg[
\prod_{e=(h,h')\in \E(\Gamma_\delta)}
\frac{1-\exp\left(-\frac{w(h)w(h')}2(\psi_h+\psi_{h'})\right)}{\psi_h + \psi_{h'}} \Bigg]
\]
in the sum over $\Gamma_\delta\in \G_{g,n,d}^\textup{nse}$, $w\in \mathsf{W}_{\Gamma_\delta,r}$ in formula 
\eqref{eqn:Pixformulafactorization}
is indeed invariant. \qed

\subsection{Proof of Invariance \invnrIV{}: {\rm Partial stabilization}}\label{sec:proof_inv_IV} 
\subsubsection{The stack $\frak N$}
We begin by introducing a stack $\frak N$ that allows us to reformulate Invariance \invnrIV{} as an equality of operational classes on $\frak N$. The stack  $\frak N$ parametrises data
\[f \colon C' \to C\,,\ \ \mathcal{L}/C\,,\]
where $f$ is a map of prestable genus $g$ curves which is a partial stabilisation (a surjection which contracts some unstable rational components of $C'$) and $\ca L$ is a line bundle on $C$ of degree $0$.  By a small extension of the arguments of \cite{Kresch2013Flattening-stra}, 
$\frak N$ is an algebraic stack 
and comes with two maps to $\Picabs_{g,0,0}$ given by:
\begin{eqnarray*}
\targetmap: \frak N \to \Picabs_{g,0,0}\, ,  & (f \colon C' \, \to\,  C, \mathcal{L}) \mapsto (C,\ca L)\, , \\
\sourcemap: \frak N \to \Picabs_{g,0,0}\, , & \ \,(f \colon C' \to C, \mathcal{L})\,  \mapsto\,  (C',f^*\ca L)\,.
\end{eqnarray*}

\begin{lemma}\label{lem:invariance_N}
\begin{equation} \label{eqn:InvIVreformulation}
    \targetmap^*\DRop_{g,\bd 0} = (\sourcemap)^*\DRop_{g, \bd 0}\, ,\ \ \ \ \targetmap^* \P_{g}^\bullet = (\sourcemap)^* \P_g^\bullet\, .
\end{equation}
\end{lemma}
This lemma will be proven in the next subsection, but first we show that Invariance \invnrIV{} is implied  by the lemma. 

For Invariance \invnrIV, we are given the data
\[C' \xrightarrow{f} C \to \ca S\, ,\ \ \ca L \to C\, ,\ \  p_1\, , \ \ \ldots, p_n : \ca S \to C\, ,\ \ p_1', \ldots, p_n' : \ca S \to C'\]
with $f$ a partial stabilization satisfying $f \circ p_i' = p_i$ and a vector $A \in \mathbb{Z}^n$ satisfying $a_i=0$ if $p_i'$ meets the exceptional locus of $f$. Invariance \invnrIV{} then says that for the maps $$\phi_{\ca L}, \phi_{f^* \ca L} : \ca S \to \Picabs_{g,n,d}$$ induced by
\[C \to \ca S\, , \ \ p_1, \ldots, p_n : \ca S \to C\, , \ \ \ca L \to C\, ,\]
\[C' \to \ca S\, ,\ \  p_1', \ldots, p_n' : \ca S \to C'\, ,\ \  f^*\ca L \to C,\]
respectively, we have $\phi_{\ca L}^* \DRop_{g,A} = \phi_{f^*\ca L}^* \DRop_{g,A}$ and $\phi_{\ca L}^* \P^\bullet_{g,A,d} = \phi_{f^*\ca L}^* \P^\bullet_{g,A,d}$. 

The condition $a_i=0$ if $p_i'$ meets the exceptional locus of $f$ implies the equality
\[f^*\Big(\ca L\big(-\sum_{i=1}^n a_i p_i\big) \Big) = (f^* \ca L)\Big(-\sum_{i=1}^n a_i p_i'\Big)\]
of line bundles on $C'$. Denote by $f:\ca S \to \frak N$ the map associated to the data
\[C' \xrightarrow{f} C \to \ca S\, ,\ \ \ \ca L\big(-\sum_{i=1}^n a_i p_i \big) \to C\, .\]
Writing $\ca L_A = \ca L(-\sum_{i=1}^n a_i p_i)$ and $f^*\ca L_A = (f^*\ca L)(-\sum_{i=1}^n a_i p_i')$,
 we obtain a commutative diagram
\begin{equation} \label{eqn:frakNdiagram}
\begin{tikzcd}
& \ca S\arrow[d,"g"] \arrow[ddl, "\phi_{\ca L_A}", swap] \arrow[ddr, "\phi_{f^*\ca L_A}"] &\\
& \frak N \arrow[dl, "\targetmap"] \arrow[dr, "\sourcemap", swap] &\\
\Picabs_{g,0,0} & & \Picabs_{g,0,0}\, .
\end{tikzcd}    
\end{equation}
From the arguments presented in \cref{sec:intro_deg_0} it follows that
\[\phi_{\ca L}^* \DRop_{g,A} = \phi_{\ca L_A}^* \DRop_{g,\emptyset}\, ,\ \ \  \phi_{f^*\ca L}^* \DRop_{g,A} = \phi_{f^*\ca L_A}^* \DRop_{g,\emptyset}\, ,\]
with parallel equations for $\P^\bullet_{g,A,d}$ and $\P^\bullet_{g,\emptyset,0}$. 
Assuming \eqref{eqn:InvIVreformulation} we have
\begin{equation*}
    \phi_{\ca L}^* \DRop_{g,A} = \phi_{\ca L_A}^* \DRop_{g,\emptyset} = g^* \targetmap^* \DRop_{g,\emptyset}
    = g^* \sourcemap^* \DRop_{g,\emptyset} = \phi_{f^*\ca L_A}^* \DRop_{g,\emptyset} = \phi_{f^*\ca L}^* \DRop_{g,A}\, ,
\end{equation*}
and similarly for $\P^\bullet_{g,A,d}$. Thus \eqref{eqn:InvIVreformulation} implies Invariance \invnrIV. 

\subsubsection{Invariance for $\frak N$}
Here we prove Lemma \ref{lem:invariance_N}; it follows immediately from \eqref{eqn:InvIVreformulation} in Lemmas \ref{lem:N_invariance_DR} and \ref{lem:N_invariance_P} below. We start with a preliminary result.

\begin{lemma}\label{lem:lflatness}
The map
\begin{eqnarray*}
\targetmap: \frak N \to \Picabs_{g,0,0}\, ,  & (f \colon C' \, \to\,  C, \mathcal{L}) \mapsto (C,\ca L)\, , 
\end{eqnarray*}
is syntomic\footnote{Flat and lci. }, and the map 
\begin{eqnarray*}
\sourcemap: \frak N \to \Picabs_{g,0,0}\, , & \ \,(f \colon C' \to C, \mathcal{L})\,  \mapsto\,  (C',f^*\ca L)\,.
\end{eqnarray*}
is smooth. 
\end{lemma}
\begin{proof}
The stack of partial stabilisations $(f\colon C' \to C)$ has smooth charts given by $\Mbar_{g,n+m}$, where $(C, p_1, \dots, p_n, q_1, \dots, q_m)$ maps to the contraction map from $C$ to the stabilisation of $(C, p_1, \dots, p_n)$ by forgetting the $q$ markings. Charts for $\frak N$ are then given by
\[\Mbar_{g,n+m} \times_{\Mbar_{g,n}} \Picabs_{\Mbar_{g,n+1}/\Mbar_{g,n}} \xrightarrow{G} \frak N.\]

Charts for the map $\targetmap$ are given by the composition of the top horizontal arrows in the commutative diagram
\begin{equation}
\begin{tikzcd}
\square \arrow[r] \arrow[d] \arrow[rr, bend left,"\targetmap"]& \frak{Pic}_{\Mbar_{g,n+1}/\Mbar_{g,n}} \arrow[r]\arrow[d] & \Picabs_g\arrow[d] \\
\Mbar_{g,n+m} \arrow[r] & \Mbar_{g,n} \arrow[r] & \frak M_g\, .
\end{tikzcd}
\end{equation}
Both squares here are pullbacks, the bottom right horizontal map is smooth, and the bottom left horizontal map is syntomic. Hence $\targetmap$ is syntomic, using that syntomicity is a flat-local property on the target. 

Charts for the map $\sourcemap$ are given by commutative diagrams
\begin{equation}
\begin{tikzcd}
\Mbar_{g,n+m} \times_{\Mbar_{g,n}} \Picabs_{\Mbar_{g,n+1}/\Mbar_{g,n}} \arrow[d,"G"] \arrow[r,"F"] & \Picabs_{\Mbar_{g,n+m+1}/\Mbar_{g,n+m}} \arrow[d] \\
\frak N \arrow[r,"\sourcemap"] & \Picabs_g\, .
\end{tikzcd}
\end{equation}
The right vertical arrow is a base-change of the smooth map $\Mbar_{g,n+m} \to \frak{M}_g$, so once we have shown $F$ to be smooth, we can conclude using that smoothness is a flat-local property on the target. 

In fact, $F$ is an open immersion: $F$ is isomorphic to the inclusion of the locus $$U \hra \Picabs_{\Mbar_{g,n+m+1}/\Mbar_{g,n+m}}$$ of line bundles which are trivial on the contracted rational components. 
We must verify  the induced map
$$F'\colon \Mbar_{g,n+m} \times_{\Mbar_{g,n}} \Picabs_{\Mbar_{g,n+1}/\Mbar_{g,n}} \to U$$
is an isomorphism. The source and target are smooth over $\Mbar_{g, n + m}$. On each geometric fibre over $\Mbar_{g, n + m}$, the map $F'$ is an isomorphism via the explicit description of the Jacobian of a prestable curve. But then $F'$ is  flat (by the fibrewise criterion), is unramified (by a pointwise check), and is universally injective (again by a pointwise check), and hence is an isomorphism. 
\end{proof}

\begin{lemma}\label{lem:N_invariance_DR}
We have $\targetmap^*\DRop_{g,\bd 0} = (\sourcemap)^*\DRop_{g, \bd 0} \, \in\, {\mathsf{CH}}^g_{\mathsf{op}}(\frak N)$.
\end{lemma}
\begin{proof}
The maps $\targetmap$ and $\sourcemap$ are syntomic
by Lemma \ref{lem:lflatness}.
It therefore suffices by \cref{lem:pushpullcommutes} to construct an isomorphism $$\phi: \targetmap^*\cat{Div}_{g,\bd 0} \isom (\sourcemap)^*\cat{Div}_{g,\bd 0}$$
of stacks over $\frak N$
and even to construct the isomorphism
on the level of $\cat{Div}^{\rel}$. 


An object of $\targetmap^*\cat{Div}^{\rel}_{g,\bd 0}$ consists of a stabilisation map $f\colon C' \to C$, a line bundle $\ca L$ on $C$, a $\bb G_m^{trop}$ torsor $\ca P$ on $S$, and a map $\alpha \colon C \to \ca P$ such that $\ca O(\alpha )$ and $\ca L$ are isomorphic up to pullback from the base. An object of $(\sourcemap)^*\cat{Div}^{\rel}_{g,\bd 0}$ consists of almost the same data, but $\alpha$ is replaced by any $\alpha'\colon C' \to \ca P$, with $\ca O(\alpha')$ and $f^*\ca L$ isomorphic up to pullback. We then define the map $\phi$ simply by composing, setting $\alpha' = f \circ \alpha$. 

We must show that $\phi$ is an isomorphism.
Suppose we are given the data $f\colon C' \to C$, $\ca L$ on $C$, $\ca P$, $\alpha' \colon C' \to \ca P$, with $\ca O(\alpha') \cong f^*\ca L$ up to pullback. Since the degree of $f^*\ca L$ vanishes on components contracted by $f$, we see that the same is true of the degree of $\ca O(\alpha')$ -- the slopes of the restriction of $\alpha'$ to the contracted graph of the curve are linear on edges. In other words, the restriction is still piecewise linear with integer slopes,  hence we can set $\alpha$ to be the restriction. 
\end{proof}

More work will be required to  prove the second equality
$$\targetmap^* \P_{g}^\bullet = (\sourcemap)^* \P_g^\bullet\, 
.$$
We start by considering the morphism $\targetmap$. For a prestable graph $\Gamma_{\delta}$ of degree $0$, consider the diagram
\begin{equation} \label{eqn:frakNdiagramp}
\begin{tikzcd}
\prod_{v \in \V(\Gamma)}\frak N_{\g(v),\n(v),\delta(v)}  \arrow[d,"\prod_v \targetmap_v"] &\frak N'_{\Gamma_{\delta}} \arrow[l] \arrow[r,"G"] \arrow[d] &\frak N_{\Gamma_{\delta}} \arrow[r,"J_{\Gamma_{\delta}}"] \arrow[d,"\targetmap_{\Gamma_\delta}"] & \frak N \arrow[d,"\targetmap"]\\
\prod_{v \in \V(\Gamma)}\Picabs_{\g(v),\n(v),\delta(v)} & \Picabs_{\Gamma_{\delta}}\arrow[l]&\Picabs_{\Gamma_{\delta}} \ar[equal]{l}\arrow[r, "j_{\Gamma_{\delta}}"] & \Picabs_{g,0,0}
\end{tikzcd}    
\end{equation}
where the left and the right squares are pullbacks and the middle square is commutative:
\begin{enumerate}
\item[$\bullet$] The stacks $\frak N_{\g(v),\n(v),\delta(v)}$ are the natural generalizations of $\frak N$ to the case of marked curves and line bundles of arbitrary total degrees. 
\item[$\bullet$]
The stack $\frak N_{\Gamma_{\delta}}$ parametrizes data 
\[
(C_v)_{v\in \V(\Gamma)}\,, \,\, \ca L/C\,, \,\, f:C'\to C
\]
where $f$ is a partial stabilisation and $\ca L$ is a multidegree $\delta$ line bundle on $C=\sqcup_{v\in \V(\Gamma)} C_v$. 
\item[$\bullet$]
The stack $\frak N'_{\Gamma_{\delta}}$ parametrizes data
\[
(C_v)_{v\in \V(\Gamma)}\,, \,\, \ca L/C\,, \,\, (f_v:C'_v\to C_v)_{v\in \V(\Gamma)}\,.
\]
\item[$\bullet$]
The gluing map $G:\frak N'_{\Gamma_{\delta}} \to \frak N_{\Gamma_{\delta}}$ sending $(f_v:C'_v\to C_v)_{v\in \V(\Gamma)}$ to $$f : \sqcup_v C'_v = C' \to C = \sqcup_v C_v$$ is proper, representable and birational.
\end{enumerate}

Properness of $G$ can be checked using the valuative criterion. 
The difference between $\frak N_{\Gamma_\delta}$ and $\frak N'_{\Gamma_\delta}$ is that, in the first space, we have sections of the non-smooth locus of $\frak C$ for each half-edge (telling us where to cut apart the curve), whereas, for the second, the sections go to the non-smooth locus of $\frak C'$. Fibres of $G$ correspond to choices of lifts of these sections along $$f: \frak C' \to \frak C\, .$$
Existence and uniqueness of such lifts follows from properness of $f$ and of the inclusion of the non-smooth locus. 
Representability of $G$ is a short argument showing $G$ is injective on stabilizer groups. Birationality follows since $G$ is an isomorphism over the dense open locus where $f$ is an isomorphism.
Let $$\widehat J_{\Gamma_\delta} : \frak N'_{\Gamma_\delta} \to \frak N$$ be 
the composition of $G$ and $J_{\Gamma_\delta}$.

In the following Lemma, we compare pullback formulas under $\targetmap$ and $\sourcemap$ for the stacks $\frak N_{g,n,d}$ above. Denote by $\psi_i = \targetmap^*\psi_i$ and $\psi_i' = (\sourcemap)^*\psi_i \in \CHop^1(\frak N_{g,n,d})$ the pullbacks of $\psi$-classes under $\targetmap$ and  $\sourcemap$
respectively.

\begin{lemma}\label{lem:pullback_comparison} 
We have
\begin{enumerate}[label=(\arabic*)]
    \item[(i)] $\targetmap^*\eta = (\sourcemap)^*\eta$,
    \item[(ii)]  $\psi_i = \psi_i' - D_i$ where $D_i$ is the class in $\CHop^1(\frak N_{g,n,d})$ associated to the boundary divisor of $\frak N_{g,n,d}$ generically parametrizing a partial stabilisation
\[
\begin{tikzpicture}[scale=0.7, baseline=-3pt,label distance=0.3cm,thick,
    virtnode/.style={circle,draw,scale=0.5}, 
    nonvirt node/.style={circle,draw,fill=black,scale=0.5} ]
    \node [nonvirt node,label=below:$(0{,}0)$] (A) {};
    \node at (2,0) [nonvirt node,label=below:$(g{,}d)$] (B) {};
    \draw [-] (A) to (B);
    \node at (-.7,.5) (n1) {$i$};
    \draw [-] (A) to (n1);
    \node at (2.7,.5) (m1) {};
    \node at (2.7,-.5) (m2) {};
    \draw [-] (B) to (m1);
    \draw [-] (B) to (m2);    
    \node at (2.7,0.2) (m3) {$\vdots$};
    \end{tikzpicture} 
\to 
\begin{tikzpicture}[scale=0.7, baseline=-3pt,label distance=0.3cm,thick,
    virtnode/.style={circle,draw,scale=0.5}, 
    nonvirt node/.style={circle,draw,fill=black,scale=0.5} ]
    \node [nonvirt node,label=below:$(g{,}d)$] (A) {};
    \node at (-.7,.5) (n1) {$i$};
    \draw [-] (A) to (n1);
    \node at (.7,.5) (m1) {};
    \node at (.7,-.5) (m2) {};
    \draw [-] (A) to (m1);
    \draw [-] (A) to (m2);    
    \node at (.7,0.2) (m3) {$\vdots$};
    \end{tikzpicture}\ \ .
\]
\end{enumerate}
\end{lemma}
\begin{proof} 

\noindent(i) There are two pairs of universal curves with line bundle $(\frak C,\frak L)$ and $(\frak C', \frak L')$
\begin{center}
\begin{tikzcd}
\frak C' \arrow[dr,"\pi'"] \arrow[rr, "f"]
&~
& \frak C \arrow[dl,"\pi"] \\
~
&\frak N_{g,n,d}
&~
\end{tikzcd}
\end{center}
over the stack $\frak{N}_{g,n,d}$ with sections $\sigma_i:\frak N_{g,n,d}\to\frak C$ and $\sigma_i':\frak N_{g,n,d}\to \frak C'$. Because $f_*[\frak C'] = [\frak C]$, we have
\[
\targetmap^*\eta = \pi_*(c_1(\frak L)^2) = \pi_*(c_1(\frak L)^2f_*[\frak C']) = \pi'_*(c_1(f^*\frak L)^2) = (\sourcemap)^*\eta\, .
\]

\noindent (ii) Let $D'_0$ be the divisor
\[
\begin{tikzpicture}[scale=0.7, baseline=-3pt,label distance=0.3cm,thick,
    virtnode/.style={circle,draw,scale=0.5}, 
    nonvirt node/.style={circle,draw,fill=black,scale=0.5} ]
    \node [nonvirt node,label=below:$(0{,}0)$] (A) {};
    \node at (2,0) [nonvirt node,label=below:$(g{,}d)$] (B) {};
    \draw [-] (A) to (B);
    \node at (0,.7) (n2) {$\downarrow$};
    \node at (2.7,.5) (m1) {};
    \node at (2.7,-.5) (m2) {};
    \draw [-] (B) to (m1);
    \draw [-] (B) to (m2);    
    \node at (2.7,0.2) (m3) {$\vdots$};
    \end{tikzpicture} 
\to 
\begin{tikzpicture}[scale=0.7, baseline=-3pt,label distance=0.3cm,thick,
    virtnode/.style={circle,draw,scale=0.5}, 
    nonvirt node/.style={circle,draw,fill=black,scale=0.5} ]
    \node [nonvirt node,label=below:$(g{,}d)$] (A) {};
    \node at (-.7,.5) (n1) {$i$};
    \draw [-] (A) to (n1);
    \node at (.7,.5) (m1) {};
    \node at (.7,-.5) (m2) {};
    \draw [-] (A) to (m1);
    \draw [-] (A) to (m2);    
    \node at (.7,0.2) (m3) {$\vdots$};
    \end{tikzpicture}
\]
in $\frak C'$, and let $D'_i$ be the divisor
\[
\begin{tikzpicture}[scale=0.7, baseline=-3pt,label distance=0.3cm,thick,
    virtnode/.style={circle,draw,scale=0.5}, 
    nonvirt node/.style={circle,draw,fill=black,scale=0.5} ]
    \node [nonvirt node,label=below:$(0{,}0)$] (A) {};
    \node at (2,0) [nonvirt node,label=below:$(g{,}d)$] (B) {};
    \draw [-] (A) to (B);
    \node at (-.7,.5) (n1) {$i$};
    \draw [-] (A) to (n1);
    \node at (0,.7) (n2) {$\downarrow$};
    \node at (2.7,.5) (m1) {};
    \node at (2.7,-.5) (m2) {};
    \draw [-] (B) to (m1);
    \draw [-] (B) to (m2);    
    \node at (2.7,0.2) (m3) {$\vdots$};
    \end{tikzpicture} 
\to 
\begin{tikzpicture}[scale=0.7, baseline=-3pt,label distance=0.3cm,thick,
    virtnode/.style={circle,draw,scale=0.5}, 
    nonvirt node/.style={circle,draw,fill=black,scale=0.5} ]
    \node [nonvirt node,label=below:$(g{,}d)$] (A) {};
    \node at (-.7,.5) (n1) {$i$};
    \draw [-] (A) to (n1);
    \node at (.7,.5) (m1) {};
    \node at (.7,-.5) (m2) {};
    \draw [-] (A) to (m1);
    \draw [-] (A) to (m2);    
    \node at (.7,0.2) (m3) {$\vdots$};
    \end{tikzpicture}
\]
in $\frak C'$. Here, the arrows pointing to the vertices with genus and degree $0$ indicate which component of the universal curve over the corresponding boundary divisor in $\frak N_{g,n,d}$ we take. The divisors $D_0', D_1', \ldots, D_n'$ are precisely the divisorial loci in $\frak C'$ which are contracted by the map $f: \frak C' \to \frak C$.
Then 
\[
\targetmap^*\psi_i = c_1(\sigma_i^*\omega_\pi)= ({\sigma}_i')^*c_1(f^*\omega_\pi) = ({\sigma}_i')^* c_1(\omega_{\pi'}(-D'_0-\sum_{i=1}^n D'_i)) = (\sourcemap)^*\psi_i - D_i\, , 
\] 
where the sections are denoted by $\sigma_i$.
\end{proof}

For the morphism $\sourcemap$, form the  fibre diagram 
\begin{equation} \label{eqn:frakNdiagrampprime}
\begin{tikzcd}
  \Picabs'_{\Gamma_{\delta}}\arrow[d]\arrow[r,"J'_{\Gamma_{\delta}}"] & \frak N\arrow[d,"\sourcemap"]\\
\Picabs_{\Gamma_{\delta}}\ar[r,"j_{\Gamma_{\delta}}"] & \Picabs_{g,0,0}. 
\end{tikzcd}
\end{equation}
By definition, the fibre product $\Picabs'_{\Gamma_{\delta}}$ parametrizes data
\[(C_v')_{v \in \V(\Gamma)}\, , \,\, \ca L'/C'=\sqcup_v C_v'\, , \,\, f: C' \to C\, , \,\, \ca L/C\, , \,\, f^* \ca L \xrightarrow{\sim} \ca L'\]
which simplifies to
\[(C_v')_{v \in \V(\Gamma)}\, , \,\, f: \sqcup_v C_v' = C' \to C\, , \,\, \ca L/C\, .\]
On the other hand, the stack $\frak N_{\Gamma_\delta}'$   parametrizes
\[
(C_v)_{v\in \V(\Gamma)}\, , \,\, \ca L/C\, , \,\, (f_v:C'_v\to C_v)_{v\in \V(\Gamma)}\, .
\]
There is a subtle difference here. For $\Picabs'_{\Gamma_{\delta}}$, the map $f$ is allowed to contract entire components $C_v'$ to points, whereas in the second case the target $C_v$ is always still 1-dimensional.

Our next step is to show 
\begin{equation} \label{eqn:Picabsprimeiso}
\Picabs'_{\Gamma_{\delta}} \cong \sqcup_{\Gamma_{\delta}\to \tilde{\Gamma}_{\delta}} \frak N_{\Gamma_{\delta}\to \tilde{\Gamma}_{\delta}}\, .\end{equation}
More precisely,
the connected components of $\Picabs'_{\Gamma_{\delta}}$ are in bijective correspondence to partial stabilizations $$\Gamma_{\delta}\to\tilde{\Gamma}_{\delta}\, .$$
We will prove, given a vertex $v \in \V(\Gamma)$ which can be contracted (with $g(v)=0$, $n(v) \leq 2$, $\delta(v)=0$), that
the locus of points in $\Picabs'_{\Gamma_{\delta}}$ where $f: C' \to C$ contracts $C_v'$ is open and closed.

For a vertex $v$ with $n(v)=2$, the universal curve $\frak C_v' \to \Picabs'_{\Gamma_{\delta}}$ has two sections (corresponding to the half-edges at $v$), and the locus where $C_v'$ is contracted equals the locus where the sections coincide, which is closed since $\frak C_v' \to \Picabs'_{\Gamma_{\delta}}$ is separated. 

To show closedness of the locus where $C_v'$ is \emph{not} contracted, assume we have a family 
\[(C_{v,S}')_{v \in \V(\Gamma)}\, , \,\, f: \sqcup_v C_{v,S}' = C_S' \to C_S\, , \,\, \ca L/C_S\, \]
of $\Picabs'_{\Gamma_{\delta}}$ over the spectrum $S$ of a strictly henselian DVR such that the fibre $C'_{v,\eta}$ of $C'_{v,S}$ over the generic point $\eta$ of $S$ is not contracted by $f$. We want to show that then also the fibre $C'_{v,L}$ over the closed point $L$ of $S$ is not contracted. By assumption,  $C'_{v,\eta}$ maps to a union $C_{v,\eta}$ of components of the fibre $C_\eta$ of $C_S$ over $\eta$. Then $C_{v,\eta}$ specializes to a union $C_{v,L}$ of components of $C_L$. Since $f$ is proper, $f$ maps the closure $C'_{v,S}$ of $C'_{v,\eta}$ to the closure of its image $C_{v,\eta}$. Since $C_{v,L}$ is still positive-dimensional, the curve $C'_{v,L}$ is indeed not contracted. For related arguments, see the proof of \cite[Proposition 2.2]{Kresch2013Flattening-stra}. 


For a vertex $v$ with $n(v)=1$, the universal curve $\frak C_v' \to \Picabs'_{\Gamma_{\delta}}$ has a single section $\sigma_h$. On the one hand, the locus where $C_v'$ is contracted is exactly the locus where $f : \frak C' \to \frak C$ maps $\sigma_h$ to the smooth locus of $\frak C \to \Picabs'_{\Gamma_{\delta}}$, thus it is open. On the other hand, it is also the preimage under $\sigma_h$ of the exceptional locus of $\frak C' \to \frak C$, and thus closed.

We have proven that the connected components of $\Picabs'_{\Gamma_{\delta}}$ are in bijective correspondence to partial stabilizations $\Gamma_{\delta}\to\tilde{\Gamma}_{\delta}$. But a point 
\[(C_v')_{v \in \V(\Gamma)}\, , \,\, f: \sqcup_v C_v' = C' \to C\, , \,\, \ca L/C\]
on the corresponding component is equivalent to the data of any collection of curves $C_v'$ for $v' \in \V(\Gamma) \setminus \V(\tilde \Gamma)$, which are contracted by $f$, together with a point
\[(C_v')_{v \in \V(\tilde \Gamma)}\, , \,\, f: (C_v' \to C_v)\, , \,\, \ca L/C = \sqcup_{v \in V(\tilde \Gamma)} C_v\]
of $\frak N'_{\Tilde{\Gamma}_{\delta}}$. Hence, we have a isomorphism
\[\frak N_{\Gamma_{\delta}\to \tilde{\Gamma}_{\delta}} \cong \frak N'_{\Tilde{\Gamma}_{\delta}}\times \prod_{v\in \V(\Gamma)\setminus \V(\Tilde{\Gamma})} \frak M_{0,\n(v)}\, ,\]
where, in the last expression, $n(v)$ is necessarily $1$ or $2$.




For each partial stabilisation $\Gamma_{\delta}\to \tilde{\Gamma}_{\delta}$, we denote by $$J_{\Gamma_{\delta}\to\tilde{\Gamma}_{\delta}}: \frak N_{\Gamma_{\delta}\to\tilde{\Gamma}_{\delta}}\to \frak N$$ the restriction of $J'_{\Gamma_{\delta}}$ to $\frak N_{\Gamma_{\delta}\to \tilde{\Gamma}_{\delta}}$.  
\begin{lemma}\label{lem:N_invariance_P}
We have $\targetmap^* \P_{g}^\bullet = (\sourcemap)^* \P_g^\bullet
\, \in\, \prod_{c=0}^\infty {\mathsf{CH}}^c_{\mathsf{op}}(\frak N)$
for all $c \geq 0$. 
\end{lemma}
\begin{proof} 
We will use  formula \eqref{eqn:P_gformula2}
for $\P_g^\bullet$. By \cref{lem:pullback_comparison}, the terms $\exp(- \eta/2)$ have identical pullback under $\targetmap$ and $\sourcemap$. We can therefore
focus on the sum over graphs and weighting mod $r$. 

We start with a few remarks about the combinatorial factors in $\P_g^\bullet$ which will arise in the proof. Let $\Gamma_{\delta}\to\tilde{\Gamma}_{\delta}$ be a partial stabilization, then the Betti numbers agree,
$$h^1(\Gamma_\delta) = h^1(\tilde \Gamma_\delta)\, .$$
Given $r$, the map $W_{\Gamma_\delta,r} \to W_{\tilde \Gamma_\delta, r}$ of admissible weightings mod $r$ (induced by the inclusion $H(\tilde \Gamma) \to H(\Gamma)$ of half-edge sets) is a bijection. 

Moreover, if the map $\Gamma_{\delta}\to\tilde{\Gamma}_{\delta}$ only contracts components with $$(\g(v),\n(v),\delta(v))=(0,2,0)\,,$$ there is a canonical isomorphism $\Aut(\Gamma_\delta) \cong \Aut(\tilde \Gamma_\delta)$. On the other hand, in the formula for $\P_g^\bullet$, every term such that $\Gamma_\delta$ has a vertex with $$(\g(v), \n(v), \delta(v))=(0,1,0)$$ necessarily vanishes. Indeed, the half-edge $h$ at this vertex must have $w(h)=0$ such that the corresponding edge term vanishes. 

To keep the notation concise, we write $\Phi_{a}$ for the power-series 
\begin{align*} 
\Phi_{a}(x)&= \frac{1-\exp({-\frac{a}{2}}x)}{x}\\ &= \sum_{m=0}^\infty (-1)^m (\frac{a}{2})^{m+1} \frac{1}{(m+1)!}x^m = \frac{a}{2} - \frac{a^2}{8} x + \ldots 
\end{align*}
appearing in the edge terms of $\P_g^\bullet$. Moreover, given a graph $\tilde \Gamma_\delta$ with a half-edge $h$ incident to a vertex $v$, denote by $\psi_h, \psi_h'$ the classes on $\frak N'_{\tilde \Gamma_\delta}$ pulled back from $\frak N_{\g(v),\n(v),\delta(v)}$ in the diagram \eqref{eqn:frakNdiagramp}. Similarly, given a partial stabilization $\Gamma_\delta \to \tilde \Gamma_\delta$, the space $\frak N_{\Gamma_\delta \to \tilde \Gamma_\delta}$ contains $\frak N'_{\tilde \Gamma_\delta}$ as a factor, hence the notation also makes sense on $\frak N_{\Gamma_\delta \to \tilde \Gamma_\delta}$ (provided $h$ is a half-edge of $\tilde \Gamma_\delta$).

Let us first compute the pullback of the graph sum in $\P_g^{\bullet,r}$ via $\sourcemap : \frak N \to \Picabs_{g,0,0}$. Using the diagram \eqref{eqn:frakNdiagrampprime} and the decomposition \eqref{eqn:Picabsprimeiso}, we see
\begin{align} 
&(\sourcemap)^* \sum_{\Gamma_\delta, w}
\frac{r^{-h^1( \Gamma_\delta)}}{|\Aut(\Gamma_\delta)| }
\;
j_{\Gamma_\delta*}\Bigg[
\prod_{e=(h,h')\in \E(\Gamma_\delta)}
\Phi_{w(h)w(h')}(\psi_h + \psi_{h'}) \Bigg] \nonumber\\
=&\sum_{\Gamma_\delta \to \tilde \Gamma_\delta, w} \label{eqn:pprimepullb}
\frac{r^{-h^1( \Gamma_\delta)}}{|\Aut(\Gamma_\delta)| }
\;
J_{\Gamma_\delta \to \tilde \Gamma_\delta*}\Bigg[
\prod_{e=(h,h')\in \E(\Gamma_\delta)}
\Phi_{w(h)w(h')}(\psi'_h + \psi'_{h'}) \Bigg]\, .
\end{align}
In the second line, the sum is over all partial stabilizations $\Gamma_\delta \to \tilde \Gamma_\delta$.

Second, we compute the pullback of the graph sum in $\P_g^{\bullet,r}$ via $\targetmap : \frak N \to \Picabs_{g,0,0}$
\begin{align} 
&\targetmap^* \sum_{\tilde \Gamma_\delta, w}
\frac{r^{-h^1(\tilde  \Gamma_\delta)}}{|\Aut(\tilde \Gamma_\delta)| }
\;
j_{\tilde \Gamma_\delta*}\Bigg[
\prod_{e=(h,h')\in \E(\tilde \Gamma_\delta)}
\Phi_{w(h)w(h')}(\psi_h + \psi_{h'}) \Bigg]\nonumber \\
=&\sum_{\tilde \Gamma_\delta, w}\label{eqn:ppullb}
\frac{r^{-h^1(\tilde \Gamma_\delta)}}{|\Aut(\tilde \Gamma_\delta)| }
\;
\widehat J_{\tilde \Gamma_\delta*}\Bigg[
\prod_{e=(h,h')\in \E(\tilde \Gamma_\delta)}
\Phi_{w(h)w(h')}(\psi_h + \psi_{h'}) \Bigg]\, .
\end{align}
We use here the right fibre diagram in \eqref{eqn:frakNdiagramp} together with the fact that $G$ is proper, representable, and birational.  So by Proposition \ref{prop:invarianceproperbirat},
the pushforward of fundamental classes under $J_{\tilde \Gamma_\delta}$ and $\widehat J_{\tilde \Gamma_\delta}$ agree.

To compare to the formula for the pullback under $\sourcemap$, we use 
$$\psi_h + \psi_{h'} = \psi_h' - D_h + \psi_{h'}'- D_{h'}$$ on $\frak N'_{\tilde \Gamma_\delta}$ by \cref{lem:pullback_comparison}.
The next step of the proof is to use the self-intersection formula for $D_h, D_{h'}$ (similar to the formula described in \cite[Appendix A]{GraberPandharipande}) to expand the edge term
\[\Phi_{w\,w'}(\psi_h' - D_h + \psi_{h'}'- D_{h'}) = \sum_{m=0}^\infty (-1)^m (\frac{w\,w'}{2})^{m+1} \frac{1}{(m+1)!}(\psi_h' - D_h + \psi_{h'}'- D_{h'})^m.\] For example, $(D_h)^2$ is equal to
\begin{equation*}
-
    \begin{tikzpicture}[scale=1.2, baseline=-3pt,label distance=0.3cm,thick,
    virtnode/.style={circle,draw,scale=0.5}, 
    nonvirt node/.style={circle,draw,fill=black,scale=0.5} ]
    \node at (0,0) [nonvirt node,label=below:$(g_{v}{,}\delta(v))$] (A) {};
    \node at (1.5,0) [nonvirt node,label=below:$(0{,}0)$] (C) {};    
    \node at (3,0) [nonvirt node,label=below:$(g_{v'}{,}\delta(v'))$] (B) {};
    \node at (.2,-.2) {$h_1$};
    \node at (1.3,-.2) {$h'_1$}; 
    \node at (1.7,-.2) {$h$};
    \node at (2.8,-.2) {$h'$};
    \draw [-] (A) to (C);
    \draw [-] (B) to (C);
    \node at (.7,.4) {$(\psi_{h_1}+\psi_{h'_1})$};
    \node at (3.7,.5) (m1) {};
    \node at (3.7,-.5) (m2) {};
    \draw [-] (B) to (m1);
    \draw [-] (B) to (m2);    
    \node at (3.7,0.1) (m3) {$\vdots$};
    \node at (-0.7,.5) (n1) {};
    \node at (-0.7,-.5) (n2) {};
    \draw [-] (A) to (n1);
    \draw [-] (A) to (n2);    
    \node at (-.7,0.1) (n3) {$\vdots$};    
    \end{tikzpicture}
+ \,2
    \begin{tikzpicture}[scale=1.2, baseline=-3pt,label distance=0.3cm,thick,
    virtnode/.style={circle,draw,scale=0.5}, 
    nonvirt node/.style={circle,draw,fill=black,scale=0.5} ]
    \node at (0,0) [nonvirt node,label=below:$(g_{v}{,}\delta(v))$] (A) {};
    \node at (1.5,0) [nonvirt node,label=below:$(0{,}0)$] (C) {};   
    \node at (3,0) [nonvirt node,label=below:$(0{,}0)$] (D) {}; 
    \node at (4.5,0) [nonvirt node,label=below:$(g_{v'}{,}\delta(v'))$] (B) {};
    \node at (3.2,-.2) {$h$};
    \node at (4.3,-.2) {$h'$};
    \draw [-] (A) to (C);
    \draw [-] (C) to (D);
    \draw [-] (D) to (B);
    \node at (5.2,.5) (m1) {};
    \node at (5.2,-.5) (m2) {};
    \draw [-] (B) to (m1);
    \draw [-] (B) to (m2);    
    \node at (5.2,0.1) (m3) {$\vdots$};
    \node at (-0.7,.5) (n1) {};
    \node at (-0.7,-.5) (n2) {};
    \draw [-] (A) to (n1);
    \draw [-] (A) to (n2);    
    \node at (-.7,0.1) (n3) {$\vdots$};    
    \end{tikzpicture}
\end{equation*}
and similarly for $(D_{h'})^2$.

The result will be a linear combination of terms
\begin{equation}
\begin{tikzpicture}[scale=1.2, baseline=-3pt,label distance=0.3cm,thick,
    virtnode/.style={circle,draw,scale=0.5}, 
    nonvirt node/.style={circle,draw,fill=black,scale=0.5} ]
    \node at (-6,0) [nonvirt node,label=below:$(g_{v}{,}\delta(v))$] (F) {};
    \node at (-4,0) [nonvirt node,label=below:$(0{,}0)$] (E) {};
    \node at (-3,0) [nonvirt node,label=below:$(0{,}0)$] (D) {};
    \node at (-1,0) [nonvirt node,label=below:$(0{,}0)$] (C) {};
    \node [nonvirt node,label=below:$(0{,}0)$] (A) {};
    \node at (2,0) [nonvirt node,label=below:$(g_{v'}{,}\delta(v'))$] (B) {};
    \draw [-] (A) to (B);
    \draw [-, dotted] (C) to (A);
    \draw [-] (D) to (C);
    \draw [-, dotted] (E) to (D);
    \draw [-] (F) to (E);
    \node at (1,.4) {$(\psi_{h_L}+\psi_{h'_L})^{e_L}$};
    \node at (-5,.4) {$(\psi_{h_1}+\psi_{h'_1})^{e_1}$};
    \node at (-2.7,.4) {$\psi_h^a$};
    \node at (-1.3,.4) {$\psi_{h'}^b$};
    \node at (2.7,.5) (m1) {};
    \node at (2.7,-.5) (m2) {};
    \draw [-] (B) to (m1);
    \draw [-] (B) to (m2);    
    \node at (2.7,0.1) (m3) {$\vdots$};
    \node at (-6.7,.5) (n1) {};
    \node at (-6.7,-.5) (n2) {};
    \draw [-] (F) to (n1);
    \draw [-] (F) to (n2);    
    \node at (-6.7,0.1) (n3) {$\vdots$};    
    \end{tikzpicture}
\end{equation}
where the edge $(h,h')$ is at position $\ell$ in the above chain ($1 \leq \ell \leq L$). The total degree of this term (before the pushforward by $\widehat J_{\tilde \Gamma_\delta}$) is
\[m = \sum_{j \neq \ell} (e_j+1) + a + b\, .\]
The total coefficient of this particular term in $\Phi_{w\,w'}(\psi_h' - D_h + \psi_{h'}'- D_{h'})$ is then
\[
\underbrace{(-1)^m (\frac{w\,w'}{2})^{m+1} \frac{1}{(m+1)!}}_{\text{coeff in $\Phi_{w\,w'}$}}\cdot \binom{m}{e_1+1, \ldots, a,b,\ldots, e_L+1} \underbrace{(-1)^{L-1}}_{\substack{\text{excess intersection}\\\text{of $-D_h, -D_{h'}$}}} 
\]
where the multinomial coefficient comes from the expansion of $$(\psi_h' - D_h + \psi_{h'}'- D_{h'})^m\, .$$
Writing $e_\ell=a+b$, we can simplify to obtain
\[
\frac{1}{m+1}
\left(\prod_{j=1}^L (-1)^{e_j} (\frac{w\,w'}{2})^{e_j+1} \frac{1}{(e_j+1)!} \right) \binom{e_\ell}{a} (e_\ell+1)\, .
\]
Summing over all choices $a+b=e_\ell$ for fixed $e_\ell$, the coefficient of the term
\begin{equation}
\begin{tikzpicture}[scale=1.2, baseline=-3pt,label distance=0.3cm,thick,
    virtnode/.style={circle,draw,scale=0.5}, 
    nonvirt node/.style={circle,draw,fill=black,scale=0.5} ]
    \node at (-6,0) [nonvirt node,label=below:$(g_{v}{,}\delta(v))$] (F) {};
    \node at (-4,0) [nonvirt node,label=below:$(0{,}0)$] (E) {};
    \node at (-3,0) [nonvirt node,label=below:$(0{,}0)$] (D) {};
    \node at (-1,0) [nonvirt node,label=below:$(0{,}0)$] (C) {};
    \node [nonvirt node,label=below:$(0{,}0)$] (A) {};
    \node at (2,0) [nonvirt node,label=below:$(g_{v'}{,}\delta(v'))$] (B) {};
    \draw [-] (A) to (B);
    \draw [-, dotted] (C) to (A);
    \draw [-] (D) to (C);
    \draw [-, dotted] (E) to (D);
    \draw [-] (F) to (E);
    \node at (1,.4) {$(\psi_{h_L}+\psi_{h'_L})^{e_L}$};
    \node at (-5,.4) {$(\psi_{h_1}+\psi_{h'_1})^{e_1}$};
    \node at (-2,.4) {$(\psi_{h}+\psi_{h'})^{e_\ell}$};
    \node at (2.7,.5) (m1) {};
    \node at (2.7,-.5) (m2) {};
    \draw [-] (B) to (m1);
    \draw [-] (B) to (m2);    
    \node at (2.7,0.2) (m3) {$\vdots$};
    \node at (-6.7,.5) (n1) {};
    \node at (-6.7,-.5) (n2) {};
    \draw [-] (F) to (n1);
    \draw [-] (F) to (n2);    
    \node at (-6.7,0.2) (n3) {$\vdots$};    
    \end{tikzpicture}
\end{equation}
is exactly equal to
\[
\frac{e_\ell +1}{m+1}
\left(\prod_{j=1}^L (-1)^{e_j} (\frac{w\,w'}{2})^{e_j+1} \frac{1}{(e_j+1)!} \right).
\]
Pushing forward by $\widehat J_{\tilde \Gamma_\delta}$, we forget where in the chain above the edge $(h,h')$ has been. Summing over the $L$ possible positions and using $m+1 = \sum_{\ell=1}^L (e_\ell +1)$, we obtain the coefficient
\[\prod_{j=1}^L \underbrace{(-1)^{e_j} (\frac{w\,w'}{2})^{e_j+1} \frac{1}{(e_j+1)!}}_{\text{coefficient of $x^{e_j}$ in $\Phi_{w\,w'}(x)$}}\, .\]
From the above discussion, we see that \eqref{eqn:ppullb} equals
\[\sum_{\Gamma_\delta \to \tilde \Gamma_\delta, w}
\frac{r^{-h^1( \tilde \Gamma_\delta)}}{|\Aut(\tilde \Gamma_\delta)| }
\;
J_{\Gamma_\delta \to \tilde \Gamma_\delta*}\Bigg[
\prod_{e=(h,h')\in \E(\Gamma_\delta)}
\Phi_{w(h)w'(h)}(\psi'_h + \psi'_{h'}) \Bigg]\, .\]
The sum goes over stabilizations $\Gamma_\delta \to \tilde \Gamma_\delta$ contracting chains of curves with $(g,n,d)=(0,2,0)$.
By the previous remarks concerning the combinatorial factors, we have
\[h^1( \tilde \Gamma_\delta) = h^1(\Gamma_\delta)\, , \ \ \ |\Aut(\tilde \Gamma_\delta)| = |\Aut(\Gamma_\delta)|\, . \]
The sum does not change if we allow arbitrary stabilizations $\Gamma_\delta \to \tilde \Gamma_\delta$, since for $\Gamma_\delta$ having a vertex with $(g,n,d)=(0,1,0)$, the summand automatically vanishes. Thus the sum above equals the term computed in \eqref{eqn:pprimepullb}. 
\end{proof}

\section{Applications}\label{sec:applications}
\subsection{Proofs of Theorem \ref{FPA1} and Conjecture A} \label{Sect:ProofConjA}
We start by recalling notions presented in Section \ref{Sect:Twistdiffintro}, but now in the more general setting of $k$-differentials.
Let $A=(a_1,\ldots,a_n)$ be a vector of zero and pole multiplicities
satisfying
$$\sum_{i=1}^n a_i= k(2g-2)\, .$$
Let $\HH^k_g(A) \subset \mathcal{M}_{g,n}$ be the closed (generally non-proper) locus 
of pointed curves $(C,p_1,\ldots,p_n)$ satisfying the condition
$$\mathcal{O}_C\Big(\sum_{i=1}^n a_i p_i\Big) \simeq \omega_C^{\otimes k}\, .$$
In other words, $\HH^k_g(A)$ is the locus of (possibly) meromorphic $k$-differentials
with zero and pole multiplicities prescribed by $A$.
In \cite{Farkas2016The-moduli-spac}, a compact moduli space of twisted $k$-canonical divisors
$$\widetilde{\HH}^k_g(A) \subset \overline{\mathcal{M}}_{g,n}$$
is constructed extending $\HH^k_g(A) = \widetilde{\HH}^k_g(A) \cap \mathcal{M}_{g,n} $ to the boundary of $\overline{\mathcal{M}}_{g,n}$.

For $k \geq 1$ and $A$ not of the form $A = k \cdot A'$ with a vector $A'$ of nonnegative integers, the locus
$\widetilde{\HH}^k_g(A)$ is of pure codimension $g$ in
$\overline{\mathcal{M}}_{g,n}$ by \cite[Theorem 3]{Farkas2016The-moduli-spac} (for $k=1$) and \cite[Theorem 1.1]{Schmitt2016Dimension-theor} (for $k>1$). 
A weighted fundamental cycle of $\widetilde{\HH}^k_g(A)$, 
\begin{equation}\label{wwww4k}
\mathsf{H}^k_{g,A}\in \mathsf{CH}_{2g-3+n}(\oM_{g,n})\ ,
\end{equation}
is constructed in
\cite[Appendix A]{Farkas2016The-moduli-spac} and \cite[Section 3.1]{Schmitt2016Dimension-theor} with explicit nontrivial 
weights on the irreducible components.
The closure
$$\overline{\HH}^k_{g,A}\subset \oM_{g,n}$$
contributes to the weighted fundamental class $\mathsf{H}^k_{g,A}$ with
multiplicity 1, but there are additional boundary contributions, as described in the references above.

The weighted fundamental class 
$\mathsf{H}^k_{g,A}$
was conjectured in \cite{Farkas2016The-moduli-spac, Schmitt2016Dimension-theor}
to equal  the class given by Pixton's formula for the double ramification cycle. To state the conjecture, consider the shifted\footnote{The shift is needed since Pixton's original formula worked with powers of the log-canonical line bundle $\omega_C^{\text{log}} = \omega_C(\sum_{i=1}^n p_i)$ instead of $\omega_C$.} vector $\widetilde A = (a_1 + k, \ldots, a_n +k)$.

\begin{conjectureA}
For $k \geq 1$ and $A$ not of the form $A = k \cdot A'$ with a vector $A'$ of nonnegative integers, we have an equality
\[\mathsf{H}^k_{g,A} = 2^{-g} P_g^{g,k}(\widetilde A)\, ,\]
where $P_g^{g,k}(\widetilde A)$ is Pixton's cycle class defined in \cite[Section 1.1]{Janda2016Double-ramifica}.
\end{conjectureA}

\noindent By combining Theorem \ref{Thm:main}  with previous results of \cite{Holmes2019Infinitesimal-s}, we can now prove the conjecture.

\begin{theorem}\label{FPA}
Conjecture A is true.
\end{theorem}
\begin{proof}
By Theorem 1.1 of \cite{Holmes2019Infinitesimal-s}, the weighted fundamental class $\mathsf{H}^k_{g,A}$ is equal to the double ramification cycle $\mathsf{DR}_{g,A, \omega^k}$ constructed in \cite{Holmes2017Extending-the-d}. By Theorem \ref{cccc},
$\mathsf{DR}_{g,A, \omega^k}$
is in turn given by the action of $\mathsf{DR}^{\mathsf{op}}_{g,A}$ on the fundamental class
of $\overline{\mathcal{M}}_{g,n}$ via the morphism $\phi_{\omega_\pi^k}: \overline{\mathcal{M}}_{g,n} \to \Picabs_{g,n,k(2g-2)}$ associated to the family
\begin{equation*}
\pi: \mathcal{C}_{g,n} \rightarrow \overline{\cM}_{g,n}\, , \ \ \ \ \omega_\pi^{k} \rightarrow \mathcal{C}_{g,n}\, .
\end{equation*}
By Theorem \ref{Thm:main}, the class $\mathsf{DR}^{\mathsf{op}}_{g,A}$ is computed by the tautological class
$$\P_{g,A,d}^g \in {\mathsf{CH}}^g_{\mathsf{op}}(\Picabs_{g,n,k(2g-2)})\, . $$ By Proposition \ref{Prop:PixtonPullbackCompatibility}, the action of $\P_{g,A,d}^g$ on $[\overline{\mathcal{M}}_{g,n}]$ is indeed given by Pixton's original formula $ 2^{-g} P_g^{g,k}(\widetilde A)$, finishing the proof. \end{proof}

The steps of the proof of Theorem \ref{FPA} are summarized as follows:
\begin{align*}
    \mathsf{H}^k_{g,A}&=\mathsf{DR}_{g,A, \omega^k} &&\cite[\text{Theorem }1.1]{Holmes2019Infinitesimal-s}\\
    &=\mathsf{DR}^{\mathsf{op}}_{g,A}(\phi_{\omega_\pi^k})([\overline{\mathcal{M}}_{g,n}]) &&\text{Theorem }\ref{cccc}\\
    &=\P_{g,A,d}^g(\phi_{\omega_\pi^k})([\overline{\mathcal{M}}_{g,n}]) &&\text{Theorem }\ref{Thm:main}\\
    &=2^{-g} P_g^{g,k}(\tilde A) &&\text{Proposition }\ref{Prop:PixtonPullbackCompatibility}\, . 
\end{align*}
The result provides a completely geometric representative
of Pixton's cycle in terms of twisted $k$-differentials. 
Theorem \ref{FPA1} of Section \ref{Sect:Twistdiffintro} is the $k=1$ case of 
Theorem \ref{FPA}.



\subsection{Closures}\label{clozzz}
Let $A=(a_1,\ldots, a_n)$ be a vector of integers satisfying
$$\sum_{i=1}^n a_i =k(2g-2)\, .$$
A careful investigation of the closure
$$\HH^k_g(A) \subset \overline{\HH}^k_g(A) \subset \oM_{g,n}$$
is carried out in \cite{Bainbridge2018Compactificatio, Bainbridge2019Strata-of-k-dif}.
By a simple method presented in \cite[Appendix A]{Farkas2016The-moduli-spac} and \cite[Section 3.4]{Schmitt2016Dimension-theor},
Theorem \ref{FPA} easily determines the cycle classes of the
closures
$$ [\overline{\HH}^k_g(A)] \in \mathsf{CH}_{*} (\oM_{g,n})\, .$$
for the cases 
\begin{itemize}
    \item $k=1$ and all $a_i$ nonnegative, when $\overline{\HH}^k_g(A)$ has pure codimension $g-1$ and
    \item $k\geq 1$ and $A$ not of the form $A = k \cdot A'$ with a vector $A'$ of nonnegative integers, when $\overline{\HH}^k_g(A)$ has pure codimension $g$. 
\end{itemize}
In particular, from the recursive formula for $[\overline{\HH}^k_g(A)]$ and the fact that Pixton's cycle on $\oM_{g,n}$ is tautological, the following is immediate.
\begin{corollary}\label{kk333}
The cycles $[\overline{\HH}^k_g(A)]$ are tautological classes in $\mathsf{CH}_*(\oM_{g,n})$.
\end{corollary}
In the case $k=1$, Corollary \ref{kk333} was known by work of Sauvaget  \cite{Sauvaget2017Cohomology-clas}, who gave a different approach to $[\overline{\HH}^1_g(A)]$ in terms of tautological classes.
The recursive formulas for $[\overline{\HH}^k_g(A)]$ from
Corollary \ref{kk333} have been implemented\footnote{In the ongoing project \cite{CMZ20}, the authors study formulas for Euler characteristics of strata of differentials in terms of intersection numbers on the compactification of these strata constructed in \cite{BCGGM3}. The
implementation of $[\overline{\HH}^1_g(A)]$ has played
a role in corroborating their formulas.}
in the software \cite{admcycles} for computations in the tautological ring of $\Mbar_{g,n}$.

Another application of Conjecture A is presented in the recent paper \cite{sauvaget_volumes} by Sauvaget. The paper studies moduli spaces of flat surfaces of genus $g$ with conical singularities at marked points $p_1, ..., p_n$. The singularities have fixed cone angles $2 \pi \alpha_i$, for $\alpha_1, \ldots, \alpha_n \in \mathbb{R}$, summing to $2g-2+n$. 
If all $\alpha_i$ are rational,  
the spaces of flat surfaces naturally contain $\mathcal{H}_g^k(k A )$,
for  $$A=(\alpha_i-1)_{i=1}^n\, ,$$
as closed subsets (for $k$ sufficiently divisibly). These subsets equidistribute (with respect to natural measures) as $k \to \infty$. Using the
equidistribution, 
Sauvaget is able to apply the recursive expression for $\overline{\HH}_g^k(k A)$ from Conjecture A to derive an explicit recursion for the volumes of the moduli spaces of flat surfaces.

\subsection{\texorpdfstring{$k$}{k}-twisted DR cycles with targets} \label{Sect:ktwistedDR}
We define  here  {\em $k$-twisted double ramification cycles with targets} via the class $\DRop_{g,A}$.

Let $X$ be a nonsingular projective variety with line bundle $\ca L$ and
an effective curve class
$\beta \in H_2(X, \mathbb{Z})$. Let
$$d_\beta= \int_\beta c_1(\ca L)\, .$$
Let $k\in \mathbb{Z}$ and $A=(a_1,\ldots, a_n)\in \mathbb{Z}^n$ satisfy
\begin{equation*}
d_\beta + k(2g-2+n) = \sum_{i=1}^{n}a_i\, .
\end{equation*}
Consider the morphism
\[
\phi\colon \MbarX \to \Picabs_{g,n,d_\beta+k(2g-2+n)}
\]
defined by the universal data
\begin{equation}\label{fftt666774}
\pi: \mathcal{C}_{g,n,\beta} \rightarrow \overline{\cM}_{g,n}(X,\beta)\, , \ \ \ \ 
f^*\ca L\otimes \omega^{\otimes k}_{log}
\rightarrow \mathcal{C}_{g,n,\beta}\, ,
\end{equation}
where
$f: \mathcal{C}_{g,n,\beta} \to X$
is the universal map.

\begin{definition}
The \textit{$k$-twisted $X$-valued double ramification cycle} is defined by
\[
\DR_{g,n,\beta}^{k}(X,\ca L)=\mathsf{DR}^{\mathsf{op}}_{g,A} (\phi) ( [\MbarX]^\textup{vir}) \in \Chow_{\vdim(g,n,\beta)-g}(\MbarX)\,.
\]
\end{definition}
In the notation of \cite[Section 0.4]{Janda2018Double-ramifica}, let $\P^{c,k,r}_{g,A,\beta}(X,\ca L)$ be the codimension $c$ part of the following expression
\begin{multline*}
\hspace{-10pt}\sum_{
\substack{\Gamma\in \G_{g,n,\beta}(X) \\
w\in \mathsf{W}_{\Gamma,r,k}(X)}
}
\frac{r^{-h^1(\Gamma)}}{|\Aut(\Gamma)| }
\;
j_{\Gamma_\delta*}\Bigg[
\prod_{i=1}^n \exp\left(\frac12 a_i^2 \psi_i + a_i \xi_i \right)\\ \hspace{+20pt}
\prod_{v \in \V(\Gamma)} \exp\left(-\frac12 \eta(v) - k \eta_{1,1}(v) - \frac{k^2}{2} \eta_{0,2}(v)    \right)
\\ \hspace{+10pt}
\prod_{e=(h,h')\in \E(\Gamma)}
\frac{1-\exp\left(-\frac{w(h)w(h')}2(\psi_h+\psi_{h'})\right)}{\psi_h + \psi_{h'}} \Bigg]\,.
\end{multline*} 
The definition of the admissible $k$-weightings $w\in \mathsf{W}_{\Gamma,r,k}(X)$ is similar to that in \cref{Ssec:weightings} but with the condition (iii) replaced by
\[
k(2g(v)-2+n(v)) + \int_{\beta(v)} c_1(\ca L) = \sum_{v(h)=v}w(h)\,\text{ for }v \in \V(\Gamma)\, .
\]
As in the case $k=0$ discussed in \cite[Proposition 1]{Janda2018Double-ramifica}, the class $\P^{c,k,r}_{g,A,\beta}(X,\ca L)$ is polynomial in $r$ for all sufficiently large $r$. 
Denote by $\P^{c,k}_{g,A,\beta}(X,\ca L)$ the value at $r=0$ of this polynomial. 

By \cref{Thm:main} and a slight generalization of the procedure for pulling back Chow cohomology classes from $\Picabs_{g,n,d_\beta+k(2g-2+n)}$ to $\MbarX$ described in \cite[Section 1.5]{Janda2018Double-ramifica}, we have 
\begin{eqnarray*}
    \DR_{g,n,\beta}^{k}(X,\ca L)&= &\mathsf{DR}^{\mathsf{op}}_{g,A}(\phi)( [\MbarX]^\textup{vir})\\
    &=&\P_{g,A,d_\beta+k(2g-2+n)}^g (\phi)( [\MbarX]^\textup{vir})\\
    &=&\P^{g,k}_{g,A,\beta}(X,\ca L)\, .
\end{eqnarray*}

\subsection{Proof of Theorem \ref{Thm:mainvan}}
For all $c>g$, we will prove 
\begin{equation} \label{eqn:PgAdc_vanishing}
    \P_{g,A,d}^c\ = 0 \
\in {\mathsf{CH}}^c_{\mathsf{op}}(\Picabs_{g,n,d})\, .
\end{equation}
The path is parallel to the proof of Theorem \ref{Thm:main}.

By definition, the claim is
equivalent to showing that the map 
\begin{equation}\label{ff55ff3}
 \P_{g,A,d}^c(\phi) : \Chow_*(B) \to \Chow_{*-c}(B)
 \end{equation}
is zero for every
morphism $\phi:B \to \Picabs_{g,n,d}$ from an (irreducible) finite type scheme $B$ 
corresponding to the data
$$C \to B\, , \ \ \  \ca L \to C\,.$$
Retracing the steps of \cref{sec:proof_of_main_theorem} (and using the invariance \cref{lem:N_invariance_P}  for the codimension $c$ part $\P_{g,A,d}^c$ of Pixton's formula), we can reduce to the situation where $\ca L$ on $C$ is relatively sufficiently
positive with respect to $C \to B$. As in \cref{sec:positive_line_bundle},
we can then find $$\psi: U_l \to B$$
such that $\psi^*$ is injective on Chow groups and such that the composition 
$$U_l \to B \to \Picabs_{g,n,d}$$ 
factors through $\overline{\mathcal{M}}_{g,n}(\PP^l, d)'$. By Theorem 3.2
of \cite{Bae}, 
we have the vanishing $$\P_{g,A,d}^c (\PP^l,\ca O(1)) = 0
\, \in \, \mathsf{CH}_{\textup{vdim}(g,n,d)-c}(\overline{\mathcal{M}}_{g,n}(\PP^l, d))\, .$$
The same combination of \cref{lem:op_eq_ch_for_smooth} and the injectivity of $\psi^*$ then shows the desired vanishing of the map \eqref{ff55ff3}. \qed



\subsection{Connections to past and future results}\label{pasfut}
The relations of Theorem \ref{Thm:mainvan}
 generalize several previous results.
For $g=0$ and $c=1$, the vanishing \eqref{eqn:PgAdc_vanishing}
was observed in \cite[Proposition 1.2]{deJongStarr}. 
In fact, in genus $0$, there are many connections to past equations, see \cite[Section 4]{Bae}
for a full discussion with many examples including classical
equations and 
the relations of \cite{LeePand}.


Randal-Williams \cite{RandalWilliams2012}  proves
a vanishing result in  cohomology with integral coefficients on the locus $\Picabs_{g,0,d}^\textup{sm}$ of smooth curves
for every $d \in \mathbb{Z}$.
We can recover a version of Randal-William's
vanishing in operational Chow  with $\mathbb{Q}$-coefficients which extends to all of $\Picabs_{g,0,d}$.  By \cref{Lem:Pixformulafactorization} and \cref{Lem:vertterm},  Pixton formula's on the locus $\Picabs_{g,0,0}^\textup{sm}$ takes the simple form
\[\P_{g,\emptyset,0}^c = \frac{1}{c!} (\P_{g,\emptyset,0}^1)^c\, ,\ \ \  \P_{g,\emptyset,0}^1 = - \frac{1}{2} \pi_* (c_1(\ca L)^2)\]
for the universal curve and the universal line bundle
$$\pi: \frak C \to \Picabs_{g,0,0}\, , \ \ \ \ca L \to \frak C\, .$$
We claim, up to scaling, the relation 
$$\Omega^{g+1}=0$$ of
\cite[Theorem A]{RandalWilliams2012} is exactly the restriction of the pullback of the relation $$(\P_{g,\emptyset,0}^1)^{g+1} = (g+1)! \P_{g,\emptyset,0}^{g+1} = 0 
\in {\mathsf{CH}}^{g+1}_{\mathsf{op}}(\Picabs_{g,0,0})$$
under the morphism
\[\Picabs_{g,0,d} \to \Picabs_{g,0,0}\, ,\ \ \ (C, \ca L) \mapsto (C, \ca L^{\otimes 2g-2} \otimes \omega_C^{\otimes (-d)})\, .\]
Indeed, over the locus of smooth curves, the pullback of $\P_{g,\emptyset,0}^1$ is given by
\begin{multline*} 
- \frac{1}{2} \pi_* (c_1(\ca L^{\otimes 2g-2} \otimes \omega_C^{\otimes (-d)})^2)=\\ -\frac{1}{2} \left((2g-2)^2 \pi_*(c_1(\ca L)^2) - 2d(2g-2)\pi_*(c_1(\ca L ) c_1(\omega_\pi)) + d^2 \pi_*(c_1(\omega_\pi)^2) \right)\, ,\end{multline*}
which matches the definition of $\Omega$ given in \cite[Theorem A]{RandalWilliams2012}
up to scalars.

In Gromov-Witten theory,  pulling back \eqref{eqn:PgAdc_vanishing} under the morphisms $$\MbarX \to \Picabs_{g,n,d}$$ described in  \cref{Sect:ktwistedDR} and capping with the virtual class $[\MbarX]^\textup{vir}$ simply recovers the known vanishing 
\begin{equation} \label{eqn:PXbetavanishing}
    \P_{g,A,\beta}^{c,k}(X,\ca L) = 0 \in \Chow_{\vdim(g,n,\beta)-c}(\MbarX)
\end{equation}
for $c>g$ proven in \cite{Bae}. However, there are new applications for \emph{reduced} Gromov-Witten theory. Indeed, for a target $X$ having a nondegenerate holomorphic $2$-form, the virtual class of $\MbarX$ vanishes when $\beta \neq 0$.
To define invariants for such targets, the reduced class
\[[\MbarX]^\textup{red} \in \Chow_{*}(\MbarX)\]
is used instead, see \cite{bryanleung, maulikpandharipande}. By 
pulling back \eqref{eqn:PgAdc_vanishing} and capping with $[\MbarX]^\textup{red}$, we obtain new relations among reduced Gromov-Witten invariants. An application to the Gromov-Witten theory of K3 surfaces will appear in \cite{BaeBuelles} related to 
conjectures of \cite{ObPand}.



 

\bibliographystyle{abbrv} 

\def\cprime{$'$}

\noindent {Department of Mathematics, ETH Z\"urich} \\
\noindent {younghan.bae@math.ethz.ch}\\

\noindent {Mathematisch Instituut, Universiteit Leiden} \\
\noindent {holmesdst@math.leidenuniv.nl}\\

\noindent {Department of Mathematics, ETH Z\"urich} \\
\noindent {rahul.pandharipande@math.ethz.ch}\\

\noindent {Mathematical Institute, University of Bonn} \\
\noindent {schmitt@math.uni-bonn.de}\\

\noindent {Mathematisch Instituut, Universiteit Leiden} \\
\noindent {r.m.schwarz@math.leidenuniv.nl}\\
\end{document}